\RequirePackage{etex}
\documentclass[11pt]{article} 
\usepackage[utf8]{inputenc}
\usepackage[T1]{fontenc}

\newif\ifnips
\nipsfalse  
\usepackage{booktabs}
\usepackage{algorithm,algpseudocode}
\usepackage[numbers]{natbib}

\ifnips
    \usepackage{neurips_2025}
  \usepackage[colorlinks, pagebackref, linkcolor=blue, citecolor=blue]{hyperref}
    \usepackage{url}
    \usepackage[showlabel,refcheck]{reduced}
  
\else
   \usepackage{arxiv}
\fi
\crefname{algocf}{alg.}{algs.}
\Crefname{algocf}{Algorithm}{Algorithms}
\crefname{equation}{\text{Eq.}}{\text{Eqs.}}

\newcommand{\eps}{\varepsilon}      

\newcommand{\E}{\mathbb{E}}         
\newcommand{\He}{\mathrm{He}}
\AtBeginDocument{%
  \renewcommand{\Pr}{\mathbb{P}}%
}

\renewcommand{\Re}{\mathrm{Re}}  
\renewcommand{\Im}{\mathrm{Im}}  

\newcommand{\N}{\mathbb{N}}  
\newcommand{\R}{\mathbb{R}}  
\newcommand{\C}{\mathbb{C}}  

\newcommand{\langlexi}{\langle\xi\rangle}
\newcommand{\indlim}{\varinjlim}

\title{Constructive Approximation under Carleman's Condition, \\ with Applications to Smoothed Analysis}

\author{
Frederic Koehler\footnote{Department of Statistics and Data Science Institute, University of Chicago. \email{\href{mailto:fkoehler@uchicago.edu}{fkoehler}}{uchicago.edu}.} 
\and Beining Wu\footnote{Department of Statistics, University of Chicago. \email{\href{mailto:beiningw@uchicago.edu}{beiningw}}{uchicago.edu}.}
}
\begin{document}
\maketitle
\begin{abstract}
A classical result of Carleman, based on the theory of quasianalytic functions, shows that polynomials are dense in $L^2(\mu)$ for any $\mu$ such that the moments $\int x^k d\mu$ do not grow too rapidly as $k \to \infty$. 
In this work, we develop a fairly tight quantitative analogue of the underlying Denjoy-Carleman theorem via complex analysis, and show that this allows for nonasymptotic control of the rate of approximation by polynomials for any smooth function with polynomial growth at infinity. 
In many cases, this allows us to establish $L^2$ approximation-theoretic results for functions over general classes of distributions (e.g., multivariate sub-Gaussian or sub-exponential distributions) which were previously known only in special cases.
As one application, we show that the Paley--Wiener class of functions bandlimited to $[-\Omega,\Omega]$ admits superexponential rates of approximation over all strictly sub-exponential distributions, which leads to a new characterization of the class.
As another application, we solve an open problem recently posed by Chandrasekaran, Klivans, Kontonis, Meka and Stavropoulos on the smoothed analysis of learning, and also obtain quantitative improvements to their main results and applications. 
\end{abstract}
\newpage
\tableofcontents
\newpage

\section{Introduction}
It was proven by Weierstrass in his 1885 paper that polynomials are dense in the space of continuous
functions on an interval \cite{akhiezer2020classical}. Later, the theory of constructive approximation by polynomials was developed
, which connected the quantitative \emph{rate} of approximation by polynomials with the smoothness properties of the function
(see, e.g., Jackson's theorem for polynomial approximation \cite{devore1993constructive}). A particularly fascinating result of Bernstein shows
that a function can be approximated at an exponentially fast rate on the interval $[-1,1]$ if and only if it admits
an analytic extension onto an open neighborhood of the interval in the complex plane \cite{devore1993constructive}. 

A simple consequence of Weierstrass's theorem in probability and statistics is that any probability measure over
the interval $[-1,1]$ (or any compact set) is uniquely determined by its moments $\mathbb E[X^k]$ for $k \in \mathbb Z_{\ge 0}$.
On the other hand, this is certainly not true for general distributions on $\mathbb R$ --- the location parameter (center) of the Cauchy distribution
cannot even be determined by its moments, because the Cauchy distribution has no moments beyond $k = 0$. There are also many distributions
for which all of their moments exist, which are nevertheless not determined by them.

In the classical study of the moment problem, \emph{Carleman's condition} is one of the most famous and well-known results. In the context of probability theory, 
Carleman's result tells us that if the moments of a probability measure $\mu$ satisfy
\[ \sum_{k = 1}^{\infty} \left[\frac{1}{\mathbb E_{\mu}[X^{2k}]}\right]^{1/2k} = \infty, \]
then $\mu$ is indeed uniquely determined by moments, and relatedly, polynomials are dense in natural function spaces
like $L_1(\mu)$ and $L_2(\mu)$ \cite{akhiezer2020classical}. 
Approximating functions by polynomials is useful in a variety of research areas, including computational and statistical learning theory, 
where approximation results can be used to reduce the problem of learning a general class of functions $\mathcal F$ to the relatively
well-understood problem of polynomial regression. In this application, the \emph{rate} of approximation, in terms of the degree
needed to achieve a certain accuracy of approximation, is again important because higher degree polynomials generally require more samples
and are computationally intensive to fit, especially in high dimensions. 

In this paper, we continue a line of work relevant to learning theory (e.g. \cite{klivans2013moment,GollakotaKlivansKothari23,chandrasekaran2024smoothed}), which studies approximation of functions by polynomials over distributions that are \emph{not product measures}, and therefore lack explicit Fourier bases. 
Without a Fourier basis, it is difficult to explicitly compute the coefficients of the best approximating polynomial up to a desired target. We develop a technique to deal with this problem using tools from classical harmonic analysis --- in particular, the Euclidean Fourier transform as well as complex analytic techniques, in the spirit of Bernstein's result.
We elaborate on these results, their applications, and their relation to previous work below.  

\subsection{Main results}

Our main theoretical contribution is a general framework that provides a quantitative characterization of the best polynomial approximation error in $L^2(\mu)$, where $\mu$ is a probability measure that satisfies Carleman's condition. 
The proposed framework essentially relies on the Fourier transform of the target function and of the $\mu$-weighted residual.
In applications, we use structural assumptions on the function class or problem instance to characterize the Fourier transform of the target functions.
Meanwhile, the Fourier transform of the $\mu$-weighted residual has a zero of order at least $D$ at the origin, which serves as our main building block to provide a quantitative characterization near the origin.   
These two elements together complete our theoretical framework as follows. 
\paragraph{Fourier-based $L^2(\mu)$ polynomial approximation bound} 

Let $\mu$ be a distribution on $\RR$ that satisfies Carleman's condition. 
We illustrate our results with a simple setting where the target function is the inverse Fourier transform of a finite Radon measure $\rho$ on $\RR$.  
Suppose that $f= (2\pi)^{-1} \int_{\RR} e^{ix\xi} d\rho(\xi)$ and set $p_D$ as the $L^2(\mu)$ orthogonal projection of $f$ onto the space of polynomials of total degree at most $D$.  
For any function $g\in L^2(\mu)$, we define $\norm{g}_\mu = \big(\int g^2  d\mu \big)^{1/2}$. 
We first state our approximation result for the special case where $\mu$ is strictly sub-exponential. 

\begin{theorem}[Strictly sub-exponential bound, informal, see \cref{thm:pw-strict}] \label{thm:strictly-subexp-informal}
Let $\mu$ be a strictly sub-exponential distribution, i.e., the moment generating function of $\mu$ is entire with finite order. 
Then for a fixed $\Omega>0$ and sufficiently large degree $D \gtrsim \mathrm{poly}(\Omega)$, we have that 
  \begin{align}
    \norm{f - p_D}_\mu \lesssim   \exp\{-O(D\log D)\} + |\rho|( \RR \setminus [-\Omega, \Omega]).  \label{eq:approx-strict}
  \end{align}
\end{theorem}
A similar result can be explicitly stated for sub-exponential distributions as follows.  
\begin{theorem}[Sub-exponential bound, informal, see \cref{thm:pw-subexp}] \label{thm:subexp-informal}
Let $\mu$ be a sub-exponential distribution, i.e., the moment generating function of $\mu$ exists in a neighborhood of zero. 
Then for any $\Omega>0$ and $D \ge 1$, we have that 
  \begin{align}
    \norm{f - p_D}_\mu \lesssim  \tanh(O(\Omega))^D + |\rho|( \RR \setminus [-\Omega, \Omega]).  \label{eq:approx-subexp}
  \end{align}
\end{theorem}
For any finite Radon measure $\rho$, the second term in \eqref{eq:approx-subexp} and \eqref{eq:approx-strict} decays at least as $O(\Omega^{-1})$ by Markov's inequality. 
Thus, we can choose a large enough $\Omega$ and an appropriately large $D$ to make the right-hand side smaller than any target accuracy.
This procedure successfully establishes quantitative guarantees of polynomial approximation for general functions without explicit knowledge of the underlying distribution.
We extend these results to multivariate settings to approximate functions that are sufficiently general to cover the target functions in applications (e.g., ReLU and Sigmoid neural networks). See \cref{sec:Rd-analytic-lemmas}.

\cref{thm:strictly-subexp-informal} and \cref{thm:subexp-informal} and their generalizations for extended function class  will be our main tools in the example applications, because for simplicity we will mostly focus on the concrete settings of strictly sub-exponential and sub-exponential distributions. 
Nevertheless the results can always extend naturally to the most general case under Carleman's condition. The general case involves a quantitative bound in terms of a certain logarithmic integral and recovers as a special case the sub-exponential bound above. 
See \cref{thm:qdc} for details. 



One interesting application of our approximation-theoretic results is a new characterization of the \emph{Paley-Wiener} class of bandlimited functions, which we leave to Section~\ref{sec:pw-1d} and Appendix~\ref{sec:pw-rd}. In essence, Paley-Wiener functions have the property that they can be very well-approximated by polynomials over the Gaussian measure --- we show that this property extends in a very natural way to all strictly sub-exponential measures. We also show related results for \emph{Gelfand-Shilov classes}, which decay rapidly in physical and frequency space --- see Section~\ref{sec:gs}. 

\paragraph{Application: Smoothed agnostic learning.} 
The analysis of \emph{agnostic learning} in computational learning theory has led to some striking negative results. Over general data distributions, there is strong evidence that even (improperly) learning a single halfspace is computationally intractable \cite{tiegel2023hardness,daniely2016complexity}. To better explain why learning often appears tractable in practice, recent work has adapted the smoothed analysis paradigm \cite{spielman2004smoothed} to agnostic learning.
Smoothed optimality \cite{blum2002smoothed,chandrasekaran2024smoothed} relaxes learners to find a hypothesis $\hat h$ whose prediction error can compete with the best prediction error of a hypothesis class $\cH$ when a small amount of noise is added to the data. 
Formally, given i.i.d. samples $\{(x_i, y_i)\}_{i\le n}\subset\RR^d \times \{\pm 1\}$, target accuracy $\eps>0$ and noise level $\sigma>0$, we aim to find $\hath$ such that
\begin{align}
  \PP\big(y_i \neq \hat{h}(x_i)\big) \leq \min_{h \in \cH} \EE_{z\sim N(0,I_d)}[\PP\big(y_i \neq h(x_i + \sigma z )\big)]+ \eps. \label{eq:smoothed-opt-informal}
\end{align}
Prior work \citep{chandrasekaran2024smoothed} shows that the smoothed optimality is indeed much more tractable than general agnostic learning, by establishing positive results for efficiently learning binary concepts with low intrinsic dimension and bounded Gaussian surface area over all strictly sub-exponential data distributions.
Surprisingly, we prove that low intrinsic dimension is sufficient to achieve this guarantee under our framework above, so the Gaussian surface area condition is unnecessary. 
We also extend the results to the data distributions with sub-exponential tails, thereby solving the main open problem of \cite{chandrasekaran2024smoothed}.
Our results are informally stated as follows:
\begin{theorem}[Learning smoothed optimality, informal, see \cref{thm:smoothed-learning}]
Let $\mu$ be a distribution over $\RR^d \times \{\pm 1\}$ and $\cH(k)$ be the class of binary concepts that depend only on a $k$-dimensional projection of the input. 
Then, there exists an algorithm that learns $\cH(k)$ in the $\sigma$-smoothed setting \cref{eq:smoothed-opt-informal} with runtime $\mathrm{poly}(N,d)$ and sample size
\begin{enumerate}
  \item $N= O(\eps^{-2} d^{\mathrm{poly}(k\sigma^{-2}\log(1/\eps))}\log(1/\delta))$ samples under strictly sub-exponential marginal $\mu_x$;
  \item $N=O(\eps^{-2} d^{\exp\{\tO( k \log(1/\eps) / \sigma )\}}\log(1/\delta))$ samples under sub-exponential marginal $\mu_x$.
\end{enumerate}
\end{theorem}
\begin{remark}[Comparison to \cite{chandrasekaran2024smoothed} in strictly sub-exponential case]
Besides removing the dependence on Gaussian surface area entirely, our result in the strictly sub-exponential case has quasipolynomial rather than exponential dependence on the error parameter $1/\epsilon$. I.e., our dependence is of the form $e^{polylog(1/\epsilon)}$ instead of $e^{poly(1/\epsilon)}$.
\end{remark}
Our improved sample complexity bound is based on $L^2(\mu_x)$ polynomial approximation results in \cref{sec:poly-approx-smoothed-targets}.
These results are specific applications of \cref{thm:strictly-subexp-informal,thm:subexp-informal} to Gaussian-smoothed functions. 

\subparagraph{Extension to learning intersections of half-spaces.} As demonstrated by \citep{chandrasekaran2024smoothed}, the problem of learning intersections of half-spaces with margin can largely be reduced to the problem of smoothed agnostic learning. 
Our improved sample complexity directly translates to an improved sample complexity bound for distribution-free learning of intersection of half-spaces with margin. The algorithm, similar to \cite{chandrasekaran2024smoothed}, uses a combination of dimension reduction and polynomial regression.
See details in \cref{sec:khalfspace}.  
\paragraph{Application: Polynomial approximation to smooth functions.} 
In addition to Gaussian-smoothed functions, our theory can also be applied to general smooth functions without explicit knowledge.
Indeed, Jackson's theorem \cite{jackson1930theory} asserts that smooth and periodic functions can be well approximated on a bounded interval. 
Our result further extends this result to general measures as follows: 
\begin{theorem}[Polynomial approximation to $C^k$ functions]
  Suppose that $f:\RR\to \RR$ is a $k$-th order differentiable function with $|f^{(k)}|\le M $. 
  Let $\mu$ be a distribution on $\RR$. 
  Then there exists a polynomial $p$ of degree 
  \begin{enumerate}
    \item $O(\mathrm{poly}(1/\eps)) $ under strictly sub-exponential $\mu$; 
    \item $O(\exp\{\mathrm{poly}(1/\eps)\})$ under sub-exponential marginal $\mu$.
  \end{enumerate}
\end{theorem}

Interestingly, the proof of this theorem relies on the classical Jackson's theorem and the bounded support of trigonometric polynomials in the frequency domain.  
Our results also connect to classical weighted polynomial approximation theory and our rates match the known tight exponents in both settings. 
See \cref{app:learning-smooth-function} for details.

\subparagraph{Implications in agnostic learning.} 
Existence of polynomial approximation often enables polynomial time learning algorithms. 
Suppose that $\sigma:\RR\to \RR$ is an activation function.
Given the class of the one layer neural network $\cH_\sigma = \{h(x) = k^{-1}\sum_{j\le k} \sigma(w_j^\top x); \norm{w_j} =1\}$ and i.i.d. samples from a distribution $\mu$, a learner needs to find a hypothesis $\hat{h}$ that with probability $1-\delta$ satisfies
\begin{align}
  \EE_{(x,y)\sim \mu}[(\hat{h}(x) - y)^2] \leq \min_{h \in \cH_\sigma} \EE_{(x,y)\sim \mu}[(h(x) - y)^2] + \eps.
\end{align}
We mainly demonstrate learning neural networks with ReLU and Sigmoid activations, as stated in \cref{thm:sigmoid_samples,thm:relu_samples}. 
Although applying our theory of polynomial approximation of the general Lipschitz functions also implies efficient agnostic learning algorithms, the corresponding sample complexity might be sub-optimal. 
When exact knowledge of the function class is given, our method directly yields better sample complexity.

\subsection{Overview}
\paragraph{Key difficulty: lack of an explicit Fourier basis.} When working over the Gaussian distribution, Hermite polynomials form an explicit Fourier basis for $L^2(N(0,1))$ with many nice algebraic properties.
See \cite{ODonnell2014} for a detailed discussion of Fourier analysis in Gaussian space and in the closely related setting of the boolean hypercube. However, we want to prove results for general distributions with sub-Gaussian or sub-exponential concentration. This makes working with an explicit basis untenable because: (1) the basis of orthogonal polynomials for each distribution in the class is different, and (2) as such distributions are not required to be product measures, their orthogonal bases lack the useful algebraic properties of Fourier bases which makes them easy to work with. 

So to prove results over general distributions, we generally want to take a basis-free approach. Indeed, it has been shown that some results in discrete Fourier analysis can be proven without relying on the Fourier basis (see, e.g., \cite{cordero2012hypercontractive,ivanisvili2024eldan,koehler2024influences,rosenthal2020ramon,ivanisvili2020rademacher}). Similarly, the line of work in smoothed analysis we build on (e.g., \cite{klivans2013moment,GollakotaKlivansKothari23,chandrasekaran2024smoothed}) has identified this core problem and developed techniques to overcome it, which we discuss next.

\paragraph{The existing strategies for approximation.} Perhaps the most natural way to try to prove results for $L^2$ approximation over general measures is to try to reduce to the classical setting of approximation by polynomials on an interval (see, e.g., \cite{devore1993constructive,akhiezer2020classical}). For any $M > 0$, classical approximation theory\footnote{Another classical area in approximation theory, related to the moment problem, studies \emph{weighted approximation} by polynomials. We discuss this more in Section~\ref{subsec:discuss}.} tells us precise results about the best approximation of a target function $f$ on the interval $[-M,M]$ in $L_{\infty}$. However, while a polynomial may be a good $L_{\infty}$ approximation to $f$ on $[-M,M]$, it will generally blow up rapidly outside of the interval (c.f., uncertainty principles \cite{nazarov1993local}). So while such an argument directly yields an approximation which is accurate under a high probability event (see, e.g., \cite{koehler2018comparative}), it does not immediately give a bound on the $L_p$ error due to tail events. 
Carefully working out the necessary bounds, this approach can in fact be made to work in the strictly sub-exponential case \cite{chandrasekaran2025learning}, but it yields suboptimal rates (see Section~\ref{subsec:discuss}), and it fails when the distribution has sub-exponential tails. 

Existing works have also used more sophisticated techniques. The works \cite{klivans2013moment,GollakotaKlivansKothari23} use an LP duality argument to relate the existence of a good polynomial approximator to the analysis of moment matching distributions. This lets them argue that a good polynomial approximation exists without having to explicitly construct one. The most recent work \cite{chandrasekaran2024smoothed} focused on Gaussian-analytic functions (i.e., a function convolved with a Gaussian) and carefully analyzed a polynomial approximation built using a truncated Taylor series expansion of the Gaussian density. None of these techniques has been successful in solving the sub-exponential case, so we will use a different approach.

\paragraph{New strategy for approximation.} Although in many agnostic learning settings the relevant metric for approximation is actually $L^1$ (see, e.g., \cite{diakonikolas2021optimality,klivans2013moment,kalai2008agnostically}), we focus on $L^2$ approximation so that we can apply Plancherel's theorem from classical Fourier analysis. We can still obtain results for $L^1$ by using Jensen's inequality.

We now explain our strategy for $L^2$ approximation.
Given a target function $f$ over probability measure $\mu$ and degree $D$, we consider the \emph{orthogonal projection} $g_D$ onto the space of degree $D$ polynomials, i.e. the $g_D = \argmin_g \|f - g\|_{\mu}^2$ where $g$ ranges over degree $D$ polynomials. From the first-order optimality conditions for $g_D$, we know that
\[ \langle f - g_D, x^k \rangle_{\mu} = 0 \]
for all $0 \le k \le D$ and similarly
\[ \langle f - g_D, g_D \rangle_{\mu} = 0. \]
Using Plancherel's theorem and the fact that Fourier transform replaces multiplication by $x$ with differentiation, we can equivalently reinterpret the former condition in terms of the Fourier transform $\mathcal F [(f - g_D)\mu](\xi) = \int (f - g_D)(x) e^{-i \xi x} d\mu$ to obtain
\begin{equation}\label{eqn:overview1}
\frac{d^k}{d\xi^k} \mathcal F [(f - g_D)\mu](\xi)\Big|_{\xi = 0} = 0 
\end{equation}
for all $0 \le k \le D$, and rewrite the approximation error as
\begin{equation}\label{eqn:overview2}
\|f - g_D\|_{\mu}^2 = \langle f, f - g_D \rangle_{\mu} = \frac{1}{2\pi} \int \mathcal F[f](\xi) \overline{\mathcal F[(f - g_D)\mu)]}(\xi) d\xi. 
\end{equation}
If $f$ is a smooth function, then its Fourier transform $\mathcal F[f]$ should be mostly supported on low frequencies, i.e. $\xi$ close to $0$. This suggests that the integral on the right hand side of \eqref{eqn:overview2} should be dominated by the contribution of small $\xi$. On the other hand, from \eqref{eqn:overview1} we know that $\mathcal F [(f - g_D)\mu]$ has a zero of degree $D$ at $0$. 
\emph{Intuitively}, we would like to deduce via Taylor series expansion about $0$ that $\mathcal F [(f - g_D)\mu](\xi) \approx 0$ for all ``small'' $\xi$, and conclude that the approximation error from \eqref{eqn:overview2} goes to $0$ as $D \to \infty$. However, the last step of the argument is not always valid --- it cannot be, because the moment problem does not always have a unique solution \cite{akhiezer2020classical}.

\paragraph{Special case: strictly sub-exponential distributions.} The previous works \cite{chandrasekaran2024smoothed,GollakotaKlivansKothari23} considered the case where $\mu$ has tails which decay like $e^{-|x|^{1 + \epsilon}}$ for $\epsilon > 0$. In this case we say that the tails of $\mu$ are \emph{strictly sub-exponential}, and it is not hard to show that $\mathcal F [(f - g_D)\mu](\xi)$ is indeed an entire function of $\xi$. So by formalizing the Taylor series argument, we can indeed bound the approximation error as sketched above. This lets our approach recover the same result as \cite{chandrasekaran2024smoothed} with some quantitative improvements. 

However, Taylor series expansion is only applicable up to its radius of convergence. If $\mu$ only has sub-exponential tails the radius of convergence will be finite (possibly zero). So with this idea alone, we cannot solve the open problem from Chandrasekaran et al \cite{chandrasekaran2024smoothed}. We next discuss how we modify our approach to handle the sub-exponential case, and then even more general case of Carleman's condition. 

\paragraph{Carleman's condition.} Carleman \cite{carleman1926fonctions} proved a fundamental result in analysis called the \emph{Denjoy-Carleman theorem}. The Denjoy-Carleman theorem precisely characterizes when we can conclude that a function $f : \mathbb{R} \to \mathbb{R}$ is zero, if we know that $f$ has a zero of infinite order (i.e., for some point $x$, $f(x) = 0$ and all of its derivatives also vanish). In the special case that $f$ is analytic, this is true due to the principle of \emph{analytic continuation}, and Carleman showed that this is true for the larger class of \emph{quasianalytic} functions, thereby proving what is known as ``Carleman's condition'' for the moment problem.  

\paragraph{Quantitative Denjoy-Carleman.} Carleman's original proof of the Denjoy-Carleman theorem is via a clever application of the principle of analytic continuation to show that the \emph{one-sided Laplace transform} of $f$ vanishes. Therefore, by Laplace inversion $f$ must vanish as well. This is not immediately useful for us, because this argument only applies in the limit $D = \infty$. There does exist a different \emph{real-analytic} proof of Carleman's theorem which gives effective bounds for finite $D$ --- see Chapter I of H\"ormander's textbook \cite{hörmander1983analysis}. Using the real-analytic argument, we could prove a result under Carleman's general condition. Unfortunately, the resulting bounds would have disappointingly poor quantitative dependence on $D$ even in the sub-exponential case. Instead, we found a way to strengthen Carleman's complex-analytic approach and obtain much more promising results (e.g., doubly exponentially faster rates in the sub-exponential case). We leave further details to Section~\ref{sec:qdc}.

\begin{remark}[Comparison to techniques in previous works]
The works \cite{klivans2013moment,kane2013learning,GollakotaKlivansKothari23} on agnostic and testable learning also observe and build upon connections to the classical moment problem. In fact, the work \cite{klivans2013moment} builds on results in Chapter 10 of a book by Rachev et al \cite{rachev2013methods} which use complex analytic methods. Nevertheless, we obtain significantly stronger results in the end. This is discussed further in Remark~\ref{rmk:comparison-rachev}.
\end{remark}

\paragraph{Applications.} We omit most of the details of the applications from this high-level overview. In the smoothed analysis application relevant to \cite{chandrasekaran2024smoothed}, the basic idea is to take advantage of the fact that Gaussian-analytic functions have sub-Gaussian tails in Fourier space when applying our main results. Stating and proving the most general versions of our results in higher dimensions is slightly involved (see Section~\ref{sec:prelim2} and Section~\ref{sec:Rd-analytic-lemmas}).
The applications to universality in Paley-Wiener and Gelfand-Shilov spaces use a similar idea to apply the main result, but require us to introduce some other tools from harmonic and functional analysis such as the \emph{Bargmann transform} \cite{grochenig2001foundations} which we elaborate upon in Section~\ref{sec:prelim3}. For functions with relatively poor smoothness and/or tail behavior, we observe that we can obtain approximation results through a three-step procedure: (1) approximate the target by the sum of a polynomial and a periodic function, (2) apply existing results in approximation theory to approximate the periodic function by a trigonometric polynomial, and (3) approximate the trigonometric polynomial using our main result. See Section~\ref{sec:app-learning-nn}.

\subsection{Related work}
\paragraph{Polynomial approximation and the Statistical Query (SQ) framework.} The Statistical Query (SQ) framework models learning algorithms which are only allowed to access their training data by querying  approximate expectation values of bounded functions (i.e., by receiving $\mathbb E[f] \pm [-\tau,\tau]$ for some error tolerance $\tau$). This framework has enabled the prediction of optimal \emph{computational-statistical} tradeoffs for many learning problems. See, e.g., \cite{diakonikolas2017statistical,diakonikolas2023algorithmic}.

It turns out for many learning problems, their SQ complexity is closely connected with the existence of \emph{moment-matching distributions}. A key reason for this is connections to a general class of problems called \emph{Non-Gaussian Component Analysis} (NGCA), where a single non-Gaussian direction is planted in what is otherwise an isotropic Gaussian distribution. In NGCA, the difficulty\footnote{Informally, the time complexity. The number of queries and tolerance of queries needed for an SQ algorithm can be used as proxies to predict the needed runtime and sample complexity of an algorithm.} of finding and/or detecting the hidden plant using SQ queries is determined by the number of moments that the planted distribution matches with a Gaussian (see, e.g., \cite{diakonikolas2023sq}). In most cases it is believed that problems with lower bounds from NGCA are genuinely hard for polynomial time algorithms (see, e.g., \cite{diakonikolas2024sum} for lower bounds against the SoS semidefinite programming hierarchy and, e.g., \cite{song2021cryptographic,tiegel2023hardness} for some conceptually related cryptographic lower bounds).  

Based on these types of ideas, for agnostic learning over Gaussian distributions the SQ complexity of learning a class of concepts $\mathcal C$ has been shown to be closely connected to the degree of polynomials needed to approximate it (see, e.g., \cite{diakonikolas2021optimality,gollakota2020polynomial,kalai2008agnostically,klivans2013moment}). This suggests that algorithms based on $L_1$ polynomial regression, which is easy to implement in the SQ framework, are close to optimal algorithm for solving many agnostic learning problems over Gaussian data in polynomial time. For this reason, the difficulty of approximating many classes of concepts over the Gaussian distribution has been closely studied. In many cases, tight results can be obtained by working with explicit Hermite expansions (e.g., in single-index models \cite{diakonikolas2021optimality}). Similar findings have also been established over the uniform measure on the hypercube --- see \cite{dachman2014approximate}. Polynomial approximation has also been established to play an important role in the related setting of learning with contamination by the recent results of \cite{klivans2025power}.

A potential worry is that while we have developed a precise understanding of agnostic learning over particular distributions like $N(0,I)$ or $Uni \{\pm 1\}^n$, it is less clear to what extent these findings extend to other distributions. Our results, in line with previous work on smoothed analysis such as \cite{klivans2013moment,GollakotaKlivansKothari23,chandrasekaran2024smoothed}, show that many guarantees for learning with polynomial regression over Gaussian space can be extended, perhaps with a small loss, to general sub-Gaussian distributions, and more generally to larger classes of distributions where polynomials are dense.

\paragraph{Cryptographic lower bounds for agnostic learning.} Another approach to showing hardness of learning results is via reductions from average case problems which are believed to be hard. See, e.g., \cite{daniely2016complexity,tiegel2023hardness,klivans2014embedding,klivans2009cryptographic,daniely2021local,song2021cryptographic}. We discuss some connections with cryptographic lower bounds in our Section~\ref{sec:crypto}.

\paragraph{Comparing sub-Gaussian and Gaussian measures.} There is a lot of interest in probability theory, mathematical physics, and statistics in extending results obtained in the Gaussian setting to the sub-Gaussian setting and beyond. See, e.g., \cite{anderson2010introduction,han2023universality,bayati2015universality,vershynin2018high} for some context in the case of random matrix theory and high-dimensional statistics. A major result in this context is Talagrand's comparison theorem which shows that the maxima of sub-Gaussian processes can be controlled by the maxima of corresponding Gaussian processes \cite{talagrand2014upper,vershynin2018high}. We mention a way to combine our result with recent work in the SoS literature built on Talagrand's result \cite{diakonikolas2025sos}  in Appendix~\ref{apdx:testable}. In a related spirit, in the case of Paley-Weiner functions, we show that the rate of approximation over Gaussian measures actually extends to all sub-Gaussian ones --- see Section~\ref{sec:pw-1d} and Appendix~\ref{sec:pw-rd}.







\paragraph{Polynomial approximation in $L_p$ spaces.} Besides the classical results on the solvability of the moment problem (see, e.g., \cite{carleman1926fonctions,akhiezer2020classical,rachev2013methods}), in some particular cases polynomial approximation in $L_p$ spaces has been deeply studied in the literature on weighted approximation. 
For instance, approximation theory in $L_p$ for probability measures of the form
\[ p(x) \propto \exp(-|x|^{1 + \alpha}) \]
for $\alpha \ge 0$ is very well-understood. See, e.g., \cite{bizeul2025polynomial} which combines complex analytic methods with special polynomial families to get some very precise results for the symmetric exponential distribution, as well as \cite{freud1977markov,freud1978approximation,lubinsky2006jackson} and the references within. To compare, our methods seem somewhat specialized to the $L_2$ setting; although we can deduce implications for, e.g., $L_1$ space using Jensen's inequality. On the other hand, our method is very general in the sense that it can be applied to any distribution satisfying Carleman's condition. We discuss the approximation theory literature in more detail in Section~\ref{subsec:discuss}.

\paragraph{Nonparametric statistics.} Our approach to polynomial approximation is naturally motivated if we first seek to build good approximations for bandlimited functions, and then extend the techniques to more general smooth functions. This perspective is morally informed by the fact that in nonparametric statistics, convolution with the \emph{sinc} kernel, which is the $L_2(\mathbb R)$ projection operator onto the space of bandlimited functions, is rate-optimal for many estimation problems.
In particular, this phenomena has been deeply studied in the literature on Fourier methods and nonparametric density estimation (e.g., \cite{davis1977mean,ibragimov1983estimation, devroye1992note,devroye2001combinatorial,hall1988choice,tsybakov2008nonparametric}), where the sinc kernel is a prototypical example of a ``superkernel'' because it adapts to the smoothness class of the underlying distribution. Gaussian-analytic functions --- equivalently, functions constructed via convolution with a Gaussian kernel --- also play an important role in diferent parts of this paper.

\paragraph{Approximating neural networks by polynomials.} The vast majority of work on approximating neural networks by polynomials has focused on the case where the data distribution satisfies very strong norm constraints (e.g., the weights of the network have $\ell_2$ norm $O(1)$ and the data has $\ell_2$ norm $O(1)$) --- see the discussion of existing work in \cite{goel2019learning,goel2020boltzmann}. This leads to very pessimistic results when applied over the Gaussian distribution $N(0,I_n)$ where the data has typical $\ell_2$ norm on the order of $O(\sqrt{n})$. In the Gaussian case, existing results for approximating a single neuron over the Gaussian measure (see, e.g., \cite{diakonikolas2021optimality}) do imply much better results, similar to what we have obtained. For the sub-Gaussian case, we are not aware of many relevant results, though it is possible to obtain bounds by reducing to the single-neuron setting in \cite{chandrasekaran2024smoothed}. As in other applications, our results lead to significant quantitative improvements (esp. in terms of error dependence $\epsilon$) than what one could derive from \cite{chandrasekaran2024smoothed}, and also allow for establishing results over general measures satisfying Carleman's condition.

\section{Preliminaries}
\subsection{Fourier transform}
In this section, we state our conventions for the Fourier transform and recall the key properties of the transform.
See the textbook \cite{rudin1987real} for a much more detailed reference covering essentially all of the content below, with the same normalization convention.
\paragraph{Fourier transform}
Let \(f\colon\mathbb{R}\to\mathbb{C}\) be a sufficiently well-behaved function (e.g. an element of $L_1(\R) \cap L_2(\R)$).  We define its Fourier transform and inverse by
\[
\hat f(\xi)
\;=\;\cF[f](\xi)
\;=\;\int_{-\infty}^{\infty}f(x)\,e^{-i\xi x}\,dx,
\qquad
f(x)
\;=\;\cF^{-1}[\hat f](x)
\;=\;\frac1{2\pi}\int_{-\infty}^{\infty}\hat f(\xi)\,e^{i\xi x}\,d\xi.
\]
Classical results show that these operations are indeed inverses and that the Fourier transform extends to an operator on all of $L_2(\R)$.

\paragraph{Differentiation.} 
Differentiation with respect to \(x\) corresponds to multiplication by \(i\xi\) in the Fourier domain.  In particular, for each \(n\in\mathbb{N}\),
\[
\cF\bigl[f^{(n)}\bigr](\xi)
=\int_{-\infty}^\infty f^{(n)}(x)\,e^{-i\xi x}\,dx
=(i\xi)^n\,\hat f(\xi),
\]
so that
\[
f^{(n)}(x)
=\cF^{-1}\bigl[(i\xi)^n\,\hat f(\xi)\bigr](x)
=\frac1{2\pi}\int_{-\infty}^{\infty}(i\xi)^n\,\hat f(\xi)\,e^{i\xi x}\,d\xi.
\]

Analogously, differentiation with respect to \(\xi\) corresponds to multiplication by \(-i x\) in the original domain.  For each \(n\in\mathbb{N}\),
\[
\frac{d^n}{d\xi^n}\,\hat f(\xi)
=\int_{-\infty}^\infty f(x)\,\frac{d^n}{d\xi^n}e^{-i\xi x}\,dx
=(-i)^n\int_{-\infty}^\infty x^n\,f(x)\,e^{-i\xi x}\,dx
=(-i)^n\,\cF\bigl[x^n f(x)\bigr](\xi).
\]

\paragraph{Plancherel's theorem}

For \(f,g\in L^2(\mathbb{R})\), the Fourier transform extends to an isometry (up to the chosen normalization) between the time and frequency domains.  Concretely,
\[
\langle f,g\rangle_{L^2(\mathbb{R})}
\;=\;\int_{-\infty}^{\infty}f(x)\,\overline{g(x)}\,dx
\;=\;\frac1{2\pi}
\int_{-\infty}^{\infty}\hat f(\xi)\,\overline{\hat g(\xi)}\,d\xi
\;=\;\frac1{2\pi}\,\langle \hat f,\hat g\rangle_{L^2(\mathbb{R})}.
\]
Equivalently,
\[
\|f\|_{L^2}^2
=\langle f,f\rangle
=\frac1{2\pi}\|\hat f\|_{L^2}^2.
\]
This identity is often used to compute inner products or norms in the time domain by passing to the frequency domain (or vice versa).

\paragraph{Convolution.}
For two functions \(f,g\colon \mathbb{R}\to\mathbb{C}\), 
their \emph{convolution} \(f * g\colon \mathbb{R}\to\mathbb{C}\) is defined by
\[
(f * g)(x)
\;=\;
\int_{-\infty}^{\infty} f(x - y)\,g(y)\,dy.
\]
The Fourier transform interchanges convolution and multiplication: 
\[
\cF[f * g]
\;=\;
\cF[f]\;\cdot\;\cF[g],
\qquad
\cF[\,f\,g\,]
\;=\;
\frac{1}{2\pi}\,\bigl(\cF[f] * \cF[g]\bigr).
\]

Define the normalized Gaussian (the \(N(0,\sigma^2)\) density)
\[
\varphi_\sigma(x)
\;=\;\frac{1}{\sqrt{2\pi}\,\sigma}
       \,\exp\!\Bigl(-\frac{x^2}{2\sigma^2}\Bigr),
\]
and let $\varphi(x)$ without a subscript denote $\varphi_1$.
Observe that
\[
\cF\bigl[\,\varphi_\sigma\,\bigr](\xi)
\;=\;\int_{-\infty}^{\infty}
       \frac{1}{\sqrt{2\pi}\,\sigma}\,
       \exp\!\Bigl(-\frac{x^2}{2\sigma^2}\Bigr)
       e^{-\,i\,\xi\,x}\,dx
\;=\;\exp\!\Bigl(-\tfrac{\sigma^2\,\xi^2}{2}\Bigr)\,.
\]

If \(f\colon\R\to\C\) and we set
\[
(f * \varphi_\sigma)(x)
\;=\;\int_{-\infty}^{\infty}
       f(y)\,\varphi_\sigma(x-y)\,dy,
\]
then, since the Fourier transform sends convolution to multiplication, we find
\[
\cF\bigl[\,f * \varphi_\sigma\,\bigr](\xi)
\;=\;\widehat{f}(\xi)\,\widehat{\varphi_\sigma}(\xi)
\;=\;\widehat{f}(\xi)\,\exp\!\Bigl(-\tfrac{\sigma^2\,\xi^2}{2}\Bigr)\,,
\]
and hence
\[
(f * \varphi_\sigma)(x)
\;=\;\frac{1}{2\pi}\,
    \int_{-\infty}^{\infty}
      \widehat{f}(\xi)\,
      \exp\!\Bigl(-\tfrac{\sigma^2\,\xi^2}{2}\Bigr)\,
      e^{\,i\,\xi\,x}\,d\xi\,.
\]
So convolving \(f\) with a Gaussian of variance \(\sigma^2\) is equivalent to
multiplying \(\widehat{f}(\xi)\) by the factor \(e^{-\sigma^2\,\xi^2/2}\) in the frequency domain. 

\subsection{Complex analysis}
Mostly for the proofs of the results in Section~\ref{sec:qdc}, we will need to use some basic terminology and concepts from complex analysis such as the Cauchy inequalities, contour integration, and the maximum modulus principle. 
See the textbooks \cite{rudin1987real} or \cite{stein2010complex} for references.

\subsection{Organization of the paper}
The key ideas for our results are contained in Section~\ref{sec:qdc}, which proves the quantitative version of the Denjoy-Carleman theorem, and Section~\ref{sec:one-dim-approx}, which shows how we can thereby derive bounds on $L^2$ approximation by polynomials in the one-dimensional case. The one-dimensional version of the approximation result already suffices for many but not all of the applications.

In order to properly state and prove the analogous polynomial approximation result in higher dimensions, we need to recall more background from functional analysis, which is done in Section~\ref{sec:prelim2}. Given this, the results from Section~\ref{sec:one-dim-approx} are generalized appropriately in Section~\ref{sec:Rd-analytic-lemmas}. 

We then go through a few applications of our results in the context of learning theory.  
In~\cref{sec:app-smoothed-analysis}, we include the result for smoothed analysis of agnostically learning binary concepts of low-intrinsic dimension. 
Following this, \cref{sec:khalfspace} collects a further implication of the result in smoothed analysis for learning half-spaces with margin. 
In addition, we illustrate the use of our framework over functions with specific structures of smoothness in \cref{sec:app-learning-nn}.
In Appendix~\ref{apdx:testable}, we explain how in the sub-Gaussian case, our results can be combined with Sum-of-Squares \cite{diakonikolas2025sos} to give ``testable learning'' algorithms \cite{RubinfeldVasilyan23,GollakotaKlivansKothari23}.
In Section~\ref{sec:crypto}, we explain how in some well-studied cases, our algorithmic results are close to matching lower bounds from conjecturally hard average-case problems. 

The main body of the paper ends with the universality-style results for Paley-Wiener and Gelfand-Shilov classes. First we go through a preliminaries section which recalls the relevant definitions and some related concepts such as the \emph{Bargmann transform} which are needed for the proofs \ref{sec:prelim3}. Then in Section~\ref{sec:pw-1d} we prove the result for Paley-Wiener class and in Section~\ref{sec:gs} we prove an analogous result for the one-sided Gelfand-Shilov class $\mathcal S^{1/2}$. In these sections we are concerned with 1-dimensional settings so the results from Section~\ref{sec:Rd-analytic-lemmas} are not required. In Appendix~\ref{sec:pw-rd}, we prove the natural $d$-dimensional generalization of the Paley-Wiener result.
The remaining appendices contain deferred proofs. 
\section{Quantitative Denjoy-Carleman via complex analysis}\label{sec:qdc}
In this section, we state our quantitative version of the Denjoy-Carleman theorem. Our methods are based on complex analysis and give significantly better control than what one can implicitly infer from the elementary/real-analytic proof of the Denjoy-Carleman theorem --- we discuss this further in the last subsection.

For many applications, it suffices to use the special case of our results in the analytic (rather than quasianalytic) case. So we encourage the first-time reader to focus on Section~\ref{subsec:qdc-analytic}.
\subsection{Preliminaries: quasianalytic functions}
 Let \(M=(M_k)_{k\ge0}\) be a sequence of positive real numbers satisfying
\[
M_0 = 1
\]
and
\[
M_k^2 \le M_{k-1}\,M_{k+1}
\quad\text{(log–convexity).}
\]
We define the corresponding Denjoy–Carleman class on the positive real line \([0,\infty)\) by
\[
\mathcal{C}\{M_k\}([0,\infty))
=\Bigl\{f\in C^\infty([0,\infty)) :
\;\exists B,K>0\;\text{such that}\;
\sup_{x\in[a,b]}\bigl|f^{(k)}(x)\bigr|\le B K^{k}\,M_k
\quad\forall\,k\ge0
\Bigr\}.
\]
As usual, we also define $a_k = M_{k - 1}/M_k$ with $a_0 = 1/M_0 = 1$; then $1/M_k = \prod_{i = 1}^k a_i$ and log-convexity implies that $a_k$ is a monotonically decreasing sequence. Define the \emph{associated function}
\[ \tau(r) = \inf_{n \in \N} M_n/r^n. \]
It can be noted that $\tau(r)$ is a monotonically decreasing function of $r > 0$ and is rapidly decreasing in the sense that it asymptotically goes to zero faster than any polynomial in $r$. The associated function has a useful reformulation in terms of the $a_k$'s, which follows from log-convexity:
\begin{lemma}[Section 5.2.3 of \cite{katznelson2004introduction}]
For any $r > 0$,
\[ \tau(r) = \prod_{k \ge 1 :\ a_k r > 1} (a_k r)^{-1}. \]
\end{lemma}
In particular, if $r \le 1$ then $\tau(r) = 1$. 
For any $N \ge 0$ define
\[ \tau_N(r) = \inf_{0 \le n \le N} M_n/r^n = \begin{cases} \tau(r) & \text{if $r \le 1/a_{N + 1}$} \\ M_N/r^N & \text{otherwise} \end{cases}\]
\begin{theorem}[Denjoy-Carleman Theorem \cite{carleman1926fonctions,rudin1987real}, see also \cite{katznelson2004introduction,hörmander1983analysis}]
The following are equivalent:
\begin{enumerate}
    \item $\sum_k a_k = \infty$. (Equivalently, one may check
$
\sum_{k=1}^{\infty}M_k^{-1/k}
\;=\;\infty$.)
    \item $\int_1^{\infty} \frac{\log \tau(r)}{1 + r^2} dr = -\infty$.
    \item The class $\mathcal C\{M_k\}([0,\infty))$ is \emph{quasianalytic}: for any point $x$, if all of the derivatives of $f \in \mathcal C\{M_k\}([0,\infty)$ vanish at $x$, then $f = 0$ everywhere.
\end{enumerate}
\end{theorem}
We will later give a quantitative version of the ``sufficiency'' direction (that either $1$ or $2$ implies $3$) of this result. In the other direction, e.g. under the assumption that $\sum_k a_k < \infty$, there are explicit constructions of compactly supported nonzero functions which show that the class is not quasianalytic: see \cite{rudin1987real,katznelson2004introduction} for Fourier-analytic constructions which can be motivated by the Paley-Wiener theorem, and see \cite{hörmander1983analysis} for a real-variable bump function construction.
\subsection{Analytic case}\label{subsec:qdc-analytic}
Before handling the more general quasianalytic case, we start with some results for the analytic case which suffice for many applications and admit very short proofs.

In the analytic case, the Denjoy-Carleman theorem reduces to the principle of analytic continuation. There are tools in complex analysis, such as the Cauchy inequalities and Hadamard three-circle/three-lines theorem, which can provide quantitative analogues of analytic continuation. See \cite{trefethen2020quantifying,demanet2019stable,franklin1990analytic,grabovsky2021optimal} for a more extensive discussion and history of quantitative analytic continuation principles. 
In the literature, the results typically concern error guarantees for extrapolating a function from a smaller set to a larger one, whereas in our case we assume the function has a zero of high multiplicity at a single point.

\begin{remark}\label{rmk:comparison-rachev}
In the existing literature, the arguments for the analytic case below seem most similar to what is used in Chapter 10 of the textbook of Rachev et al \cite{rachev2013methods}. In this chapter they derive upper bounds on the metric between distributions $F$ and $G$ given by
\[ \lambda(F,G) = \min_{T > 0}\max\{\frac{1}{2} \max_{|t| \le T} |f(t) - g(t)|, 1/T\} \} \]
where $f$ and $g$ are the characteristic functions of $F$ and $G$ respectively.
In particular, the authors use the same technique of conformal mapping from the strip to the disk using $\tanh$ which we do below --- see the intermediate Equation (10.2.54) from the proof of Theorem 10.2.2 there. They also observe connections to the proof of Bernstein's famous approximation theorem for analytic functions on $[-1,1]$ \cite{bernstein1912ordre,devore1993constructive}. 

Previous work \cite{klivans2013moment,kane2013learning,GollakotaKlivansKothari23} built on estimates from \cite{rachev2013methods} to obtain downstream consequences for agnostic and testable learning. However, in many cases their quantitative bounds for agnostic learning are quantitatively worse than the more recent work \cite{chandrasekaran2024smoothed}, which we are in turn improving on. One intuitive reason for the inefficiency is that as we match more moments between $F$ and $G$, the rate at which $|f(t) - g(t)| \to 0$ for small $t$ can be much faster than the rate at which the interval $[-T,T]$ on which $\max_{t \in [-T,T]} |f(t) - g(t)| \approx 0$ grows. The $\lambda$-distance metric essentially measures the second, slower, rate --- so for our applications, it loses useful quantitative information.  
\end{remark}
\subsubsection{Analytic in a strip}
\begin{lemma}\label{lem:complex}
Suppose that $\varphi$ is an analytic function in the strip $\{\xi:|\Im(\xi)| < \pi/4 \}$ with a  zero of order $D$ at $\xi=0$,  such that $\sup_{\xi : |\Im(\xi)| < \pi/4} |\varphi(\xi)| \le 1$.
 Then for any $\xi \in \mathbb R$, it holds that
\begin{align}
     |\varphi(\xi)| \le \tanh(|\xi|)^D  . 
\end{align}
\end{lemma}
\begin{proof}[Proof of \cref{lem:complex}]
By basic complex analysis, we know that the function $\tanh^{-1}(z)$, in its principal branch, is analytic on the disc $\{\xi: |\xi|<1\}$ and maps the unit disc to the strip $\{\xi: |\Im(\xi)| < \pi/4\}$.
This is becuase for $|u|<1$, it holds that   
\begin{align}
  \big|\Im[\tanh^{-1}(u)]\big| = \left|\frac{\Im[\log\frac{1 + u}{1 - u}]}{2}\right| =\frac{1}{2}\mathrm{Arg}\, \left(\frac{1 + u}{1 - u}\right)   \le \frac{\pi}{4}, 
\end{align}  
as the map $u\mapsto (1+u)/ (1-u)$ is a conformal map from the unit disc to the right half-plane. 

Therefore, we define the function $\zeta(u) = \varphi\big(\tanh^{-1}(u)\big)$, and observe that $\sup_{\|u\| \le 1} |\zeta(u)| \le 1$. 
Note that $(\tanh^{-1})' = 1/(1 - u^2)\neq 0$ for $|u| < 1$, $\tanh^{-1}(0) = 0$ and that $\xi = 0$ is an order-$D$ zero of $\varphi$, the function $\zeta$ is analytic in the unit disc $\{u: |u| < 1\}$ and has a zero of order $D$ at $u=0$.
Therefore $\zeta(u)/u^D$ is analytic on the disc $\{u: |u| < 1\}$, and the maximum modulus principle implies that 
\begin{align}
  \sup_{\|u\| < 1} |\zeta(u)/u^D| = \lim_{r \to 1} \sup_{\|u\| = r} |\zeta(u)/u^D| \le 1.
\end{align}
Observe that for any $\xi \in \mathbb R$, defining $v = \tanh(\xi)$, we have that $\zeta(v) = \varphi\left(\tanh^{-1}(v)\right) = \varphi(\xi)$. 
Therefore,  $|\varphi(\xi)| = |\zeta(v)| \le |v|^D =\tanh(|\xi|)^D$. 
\end{proof}
\begin{remark}
This bound is sharp since the function $g(x) = \tanh(x)^D$ satisfies these hypotheses. (The maximum value of $\tanh$ on the strip is attained on the imaginary axis, because $\tan$ has all nonnegative coefficients in its Taylor expansion.)
\end{remark}

\subsubsection{Entire functions}
\begin{lemma}
Suppose that $\varphi$ is an entire function that is not identically zero, and let
\[ M(r) = \max_{\|z\| = r} |M(z)|. \]
Suppose $\varphi$ has a zero of order $D$ at zero. Then for any $r > 0$,
\[ \log M(r) \le \inf_{\lambda \ge 1}[ \log(M(\lambda r)) - D\log(\lambda)]\]
\end{lemma}
\begin{proof}
Note that $\varphi(\xi)/\xi^D$ is analytic by assumption. 
So by the maximum modulus principle applied to the function $\varphi(\xi)/\xi^D$, we have that for any $\xi$ with $|\xi| = r$,
\begin{align}
|\varphi(\xi)| \le \frac{1}{\lambda^D}  \sup_{|\zeta| =  \lambda |\xi |} |\varphi(\zeta)| 
\end{align}
which yields the result after taking a logarithm.
\end{proof}
\begin{proof}[Alternative proof]
By the Hadamard three-circle theorem \cite{rudin1987real}, $\log M(r)$ is a convex function of $\log(r)$. So for any $\alpha \in (0,1)$
\[ \log M(r) \le \frac{-\log(\alpha) \log M(\lambda r) + \log(\lambda) \log M(\alpha r)}{\log(\lambda) - \log(\alpha)}. \]
Taking $\alpha \to 0$, we have $\frac{\log M(\alpha r)}{D\log(\alpha r)} \to C$ for some constant $C$, and $\log(\alpha) \to -\infty$, so dividing the numerator and denominator by $-\log(\alpha)$ we obtain the result.
\end{proof}
For example, we obtain the following result as a special case:
\begin{lemma}\label{lem:complex2}
Suppose that $\varphi$ is an complex analytic function of finite order, i.e., for some $K > 0,$ and $r > 0$ we have for all $\xi \in \mathbb C$ that
\[ |\varphi(\xi)| \le \exp\big( (K|\xi|)^r\big). \]
Also suppose $\varphi$ has a zero of order $D$ at zero.  
Then for any $\xi$ such that $D>  r(K|\xi|)^r$, it holds that
\begin{align}
 |\varphi(\xi)| \le A\cdot \Big(\frac{er}{D}\Big)^{D/r}\cdot (K|\xi|)^{D}  
\end{align} 
\end{lemma}
\begin{proof}[Proof of \cref{lem:complex2}]
By the previous lemma, we have for any $\lambda \ge 1$ that
\begin{align}
|\varphi(\xi)| 
 &\le \exp\, \Big(-D \log \lambda + \big(\lambda K |\xi|\big)^r \Big). 
\end{align}

The exponent $- D\log  \lambda +  \big(\lambda K  |\xi| \big)^r$ attains its minima at $ \lambda_0 =  (D/r)^{1/r} /  (K \xi) $, which is feasible under the constraint $\lambda\ge 1$ when $|\xi| \le K^{-1} (D/r)^{1/r}$. 
Plugging $\lambda_0$ into the exponent yields $-(D/r) \log(D/r) + D\log(K\xi) + D/r$, equivalently  
\begin{align}
|\varphi(\xi)| &\le (er/D)^{D/r} (K|\xi|)^{D}.  
\end{align}
where $D' = D/r$. This concludes the desired upper bound. 
\end{proof}
This lemma shows that we can gain a local polynomial rate when  $\varphi$ is an entire function of finite order with a finite-order zero at $0$. 
\begin{remark}
A more general bound can be derived in the same way for any distribution with a moment generating function defined on all of $\mathbb R$.
\end{remark}
\subsection{Quantitative estimate for general quasi-analytic classes}
We will use the following standard result. One way to prove this theorem is via its probabilistic interpretation: it is a consequence of the fact that $\log |f(z)|$ is subharmonic and the density $\frac{1}{\pi} \frac{Re(z)}{\pi |z - it|^2}$ is the harmonic measure, i.e., the density of the point $it$ where Brownian motion first hits the imaginary axis when started at $z$.   
\begin{theorem}[Special case of Theorem III.G.2 of \cite{koosis1998logarithmic}]\label{thm:poisson-estimate}
Let $f$ be analytic for $\Re(z) > 0$ and continuous up to the imaginary axis.
Suppose that $\log |f(z)| = o(|z|)$ as $|z| \to \infty$ for $\Re(z) \ge 0$. Then
\[ \log |f(z)| \le \frac{1}{\pi} \int_{-\infty}^{\infty} \frac{\Re(z) \log |f(it)|}{|z - it|^2} dt\]
\end{theorem}
\paragraph{Main Result.} Now we state the quantitative Denjoy-Carleman estimate.
\begin{theorem}[Quantitative Denjoy-Carleman]\label{thm:qdc}
Suppose that $f$ vanishes up to degree $N$ at 0. Then for any $Y \in (0,K/a_{N + 1}]$,
\[ |f(x)| \le \inf_{\alpha > 0} \left(\frac{BYe^{K\alpha x}}{\alpha \pi K} \left[\exp\left(\frac{2\alpha}{\pi} \int_{1}^{1/a_{N + 1}} \frac{1}{\alpha^2 + t^2} \log(\tau(t)) dt \right) + \tau(Y/K) \right]\right). \]
\end{theorem}
\begin{proof}
Define the Laplace transform 
\[ F(s) = \int_0^{\infty} e^{-sx} f(x) dx, \]
which is analytic on the right half-plane,
and recall the inversion formula for any $\sigma > 0$ given by
\[ f(x) = \frac{1}{2\pi i} \int_{\sigma - i\infty}^{\sigma + i\infty} e^{sx} F(s) ds. \]
This yields the inequality
\[ |f(x)| \le \frac{e^{\sigma x}}{2\pi i} \int_{\sigma - i\infty}^{\sigma + i\infty} |F(s)| ds \]
and applying the high frequency and low frequency lemmas below yields for any $Y \in (0,K/a_{n + 1}]$ that
\begin{align} |f(x)| &\le \frac{Be^{\sigma x}}{2\sigma \pi i} \int_{\sigma - i Y}^{\sigma + i Y} \exp\left(\frac{1}{\pi} \int_{-K/a_{N + 1}}^{K/a_{N + 1}} \frac{\sigma}{|s - iy|^2} \log(\tau(|y|/K)) dy \right) ds +  \frac{e^{\sigma x}}{\sigma \pi} BY \tau(Y/K).
\end{align}
Equivalently,
\[ |f(x)| \le \frac{Be^{\sigma x}}{2\sigma \pi} \int_{-Y}^{Y} \exp\left(\frac{1}{\pi} \int_{-K/a_{N + 1}}^{K/a_{N + 1}} \frac{\sigma}{\sigma^2 + (w - y)^2} \log(\tau(|y|/K)) dy \right) dw +     \frac{e^{\sigma x}}{\sigma \pi} BY \tau(Y/K).\]
Making the change of variable $y = Kt, dy = K dt$ and $w = Ku, dw = K du$ yields
\begin{align} 
|f(x)| 
&\le \frac{BKe^{\sigma x}}{2\sigma \pi} \int_{-Y/K}^{Y/K} \exp\left(\frac{K}{\pi} \int_{-1/a_{N + 1}}^{1/a_{N + 1}} \frac{\sigma}{\sigma^2 + K^2(u - t)^2} \log(\tau(|t|)) dt \right) du + \frac{Ke^{\sigma x}}{\sigma \pi}   \frac{e^{\sigma x}}{\sigma \pi} BY \tau_n(Y/K)\\
&\le  \frac{BKe^{\sigma x}}{2\sigma \pi} \int_{-Y/K}^{Y/K} \exp\left(\frac{\sigma}{K\pi} \int_{-1/a_{N + 1}}^{1/a_{N + 1}} \frac{1}{\sigma^2/K^2 + (u - t)^2} \log(\tau(|t|)) dt \right) du +   \frac{e^{K \alpha x}}{K \alpha \pi} BY \tau(Y/K).
\end{align}
Letting $\sigma = \alpha K$ and taking the infimum yields 
\[ |f(x)| \le \inf_{\alpha > 0} \left[\frac{Be^{K\alpha x}}{2\alpha \pi} \int_{-Y/K}^{Y/K} \exp\left(\frac{\alpha}{\pi} \int_{-1/a_{N + 1}}^{1/a_{N + 1}} \frac{1}{\alpha^2 + (u - t)^2} \log(\tau(|t|)) dt \right) du +   \frac{e^{\sigma x}}{\sigma \pi} BY \tau(Y/K) \right].\]
We can weaken this inequality by observing that since $\tau \le 1$, the innermost integrand is always negative, and for any $u \in [0,1/a_{N + 1}]$
\begin{align} \int_{-1/a_{N + 1}}^{1/a_{N + 1}} \frac{1}{\alpha^2 + (u - t)^2} \log(\tau(|t|)) dt 
&\le \int^{1/a_{N + 1}}_{u} \frac{1}{\alpha^2 + (u - t)^2} \log(\tau(t)) dt + \int_0^{u} \frac{1}{\alpha^2 + (u - t)^2} \log(\tau(t)) dt  \\
&\le  \int^{1/a_{N + 1}}_{u} \frac{1}{\alpha^2 + t^2} \log(\tau(t)) dt +  \int_0^{u} \frac{1}{\alpha^2 + t^2} \log(\tau(t)) dt \\
&= \int_0^{1/a_{N + 1}} \frac{1}{\alpha^2 + t^2} \log(\tau(t)) dt 
\end{align}
where the comparison of the latter integrals was done by the rearrangement inequality, since $\tau$ is a decreasing function. A symmetrical argument handles the case when $u \in [-a_{N + 1}, 0]$, so we have
\[  \int_{-Y/K}^{Y/K} \exp\left(\frac{\alpha}{\pi} \int_{-1/a_{N + 1}}^{1/a_{N + 1}} \frac{1}{\alpha^2 + (u - t)^2} \log(\tau(|t|)) dt\right)du \le (2Y/K)\exp\left(\frac{2 \alpha}{\pi} \int_0^{1/a_{N + 1}} \frac{1}{\alpha^2 + t^2} \log(\tau(t)) dt \right), \]
hence 
\begin{align} |f(x)| 
&\le \inf_{\alpha > 0} \left[\frac{BYe^{K\alpha x}}{K \alpha \pi} \exp\left(\frac{2\alpha}{\pi} \int_{0}^{1/a_{N + 1}} \frac{1}{\alpha^2 + t^2} \log(\tau(t)) dt \right) +    \frac{e^{K \alpha x}}{K \alpha \pi} BY \tau(Y/K) \right] \\
&=  \inf_{\alpha > 0} \left(\frac{BYe^{K\alpha x}}{K \alpha \pi} \left[\exp\left(\frac{2\alpha}{\pi} \int_{0}^{1/a_{N + 1}} \frac{1}{\alpha^2 + t^2} \log(\tau(t)) dt \right) + \tau(Y/K) \right]\right).
\end{align}
Finally, because $\log \tau(t) = 0$ for $t \in (0,1]$, the lower limit of the integral integral can be changed to $1$ instead of $0$.
\end{proof}

\begin{lemma}
For any $z$ with $\Re(z) > 0$ and $n \le N$, we have
\[ F(z) = \frac{1}{z^n} \int_0^{\infty} e^{-zx} f^{(n)}(x) dx \]
and
\[ |F(z)| \le \frac{\|f^{(n)}\|_{\infty}}{|z|^n \Re(z)} . \]
\end{lemma}
\begin{proof}
The equality follows by applying integration by parts $n$ times: the boundary terms for integration by parts vanish at $0$ because of the assumption that $f$ and its derivatives vanish there, and the boundary terms for $\infty$ vanish because the integrand decays exponentially when $\Re(z) > 0$.

The inequality follows from Holder's inequality and the computation for $r > 0$ that
\[ \int_0^{\infty} e^{-rx} dx = 1/r. \]
\end{proof}
\begin{lemma}[High-frequency bound]
For any $Y > 0$,
\[ \int_{|y| > Y} |F(\sigma + iy)| dy \le \frac{2Y \tau_n(Y/K)}{\sigma}. \]
Also, if $Y \le K/a_{N + 1}$ then we may replace $\tau_n$ by $\tau$.
\end{lemma}
\begin{proof}
For any $0 \le n \le N$, we have
\[ \int_{|y| > Y} |F(\sigma + iy)| dy \le \int_{|y| > Y} \frac{\|f^{(n)}\|_{\infty}}{|y|^n \sigma} = \frac{2 \|f^{(n)}\|_{\infty}}{(n - 1)Y^{n - 1}\sigma} \le \frac{2BY \tau_n(Y/K)}{\sigma}. \]
\end{proof}
\begin{lemma}[Low-frequency bound]
For any $s$ with $\Re(s) > 0$,
\[ \log |F(s)| \le \log(2B/\Re(s)) + \frac{1}{2\pi} \int_{-K/a_{N + 1}}^{K/a_{N + 1}} \frac{\Re(s)}{|s - iy|^2} \log(\tau(|y|/K)) dy. \]
\end{lemma}
\begin{proof}
Let $\gamma$ be such that $0 < \gamma < \Re(s)$, to be optimized later. 
Using Theorem~\ref{thm:poisson-estimate} applied to $F$ shifted by $\gamma$, we have that 
\[ \log |F(s)| \le \frac{1}{i \pi} \int_{\gamma - i\infty}^{\gamma + i\infty} \frac{\Re(s) - \gamma}{|s - z|^2} \log |F(z)| dz = \log B +  \frac{1}{i \pi} \int_{\gamma - i\infty}^{\gamma + i\infty} \frac{\Re(s) - \gamma}{|s - z|^2} \log(|F(z)|/B) dz  \]
Using the previous lemma yields 
\[ |F(\gamma + iy)|/B \le \frac{1}{(\gamma^2 + y^2)^{n/2}} K^n M_n/\gamma = (1/\gamma) \tau_N \left(\frac{K^2}{\gamma^2 + y^2}\right)^{n/2} \]
and taking the infimum over $n \le N$ yields
\[ |F(\gamma + iy)|/B \le (1/\gamma) \tau_N\left(\sqrt{\frac{\gamma^2 + y^2}{K^2}}\right) \le (1/\gamma) \tau_N(|y|/K).  \]
Therefore
\begin{align}  
&\log |F(s)| \\ 
&\le \log (B/\gamma) + \frac{1}{\pi} \int_{\mathbb R} \frac{\Re(s) - \gamma}{|s - (\gamma + iy)|^2} \log(\tau_N(|y|/K)) dy \\
&\le \log(B/\gamma) + \frac{1}{\pi} \int_{-K/a_{N + 1}}^{K/a_{N + 1}} \frac{\Re(s) - \gamma
}{|s - (\gamma + iy)|^2} \log(\tau(|y|/K)) dy
\end{align}
where we used the equality of $\tau_N$ and $\tau$ below $1/a_{N + 1}$ and that $\tau_N \le 1$ so $\log(\tau_N) \le 0$ to drop a negative term in the last step. 

Taking $\gamma = \Re(s)/2$ and using that $0 < \gamma < \Re(s)$, we obtain
\[ \log |F(s)| \le \log(2B/\Re(s)) + \frac{1}{2\pi} \int_{-K/a_{N + 1}}^{K/a_{N + 1}} \frac{\Re(s)}{|s - iy|^2} \log(\tau(|y|/K)) dy \]
as claimed.
\end{proof}
\subsection{Recovering the sufficiency direction of Denjoy-Carleman}
Taking $\alpha = 1$ in our bound, we have
\begin{align} 
|f(x)| 
&\le \frac{BYe^{K x}}{\pi K} \left[\exp\left(\frac{2}{\pi} \int_{1}^{1/a_{N + 1}} \frac{1}{1 + t^2} \log(\tau(t)) dt \right) + \tau(Y/K) \right].
\end{align}
First taking $N \to \infty$ with $Y$ fixed, the exponential term goes to zero, and then taking $Y \to \infty$, we have $Y \tau(Y/K) \to 0$ so the right hand side as a whole goes to zero under the Denjoy-Carleman condition.
\begin{corollary}[Sufficiency direction I of Denjoy-Carleman]
If
\[ \int_{1}^{\infty} \frac{1}{1 + t^2} \log(\tau(t)) dt = -\infty \]
then $C\{M_k\}([0,\infty))$ is quasi-analytic. 
\end{corollary}
We can also express the logarithmic integral directly in terms of the $a_k$ as follows.
\begin{lemma}
\begin{align} \int_1^{1/a_{N + 1}} \frac{1}{\alpha^2 + t^2} \log(\tau(t)) dt 
&=  \sum_{k = 1}^{N} \int_{1/a_k}^{1/a_{N + 1}} \frac{-\log(a_k t)}{\alpha^2 + t^2} dt 
\end{align}
so under the assumptions of Theorem~\ref{thm:qdc}
\[ |f(x)| \le \inf_{\alpha > 0} \left(\frac{BYe^{K\alpha x}}{\alpha \pi K} \left[\exp\left(\frac{2\alpha}{\pi} \sum_{k = 1}^{N} \int_{1/a_k}^{1/a_{N + 1}} \frac{-\log(a_k t)}{\alpha^2 + t^2} dt \right) + \tau(Y/K) \right]\right). \]
\end{lemma}
\begin{proof}
Since
\[ \log \tau(t) = \sum_{k = 1}^{\infty} -\log(a_k t) 1(t > 1/a_k), \]
we can compute that
\begin{align} 
\int_1^{1/a_{N + 1}} \frac{1}{\alpha^2 + t^2} \log(\tau(t)) dt 
&= \sum_{k = 1}^{N} \int_{1/a_k}^{1/a_{N + 1}} \frac{-\log(a_k t)}{\alpha^2 + t^2} dt.
\end{align}
\end{proof}
\begin{corollary}[Sufficiency direction II of Denjoy-Carleman]
If $\sum_k a_k = \infty$ then the class $C\{M_k\}([0,\infty))$ is quasianalytic. 
\end{corollary}
\begin{proof}
Taking $\alpha = 1$ and $N \to \infty$, we either have that (a) $\lim_{N \to \infty} a_N > 0$, in which case we see $M_k = \prod_{i = 1}^k (1/a_i)$ is dominated by an exponential, so the function $f$ is entire and  the class is analytic, or (b) $1/a_N \to \infty$ so that 
\begin{align} 
\lim_{N \to \infty} \sum_{k = 1}^{N} \int_{1/a_k}^{1/a_N} \frac{-\log(a_k t)}{1 + t^2} dt \le \sum_{k = 1}^{\infty} \int_{2/a_k}^{4/a_k} \frac{-\log(a_k t)}{1 + t^2} dt &\le \sum_{k = 1}^{\infty} (2/a_k)\frac{-\log(2)}{1 + 16/a_k^2}  \\
&\le (-\log(2)/9) \sum_k a_k = -\infty.  
\end{align}
So if all derivatives of $f$ vanish at zero, then taking $N \to \infty$ first and then $Y \to \infty$, the upper bound on $|f(x)|$ converges to zero for any fixed $x$, hence $f = 0$.
\end{proof}

\subsection{Example applications}\label{sec:example-qdc}
We show how the fully general quantitative Denjoy-Carleman estimate specializes in various applications.
\begin{example}[Specialization to analytic functions]
For the class of analytic functions, we have $M_n = n!$ and $a_n = 1/n$ with
\[ \tau(r) = \prod_{k = 1}^{\lfloor r \rfloor} (k/r) = (\lfloor r \rfloor!)/r^{\lfloor r \rfloor} \le e^{-\lfloor r \rfloor} \]
using the standard inequality $n! \le (n/e)^n$.
We have for any $\alpha = \Omega(1)$, taking $Y = K/a_{N + 1}$, that
\begin{align} |f(x)|
&\le \left(\frac{B(N + 1)e^{K\alpha x}}{\alpha \pi} \left[\exp\left(\frac{\alpha}{\pi} \int_{1}^{N + 1} \frac{-\lfloor t \rfloor}{\alpha^2 + t^2} dt \right) + e^{-N - 1} \right]\right) \\
&\lesssim \left(\frac{B(N + 1)e^{K\alpha x}}{\alpha \pi} \left[\exp\left(\frac{\alpha}{2\pi}  [ \log(\alpha^2 + 1) - \log(\alpha^2 + (N + 1)^2)] \right) + e^{-N - 1} \right]\right) \\
&= \left(\frac{B(N + 1)e^{K\alpha x}}{\alpha \pi} \left[\exp\left(\frac{-\alpha}{2\pi}  \log(1 + (N^2 + 2N)/(\alpha^2 + 1)) \right) + e^{-N - 1} \right]\right) \\
\end{align}
So provided that $x \le \log(N)/2K\pi$, by taking $\alpha = N/e^{2K \pi x} \ge 1$, we find that
\[ |f(x)| \le B \exp(-\Theta(N)). \]
\end{example}
\begin{example}[Specialization to sub-Gaussian-type functions]
If $M_n = n^{n/2}$ then 
\[ a_n = M_{n - 1}/M_n = (n - 1)^{(n - 1)/2}/n^{n/2} = (1/\sqrt{n}) (1 - 1/n)^{(n - 1)/2} \le 1/\sqrt{e n} \] 
with
\[ \tau(r) = \prod_{k \ge 1 : r/\sqrt{ek} > 1} (\sqrt{e k}/r) = \sqrt{\lfloor r^2/e \rfloor !}(\sqrt{e}/r)^{\lfloor r^2/e \rfloor} \le (\lfloor r^2/e \rfloor/e)^{\lfloor r^2 /e \rfloor/2} (e/r^2)^{\lfloor r^2/e \rfloor/2} \le e^{-\lfloor r^2/e \rfloor/2} \]
so for any $\alpha = \Omega(1)$ we can compute that, taking $Y = K/a_{N + 1}$, 
\begin{align} |f(x)| 
\le \frac{Be^{K\alpha x} \sqrt{e(N + 1)}}{\alpha \pi} \left[\exp\left(\frac{\alpha}{2\pi} \int_{1}^{\sqrt{e(N + 1)}} \frac{-\lfloor t^2/e\rfloor}{\alpha^2 + t^2} dt \right) + e^{-(N + 1)/2}\right]
\end{align}
so taking $\alpha = \sqrt{N}$ and requiring that $x \ll \sqrt{N}/K$, we get $|f(x)| \le B \exp(-\Theta(N))$.
\end{example}

\begin{example}[A quasi-analytic but not analytic example]
Let $M_n = n! \prod_{k = 1}^n \log(ek)$ so that $a_n = 1/n\log(en)$. Then
\begin{align} 
\tau(r) 
&= \prod_{k \ge 1 : r/k\log(ek) > 1} (k\log(ek)/r) \\
&\le \prod_{k \ge 1 : r/\log(er) > k} (k \log(er)/r) = \lfloor r/\log(er) \rfloor! (\log(er)/r)^{\lfloor r/\log(er) \rfloor} \\
&\le \exp(-\lfloor r/\log(er) \rfloor)
\end{align}
so for $\alpha > 1$ with $\alpha = O(n)$ we have
\begin{align}
|f(x)| 
&\le \frac{BYe^{K\alpha x}}{\alpha \pi K} \left[\exp\left(\frac{2\alpha}{\pi} \int_{1}^{(N + 1)\log(e(N + 1))} \frac{-\lfloor t/\log(et) \rfloor}{\alpha^2 + t^2} dt \right) + e^{-\lfloor Y/K \rfloor/\log(e\lfloor Y/K \rfloor)} \right] \\
&\lesssim \inf_{\alpha > 0} \frac{BYe^{K\alpha x}}{\alpha \pi K} \left[\exp\left(\frac{2\alpha}{\pi} \int_{\alpha}^{(N + 1)\log(e(N + 1))} \frac{-1}{2t\log(et)} dt \right) + e^{-\lfloor Y/K \rfloor/\log(e\lfloor Y/K \rfloor)} \right] \\
&= \frac{BYe^{K\alpha x}}{\alpha \pi K} \left[\exp\left(\frac{-\alpha}{\pi} [\log(\log(t) + 1)]_{\alpha}^{(N + 1)\log(e(N + 1))} \right) + e^{-\lfloor Y/K \rfloor/\log(e\lfloor Y/K \rfloor)} \right] \\
&\lesssim \frac{BYe^{K\alpha x}}{\alpha \pi K} \left[\exp\left(\frac{-\alpha}{\pi} \log\frac{\log(N + 1)}{\log(\alpha + 1)} \right) + e^{-\lfloor Y/K \rfloor/\log(e\lfloor Y/K \rfloor)} \right] \\
\end{align}
Let $\delta \in (0,1)$ and consider taking $\alpha = \exp(\log(N + 1)^{\delta}) - 1$, so $\log \frac{\log(N + 1)}{\log(\alpha + 1)} = \log(1/\delta) > 0$. Taking $\log(1/\delta) = 2K\pi x$ so $\delta = \delta(x) = e^{- 2K\pi x}$, we find that
\[ |f(x)| \le B\exp(-\Theta(\alpha K\pi x))) = Be^{-\Theta(x\exp(\log(N)^{\delta(x)}))}. \]
Note that this goes to zero for fixed $x$ as $N \to \infty$ faster than any inverse polynomial. 
\end{example}

\subsection{Discussion: comparison of complex-variable and real-variable approaches}
There is an interesting real-variable proof of the Denjoy-Carleman theorem which requires only basic facts about integration and differentiation: see \cite{cohen1968simple,hörmander1983analysis} for an exposition and history. This method naturally gives quantitative bounds, which we state below.
\begin{lemma}[Lemma 1.3.6 and Equation (1.3.17) in Hormander \cite{hörmander1983analysis}]\label{lem:dc}
Suppose that $f \in \mathcal{C}\{M_k\}([0,b])$, and let $B,K \ge 0$ be the corresponding values such that
\[ \sup_{x\in[0,b]}\bigl|f^{(k)}(x)\bigr|\le B K^{k}\,M_k
\quad\forall\,k\ge0. \]
Let $m \ge 1$ be such that 
\[ f^{(k)}(0) = 0 \]
for all $1 \le k \le m$. Then for any $x \le (a_1 + \cdots + a_m)/8K$, 
the following inequality holds:
\[ |f(x)| \le [8B K/(a_1 + \cdots + a_m)] |x|. \]
\end{lemma}
Note that for fixed $x$, this estimate becomes stronger as $m$ grows, and taking $m \to \infty$ suffices to prove the Denjoy-Carleman theorem \cite{hörmander1983analysis}. 
However, while this result is strong enough to recover the Denjoy-Carleman theorem, the quantitative estimates we obtain via complex analysis are much tighter.  

\paragraph{Approximation bound from real-variable approach.} 
We require Carleman's condition that for all $k \ge 0$, the moments satisfy 
\[ \|x^k\|_{\mu} \le B K^k M_k \]
for $B > 0$ and a log-convex sequence $\{M_k\}$ satisfying $M_0 = 1$. Define $a_k = M_{k - 1}/M_k$ as usual. 

\begin{theorem}[Pure real-variable bound, for comparison]
Let $D \ge 1$ such that
\[ \Omega \le (a_1 + \cdots + a_D)/8K. \] 
Then
\[ \|r_D\|_{\mu}  \le \frac{8B\|\hat f\|_{L^1([-\Omega,\Omega])}}{2\pi} \Omega K/(a_1 + \cdots + a_D) + \frac{1}{2\pi} \int_{\mathbb R \setminus [-\Omega,\Omega]} |\hat f(\xi)| d\xi.  \]
\end{theorem}
\begin{proof}

For any $\xi \in \mathbb R$ we have
\[ \left|\frac{d^m}{d\xi^m} \varphi(\xi)\right| \le \|r\|_{\mu} \|x^m\|_{\mu} \le \|r\|_{\mu} B K^m M_m. \]
Applying the quantitative estimate for Denjoy-Carleman classes (Lemma~\ref{lem:dc}), we conclude that
provided $\Omega \le (a_1 + \cdots + a_D)/8K$, for all $\xi \in [-\Omega, \Omega]$,
\begin{equation}\label{eqn:dj-applied}
|\varphi(\xi)| \le 8B\|r\|_{\mu} |\xi| K/(a_1 + \cdots + a_D). 
\end{equation}
Therefore the result follows by Lemma~\ref{lem:fourier-apx-bound}.
\end{proof}

\paragraph{Pessimism of real-variable approach.} 
So for fixed $\Omega$, this result proves an upper bound which converges to zero as 
\[ \frac{1}{\sum_{j = 1}^D a_j}. \]
However, even in simple examples like sub-Gaussian and sub-exponential distributions, this is either exponentially or doubly-exponentially worse than what complex analysis gives:
\begin{example}Suppose that the distribution is sub-Gaussian. Then up to constants, $M_k = k^{k/2}$ so $a_k = (k - 1)^{(k - 1)/2}/k^{k/2} = (1/k)^{1/2} (1 - 1/k)^{(k - 1)/2} \approx (1/\sqrt{ek})$. So
\[ \sum_{j = 1}^D a_j \approx \int_1^D 1/\sqrt{ex} dx = \Theta(\sqrt{D}). \]
\end{example}
\begin{example} Suppose that the distribution is only sub-Exponential. Then up to constants, $M_k = k^k$ so $a_k = (k - 1)^{(k - 1)}/k^k \approx 1/ek$. So $\sum_{j = 1}^D a_j = \Theta(\log D)$.  
\end{example}
This is very pessimistic, because in the sub-Gaussian and even in sub-exponential  case the error actually shrinks exponentially with $D$. It is an interesting open question if there is a way to close the gap between the elementary and complex-analytic estimates. 

\section{\texorpdfstring{Constructing $L^2$-approximating polynomials in one dimension}{Constructing L2-approximating polynomials in one dimension}}\label{sec:one-dim-approx}
The following notation is used throughout this section.
Let $\mu(x)$ be a probability density on $\mathbb R$ such that all moments exist. For any function $g : \mathbb R \to \mathbb C$, define the corresponding $L^2(\mu)$ norm by 
\[ \|g\|_{\mu}^2 = \int \mu(x) |g(x)|^2 dg \]
and corresponding $L^2(\mu)$ inner product $\langle f, g \rangle_{\mu} = \int f(x) \overline{g(x)} dx$. 
Let $f$ be a function which is in $L^2(\mu)$ but also has a Fourier transform $\hat f$ in $L^1(\mathbb R)$, so that
\[ f(x) = \frac{1}{2\pi} \int \hat{f}(\xi) e^{i\xi x} d\xi \]
Let $g_D$ be the $L^2(\mu)$ projection of $f$ onto the space of polynomials of degree at most $D$, and let $r_D = f - g_D \in L^2(\mu)$ be the residual/orthogonal projection. 
\begin{remark}[Extensions in Section~\ref{sec:Rd-analytic-lemmas}]
Even in the one-dimensional case, our results apply to a broader class of functions $f$ than we are considering above. In the interest of simplicity, we leave the more general and powerful results to Section~\ref{sec:Rd-analytic-lemmas}, where we also handle the higher-dimensional case.
\end{remark}
\subsection{Fourier interpretation of approximation error}
We define the \emph{Fourier transform of the residual} in the distributional sense as
\begin{align}
  \varphi(\xi) = \mathcal F[\mu r_D] = \int \mu(x) r_D(x) e^{-ix \xi} dx.  \label{eqn:fourier-residual}
\end{align}
\begin{remark}
In the special case where $\mu$ is the density of $N(0,1)$, the function $\varphi$ is, up to scaling factors, the Bargmann transform of $\widehat r_D$. 
\end{remark}
As a characterization of the residual function of the $L^2(\mu)$-projection onto polynomials of degree at most $D$, the function $\varphi$ inherits some nice orthogonality properties, as stated in the following lemma. 
\begin{lemma}\label{lem:fourier-reinterpret}
The function $\varphi$ defined in \eqref{eqn:fourier-residual} satisfies that $\varphi^{(k)}(0) = 0$ for all $k \le D$.
Furthermore, for any $\xi \in \mathbb R$ and $m \ge 0$, we have 
\[ \left|\frac{d^m}{d\xi^m} \varphi(\xi)\right| \le \|r_D\|_{\mu} \|x^m\|_{\mu}. \]
In particular, we have that $\sup_{\xi} |\varphi(\xi)| \le  \|r_D\|_\mu$ by setting $m=0$ above.
\end{lemma}

\begin{proof}[Proof of \cref{lem:fourier-reinterpret}]
By the orthogonality of the $L^2(\mu)$ projection, we have that for any integer $0\le k\le D$ that $\langle r_D, x^k \rangle_{\mu} = 0 $. 
Recall the differentiation identity
\begin{equation} 
  \frac{d^m}{d\xi^m} \varphi(\xi) = \int (-ix)^m \mu(x) r_D(x) e^{-i\xi x} dx. \label{eq:fourier-diff}
\end{equation}
Evaluating \cref{eq:fourier-diff} at $\xi = 0$, we find that $d^m \varphi(0)/d\xi^m = (-i)^m\, \langle r_D, x^m \rangle_{\mu} = 0$ for all $0 \le m \le D$, as desired. 
Also, by using Cauchy-Schwarz we find
\[ \left|\frac{d^m}{d\xi^m} \varphi(\xi)\right| = |\langle r_De^{-i\xi x}, x^m \rangle_{\mu}| \le \|r_D\|_{\mu} \|x^m\|_{\mu},\]
which proves the second claim. 
\end{proof} 
\cref{lem:fourier-reinterpret} shows that the function $\varphi$ has a zero of order $D$ at $0$.  
Combining this with various instances of distributions, we can derive tight bounds on $|\varphi|$ that are useful to control the approximation error, especially if $\hat f$ has rapidly decaying tails. 
Next lemma provides an accessible tool to bound $\|r_D\|_\mu$ for functions with different patterns of Fourier transform. 
\begin{lemma}\label{lem:fourier-apx-bound}
Suppose that $\|r_D\|_\mu>0$. 
It holds for any $\Omega \ge 0$ that 
\[ \|r_D\|_{\mu}  \le \frac{\|\hat f\|_{L^1([-\Omega,\Omega])}}{2\pi {\|r_D\|_{\mu}}} \sup_{\xi \in [-\Omega,\Omega]} |\varphi(\xi)| + \frac{1}{2\pi} \int_{\mathbb R \setminus [-\Omega,\Omega]} |\hat f(\xi)| d\xi.  \]
\end{lemma}
\begin{proof}
By Plancherel's theorem and the Cauchy-Schwarz inequality,
\[ \|r_D\|_{\mu}^2 = \langle r_D, f \rangle_{\mu} = \frac{1}{2\pi} \int \varphi(\xi) \overline{\hat f(\xi)} d\xi \le \frac{1}{2\pi}\|\hat f\|_{L^1([-\Omega,\Omega])} \sup_{\xi \in [-\Omega,\Omega]} |\varphi(\xi)| + \frac{\|r_D\|_{\mu}}{2\pi} \int_{\mathbb R \setminus [-\Omega,\Omega]} |\hat f(\xi)| d\xi \]
where in the last inequality we used the previous lemma to bound $\sup_{\xi \in \mathbb R} |\varphi(\xi)|$. Dividing through by $\|r_D\|_{\mu} < \infty$ proves the result.
\end{proof}
In the following, we leverage these lemmas to derive tight controls of $\|r_D\|_\mu$ for Gaussian analytic functions under sub-exponential and strictly sub-exponential distributions.  

\subsection{Approximation via complex analysis} \label{sec:complex-approx}
\subsubsection{General bound for strictly sub-exponential distributions}
Recall by \cref{lem:fourier-reinterpret} that, for the function $\varphi$ has a zero of order $D$ at $0$, we can gain tighter local control of $|\varphi(\xi)|$, when $\varphi$ is an entire function of finite order, as the following lemma shows.  
This lemma shows that we can gain local polynomial rate when  $\varphi$ is an entire function of finite order with a finite-order zero at $0$. 
To utilize this result, we assert that the function $\varphi$ defined in \eqref{eqn:fourier-residual} satisfies the assumptions of \cref{lem:complex2} when the cumulant generating function of $\mu$ is at most growing polynomially.  
\begin{theorem}\label{thm:pw-strict}
Let $r \ge 1$, $K > 0$ and $A > 0$ be constants.  
Suppose that the moment generating function of $\mu$ satisfies that
\[
  \|e^{x\xi}\|_{\mu} \;\le\; A \exp\bigl((|\xi|\,K)^r\bigr)
  \quad\text{for all }\xi\in\R.
\]
Then for any $\Omega\ge0$ and $D\ge r\,( K\Omega)^r$, it holds that 
\begin{align}
  \|r_D\|_{\mu}
  &\le \frac{A\,\|\hat f\|_{L^1([-\Omega,\Omega])}}{2\pi}
       \Bigl(\frac{r\,e}{D}\Bigr)^{D/r} \cdot (K\Omega)^{D}
    + \frac{1}{2\pi}\int_{\R\setminus[-\Omega,\Omega]}
       \bigl|\hat f(\xi)\bigr|\,d\xi.
\end{align}
\end{theorem}

\begin{proof}[Proof of \cref{thm:pw-strict}]
For any $\xi\in\C$, by Cauchy–Schwarz inequality, we have
\begin{align}
  |\varphi(\xi)|
  &= \bigl|\langle r_D,\,e^{i x\xi}\rangle_{\mu}\bigr|
  \;\le\;
    \|r_D\|_{\mu}\,\|e^{-x\,\Im(\xi)}\|_{\mu}
  \;\le\;
    A\,\|r_D\|_{\mu}\,\exp\bigl((|\xi|\,K)^r\bigr). 
\end{align}
Therefore, $\varphi$ extends to an entire function on $\mathbb{C}$ with a finite order. 
Using \cref{lem:complex2}, we can bound $|\varphi(\xi)|$ for $|\xi | \le K^{-1} (D/r)^{1/r}$ as
\begin{align}
  |\varphi(\xi)| &\le A\norm{r_D}_{\mu} \, \Big(\frac{er}{D}\Big)^{D/r} \cdot (K|\xi|)^{D}
\end{align}
The claimed inequality then follows by Lemma~\ref{lem:fourier-apx-bound}. 
\end{proof}
\begin{remark}
By the same argument, a more general bound holds for any distribution with a generating function defined on all of $\mathbb R$.
\end{remark}

This shows we can potentially get rates of $\log(1/\epsilon)/\log\log(1/\epsilon)$ in this setting instead of just $\log(1/\epsilon)$. This is in fact sharp even for the Gaussian case --- we will return to this point in the later section on Paley-Wiener classes. 


  

\subsubsection{General bound for sub-exponential distributions}
The previous result is based on the global complex analyticity of the function $\varphi$ defined in \eqref{eqn:fourier-residual}. 
However, the characteristic function of the sub-exponential distribution $\mu$ exists only in a strip of the complex plane. 
To this end, we state the following result, which is a counterpart of \cref{thm:pw-strict} for strip-analytic characteristic functions, i.e., sub-exponential distributions.

\begin{theorem}\label{thm:pw-subexp}
Suppose that the distribution is sub-exponential, in the sense that $\mathbb E_\mu[e^{2|X|/K}] \le e^2$ for some sub-exponential constant $K > 0$. 
For any $\Omega \ge 0$,
\[ \|r_D\|_{\mu}  \le \frac{e\|\hat f\|_{L^1([-\Omega,\Omega])}}{2\pi} \tanh(K\pi \Omega/4)^D + \frac{1}{2\pi} \int_{\mathbb R \setminus [-\Omega,\Omega]} |\hat f(\xi)| d\xi. \label{eqn:pw-subexp-formula}  \]
\end{theorem}
\begin{proof}[Proof of \cref{thm:pw-subexp}]  
Note that
\begin{align}
\varphi(\xi) = \langle r, e^{ix \xi} \rangle_{\mu} \le \|r\|_{\mu} \|e^{|\Im(\xi)| \cdot |x|}\|_{\mu} \le \|r\|_{\mu} e, 
\end{align}
when $|\Im(\xi)| < K^{-1}$. 
Define $\rho(z) = \|r\|_{\mu}^{-1}\varphi\big(4z/(\pi\,K) \big)$, which is analytic in the strip $|\Im(z)| < \pi/4$ and bounded by $e$. 
Applying Lemma~\ref{lem:complex}, we find that $|\rho(z)| \le e|\tanh(\xi)|^D$ so
\[ |\varphi(\xi)| = \|r\|_{\mu} |\rho(\frac{K \pi}{4}\xi)| \le e\|r\|_{\mu} |\tanh(\frac{K \pi}{4} \xi)|^D. \]
The result follows from Lemma~\ref{lem:fourier-apx-bound}.
\end{proof}

\subsubsection{General bound under Carleman's condition}
The same argument from the sub-exponential case can also be applied under Carleman's condition. The only change is that we need to replace the use of Lemma~\ref{lem:complex} by the more general Theorem~\ref{thm:qdc}. In the more general formula, the $\tanh$ term on the right hand side of \eqref{eqn:pw-subexp-formula} is replaced by a more general expression in terms of logarithmic integrals coming from Theorem~\ref{thm:qdc}.
The formula in full generality is slightly complex, but up to constants it recovers the previous result for sub-exponential distributions (due to the discussion in Section~\ref{sec:example-qdc}). 
\section{Further preliminaries from functional analysis}\label{sec:prelim2}
In this section, we recall some preliminaries from functional analysis and distribution theory which are necessary to state our final results in the appropriate level of generality. 
See references \cite{reed1980methods,reed1975ii,hörmander1983analysis,rudin2017fourier,rudin1987real} for more detailed background. 
Throughout, \(\mathcal{B}(\mathbb{R}^d)\) denotes the Borel \(\sigma\)-algebra on \(\mathbb{R}^d\).




\subsection{Complex Radon measures on \texorpdfstring{$\mathbb{R}^d$}{Rd}}

\begin{definition}[Radon measure]
A (real) \emph{Radon measure} \(\mu\) on \(\mathbb{R}^d\) is a countably additive set function
\[
  \mu: \mathcal{B}(\mathbb{R}^d) \;\longrightarrow\; \mathbb{R}
\]
satisfying the following two regularity properties:
\begin{enumerate}
  \item \textbf{Local finiteness.} For every compact set \(K \subset \mathbb{R}^d\), one has \(\lvert \mu\rvert(K) < \infty\).
  \item \textbf{Regularity.}
  \begin{itemize}
    \item \emph{Inner regularity:} For every Borel set \(E \subset \mathbb{R}^d\),
    \[
      \mu(E) \;=\; \sup\,\{\,\mu(K)\colon K \subset E,\; K \text{ compact}\}.
    \]
    \item \emph{Outer regularity:} For every Borel set \(E \subset \mathbb{R}^d\),
    \[
      \mu(E) \;=\; \inf\,\{\,\mu(U)\colon U \supset E,\; U \text{ open}\}.
    \]
  \end{itemize}
\end{enumerate}
\end{definition}

\begin{definition}[Complex Radon Measure]
A \emph{complex Radon measure} \(\mu\) on \(\mathbb{R}^d\) is a countably additive function
\[
  \mu: \mathcal{B}(\mathbb{R}^d) \;\longrightarrow\; \mathbb{C}
\]
such that both the real part \(\Re(\mu)\) and the imaginary part \(\Im(\mu)\) are (real) Radon measures. Equivalently, if one writes
\[
  \mu \;=\; \mu_{1} \;+\; i\,\mu_{2}, 
  \quad 
  \mu_{1},\,\mu_{2}\colon \mathcal{B}(\mathbb{R}^d)\to\mathbb{R},
\]
then \(\mu_{1}\) and \(\mu_{2}\) are each (real) Radon measures in the sense of the previous definition.
\end{definition}

\begin{definition}[Total Variation of a Complex Radon Measure]
Let \(\mu\) be a complex Radon measure on \(\mathbb{R}^d\). Its \emph{total variation} \(\lvert \mu\rvert\) is the unique (real) Radon measure characterized by
\[
  \lvert \mu\rvert(E)
  \;=\;
  \sup
  \Bigl\{
    \sum_{j=1}^\infty \bigl\lvert \mu(E_j)\bigr\rvert
    \,\colon\,
    \{E_j\}_{j=1}^\infty \text{ is a Borel partition of } E
  \Bigr\},
  \quad
  E \in \mathcal{B}(\mathbb{R}^d).
\]
\end{definition}

\begin{definition}[Total Variation Norm]
Given a complex Radon measure \(\mu\) on \(\mathbb{R}^d\), its \emph{total variation norm} is defined by
\[
  \bigl\|\mu\bigr\|
  \;=\;
  \lvert \mu\rvert\!\bigl(\mathbb{R}^d\bigr),
\]
i.e.\ the total variation of \(\mu\) evaluated on the whole space. In particular, \(\|\mu\|\) is finite if and only if \(\lvert \mu\rvert\) is a finite (real) measure on \(\mathbb{R}^d\).
\end{definition}
\paragraph{Riesz representation theorem.} 
Let \(C_{0}(\mathbb{R}^d)\) denote the space of continuous, complex‐valued functions on \(\mathbb{R}^d\) that vanish at infinity, equipped with the supremum norm \(\|\cdot\|_{\infty}\). 

Let \(\mu\) be a complex Radon measure on \(\mathbb{R}^d\) with total variation norm \(\|\mu\|\).  Define
\[
  T_\mu(f) \;=\; \int_{\mathbb{R}^d} f(x)\,d\mu(x),
  \qquad
  f \in C_0(\mathbb{R}^d).
\]
Then \(T_\mu\) is a bounded linear functional on \(C_0(\mathbb{R}^d)\) endowed with the supremum norm \(\|f\|_\infty = \sup_{x\in\mathbb{R}^d}|f(x)|\), and 
\[
  \|T_\mu\|_{(C_0)^*} \;=\; \|\mu\|.
\]
In particular, for every \(\varphi \in \mathcal{S}(\mathbb{R}^d)\subset C_0(\mathbb{R}^d)\),
\[
  \bigl|T_\mu(\varphi)\bigr|
  \;=\;
  \Bigl|\int_{\mathbb{R}^d} \varphi(x)\,d\mu(x)\Bigr|
  \;\le\;
  \|\varphi\|_\infty\,\|\mu\|,
\]
showing that \(T_\mu\in\mathcal{S}'(\mathbb{R}^d)\).

\begin{theorem}[Riesz–Markov–Kakutani on \(\mathbb{R}^d\) \cite{rudin1987real}]
Every bounded linear functional 
\[
  L: C_{0}(\mathbb{R}^d) \;\longrightarrow\; \mathbb{C}
\]
is of the form
\[
  L(f) \;=\; \int_{\mathbb{R}^d} f(x)\,d\mu(x),
  \qquad f \in C_{0}(\mathbb{R}^d),
\]
for a unique complex Radon measure \(\mu\) on \(\mathbb{R}^d\).  Moreover,
\[
  \|L\|_{(C_{0}(\mathbb{R}^d))^*} 
  \;=\; \|\mu\|
  \;=\; \lvert \mu\rvert(\mathbb{R}^d),
\]
where \(\|\mu\|\) is the total variation of \(\mu\).  Conversely, each complex Radon measure \(\mu\) on \(\mathbb{R}^d\) of finite total variation defines a bounded linear functional on \(C_{0}(\mathbb{R}^d)\) via integration.
\end{theorem}

\subsection{Schwartz space and tempered distributions}\label{subsec:tempered-distributions}

\begin{definition}[Schwartz space \(\mathcal{S}(\mathbb{R}^d)\)]
The \emph{Schwartz space} \(\mathcal{S}(\mathbb{R}^d)\) is the space of all infinitely differentiable functions 
\[
  \varphi: \mathbb{R}^d \;\longrightarrow\; \mathbb{C}
\]
such that for every pair of multi-indices \(\alpha,\beta\in\mathbb{N}_0^d\), the seminorm
\[
  p_{\alpha,\beta}(\varphi)
  \;=\;
  \sup_{x\in\mathbb{R}^d}
  \bigl\lvert x^\alpha\,D^\beta \varphi(x)\bigr\rvert
\]
is finite. Here \(x^\alpha = x_1^{\alpha_1}\cdots x_d^{\alpha_d}\) and 
\[
  D^\beta = \frac{\partial^{\beta_1}}{\partial x_1^{\beta_1}} \cdots \frac{\partial^{\beta_d}}{\partial x_d^{\beta_d}}.
\]
Equivalently, 
\[
  \mathcal{S}(\mathbb{R}^d)
  \;=\;
  \Bigl\{
    \varphi\in C^\infty(\mathbb{R}^d)\;:\;\forall\,\alpha,\beta\in\mathbb{N}_0^d,\;
    \sup_{x\in\mathbb{R}^d} \lvert x^\alpha D^\beta \varphi(x)\rvert < \infty
  \Bigr\}.
\]
\end{definition}
We define the Fourier transform $\cF:\mathcal{S}(\R^d)\to C^\infty(\R^d)$ by
\[
\widehat{f}(\xi)
:=
\int_{\R^d} f(x)\,e^{-i\,x\cdot\xi}\,dx,
\]
and its inverse by
\[
\check{g}(x)
:=
\frac{1}{(2\pi)^d}
\int_{\R^d} g(\xi)\,e^{i\,x\cdot\xi}\,d\xi.
\]

\begin{theorem}
The Fourier transform $\cF$ is a bijection
\[
\cF:\mathcal{S}(\R^d)\;\longrightarrow\;\mathcal{S}(\R^d),
\]
and its inverse is given by the above formula for $\check{(\cdot)}$.
\end{theorem}

\begin{proof}
Let $f\in\mathcal{S}(\R^d)$.  We show $\widehat{f}\in\mathcal{S}(\R^d)$ by checking the seminorms.  Using the standard identities
\[
D_\xi^\beta\widehat{f}(\xi)
=
(i)^{|\beta|}\,\widehat{x^\beta f}(\xi),
\qquad
\xi^\alpha\widehat{f}(\xi)
=
\widehat{D^\alpha f}(\xi),
\]
valid for all multi‐indices $\alpha,\beta$, we obtain
\[
\sup_{\xi\in\R^d}
\bigl|\xi^\alpha D_\xi^\beta\widehat{f}(\xi)\bigr|
=
\sup_{\xi\in\R^d}
\bigl|\widehat{D^\alpha(x^\beta f)}(\xi)\bigr|
\;\le\;
\int_{\R^d}
\bigl|D^\alpha(x^\beta f)(x)\bigr|
\,dx
<\infty,
\]
since $D^\alpha(x^\beta f)\in\mathcal{S}(\R^d)\subset L^1(\R^d)$.  Hence every seminorm $p_{\alpha,\beta}(\widehat{f})$ is finite.

Surjectivity and the fact that the inverse transform $\check{(\cdot)}$ also maps $\mathcal{S}(\R^d)$ into itself follow by the same type of computation, or by observing that $\check{\bigl(\widehat{f}\bigr)}=f$ for all $f\in\mathcal{S}(\R^d)$.  This completes the proof.
\end{proof}

\begin{definition}[Tempered distribution]
A \emph{tempered distribution} on \(\mathbb{R}^d\) is a continuous linear functional
\[
  T: \mathcal{S}(\mathbb{R}^d) \;\longrightarrow\; \mathbb{C}.
\]
The space of all tempered distributions is denoted by \(\mathcal{S}'(\mathbb{R}^d)\). That is,
\[
  \mathcal{S}'(\mathbb{R}^d)
  \;=\;
  \bigl\{\,T: \mathcal{S}(\mathbb{R}^d)\to\mathbb{C} \text{ linear} 
  \;\bigm|\; T \text{ is continuous w.r.t.\ the Fréchet topology on } \mathcal{S}(\mathbb{R}^d)\bigr\}.
\]
Continuity here means that whenever a sequence \(\{\varphi_j\}\subset\mathcal{S}(\mathbb{R}^d)\) converges to zero (in the sense that each seminorm \(p_{\alpha,\beta}(\varphi_j)\to 0\)), we have \(T(\varphi_j)\to 0\). 
\end{definition}

\noindent
Typical examples include:
\begin{itemize}
  \item Every (complex) Radon measure \(\mu\) of finite total variation defines a tempered distribution via
  \[
    T_\mu(\varphi)
    \;=\;
    \int_{\mathbb{R}^d} \varphi(x)\,d\mu(x),
    \qquad
    \varphi\in\mathcal{S}(\mathbb{R}^d).
  \]
  \item Every sufficiently regular function \(f\in L^1_{\mathrm{loc}}(\mathbb{R}^d)\) that grows at most polynomially (e.g.\ \(f(x) = (1+\lvert x\rvert)^m\) for some \(m\)) defines a tempered distribution by
  \[
    T_f(\varphi)
    \;=\;
    \int_{\mathbb{R}^d} f(x)\,\varphi(x)\,dx,
    \qquad
    \varphi\in\mathcal{S}(\mathbb{R}^d).
  \]
\end{itemize}

\medskip
\noindent
\begin{definition}[Order of a tempered distribution]
Let \(T\in\mathcal{S}'(\mathbb{R}^d)\) be a tempered distribution (i.e.\ a continuous linear functional on \(\mathcal{S}(\mathbb{R}^d)\)).  We say that \(T\) is of \emph{order} at most \(m\) if there exists a constant \(C>0\) such that
\[
  \bigl|T \varphi\bigr|
  \;\le\;
  C \,\sum_{\substack{|\alpha|\le m \\|\beta|\le m}}
  \;\sup_{x\in \mathbb{R}^d}
  \bigl|\,x^{\alpha}\,\partial^{\beta}\varphi(x)\bigr|
  \quad
  \text{for all }\varphi\in \mathcal{S}(\mathbb{R}^d).
\]
The \emph{order} of \(T\) is the least integer \(m\) for which such an estimate holds.
\end{definition}
With these definitions in place, one often identifies a complex Radon measure \(\mu\) of finite total variation with the tempered distribution \(T_\mu\). Conversely, every tempered distribution is (by definition) a continuous extension of such "nice" objects to all of \(\mathcal{S}(\mathbb{R}^d)\). 
Note that tempered distributions of order 0 are bounded linear functionals, so they correspond to complex Radon measures.  


\paragraph{Differentiation and multiplication.} \label{par:differentiation-multiplication}

Suppose that $T \in \mathcal{S}'(\RR^d)$ is a tempered distribution, and $u\in C^\infty(\RR^d)$ is a smooth function with polynomial growth, i.e., there exists $C,N>0$ such that $|u(x)|\le C(1+|x|)^N$ for all $x\in \RR^d$.  
In this way, for any test function $\varphi \in \mathcal{S}(\RR^d)$, the product $u\varphi \in \mathcal{S}(\RR^d)$.
Hence, we can define the multiplication $uT$ as 
\begin{align}
  \dotp{uT}{\varphi} := \dotp{T}{u\varphi}, \quad \forall \varphi \in \mathcal{S}(\RR^d).   
\end{align}
For any multi-index $\alpha\in \NN^d$, we can also define the derivative $\partial^\alpha T$ as 
\begin{align}
  \dotp{\partial^\alpha T}{\varphi} := (-1)^{|\alpha|} \dotp{T}{\partial^\alpha \varphi}, \quad \forall \varphi \in \mathcal{S}(\RR^d).  
\end{align}

\paragraph{Support of a tempered distribution. }  

Let \(T \in \mathcal{S}'(\mathbb{R}^d)\). An open set \(U \subset \mathbb{R}^d\) is said to be a \emph{zero region} for \(T\) if
\[
  \bigl\langle T, \varphi \bigr\rangle
  \;=\;
  0,
  \quad
  \forall\,\varphi \in \mathcal{S}(\mathbb{R}^d)
  \text{ with } \operatorname{supp}\varphi \subset U.
\]
The \emph{support} of \(T\), denoted \(\operatorname{supp}T\), is the complement in \(\mathbb{R}^d\) of the largest open zero region.  Equivalently, \(\operatorname{supp}T\) is the smallest closed set \(K \subset \mathbb{R}^d\) such that
\[
  \bigl\langle T, \varphi \bigr\rangle
  \;=\;
  0,
  \quad
  \forall\,\varphi \in \mathcal{S}(\mathbb{R}^d)
  \text{ with } \operatorname{supp}\varphi \subset \mathbb{R}^d \setminus K.
\]

\paragraph{Structures of compactly supported tempered distributions.}   

In many cases, we will encounter functions whose Fourier transform is compactly supported. 
Surprisingly, the finiteness of the order of a tempered distribution connects with the compactness of its support. 

\begin{theorem}[Theorem 24.3, 24.4 in \cite{treves2016topological}]\label{thm:structure-compact-tempered-distribution}

  Let $T$ be a tempered distribution on $\RR^d$ with compact support. 
Then $T$ is of finite order $m<\infty$, and there exists a family of Radon measures $\{\mu_p\}_{p\in \NN^d, |p|\le m}$ with $\supp(\rho_p) \subset \supp(T)$ such that 
\begin{align}
  T = \sum_{ |\alpha|\le m}  d\mu_p(x)\, \partial^\alpha.
\end{align} 
In other words, for any test function $\varphi \in \mathcal{S}(\RR^d)$, we have 
\begin{align}
  \dotp{T}{\varphi} = \sum_{ |\alpha|\le m}  \int_{\RR^d} \partial^\alpha \varphi(x) d\mu_p(x).  
\end{align}
\end{theorem}

\paragraph{Fourier transform of tempered distributions.} 

The Fourier transform extends continuously to the space of tempered distributions \(\mathcal{S}'(\mathbb{R}^d)\).  Concretely, if \(T \in \mathcal{S}'(\mathbb{R}^d)\), its Fourier transform \(\widehat{T} \in \mathcal{S}'(\mathbb{R}^d)\) is defined by
\[
  \bigl\langle \widehat{T}, \varphi \bigr\rangle
  \;=\;
  \bigl\langle T, \widehat{\varphi} \bigr\rangle,
  \qquad
  \forall\,\varphi \in \mathcal{S}(\mathbb{R}^d),
\]
where
\[
  \widehat{\varphi}(\xi)
  \;=\;
  \int_{\mathbb{R}^d} \varphi(x)\,e^{-\,i\,x\cdot\xi}\,dx.
\]
In particular, \(\widehat{T}\) is again a tempered distribution.

\paragraph{Functions with compactly supported Fourier transforms.} The following theorem provides an equivalence between compactly supported tempered distributions and Fourier transforms with polynomial growth rate. 

\begin{theorem}[Paley–Wiener–Schwartz, Theorem 7.3.1 of \cite{hörmander1983analysis}]\label{thm:paley-wiener-schwartz}
Let \(T \in \mathcal{S}'(\mathbb{R}^d)\) by a tempered distribution of order $N$. Then \(\mathrm{supp}\,T \subseteq \{x \in \mathbb{R}^d : \lvert x\rvert \le R\}\) for some \(R \ge 0\) if and only if its Fourier transform \(\widehat{T}\) extends to an entire function on \(\mathbb{C}^d\) satisfying the following growth condition: there exists a constant \(C > 0\) such that
\[
  \bigl|\widehat{T}(z)\bigr|
  \;\le\;
  C\,\bigl(1 + \lvert z\rvert\bigr)^{N}\,
  \exp\bigl(R\,\lvert \Im z\rvert\bigr),
  \quad
  \forall\,z \in \mathbb{C}^d.
\]
Conversely, any entire function on \(\mathbb{C}^d\) satisfying such a bound is the Fourier transform of a unique tempered distribution of order $N$ supported in the closed ball \(\{x : \lvert x\rvert \le R\}\).
\end{theorem}

\subsection{Fourier algebra and pseudomeasures}
\subsubsection{Fourier algebra and pseudomeasures on \texorpdfstring{$\R^d$}{Rd}}
See Chapter VI.1 and Chapter VI.4 of Katznelson's book \cite{katznelson2004introduction} for a more detailed discussion of the Fourier algebra and of pseudomeasures, and also Chapter 4.2 of Larsen's book \cite{larsen2012introduction} for an explicit discussion of their connection via duality.
\begin{definition}    
Let
\[
A(\R^d):=\mathcal F(L^1(\R^d)) \subset L^{\infty}(\R^d)
\quad\text{with}\quad \|a\|_{A}:=\|\mathcal F^{-1}a\|_{L^1},
\]
and
\[
PM(\R^d):=\mathcal F(L^\infty(\R^d)) \supset L^1(\R^d)
\quad\text{with}\quad 
\|u\|_{PM}:=\inf\{\|g\|_{L^\infty}: u=\widehat g\}.
\]
Then $A(\R^d)$ is a Banach algebra under pointwise multiplication called the Fourier algebra, and $PM(\R^d)\subset\mathcal S'(\R^d)$ is the space of pseudomeasures.
\end{definition}

\textbf{Duality / ``Plancherel pairing'' for $(A,PM)$.}
Recall that $L^{\infty}$ is isomorphic to the dual of $L_1$ (whereas $(L^{\infty})' \subsetneq L_1$). 
Transporting the $L^\infty\!-\!L^1$ pairing by $\mathcal F$ yields an isomorphism
\[
(A(\R^d))'\ \cong\ PM(\R^d),\qquad
\langle u,a\rangle_{PM\!-\!A}:= \frac{1}{(2\pi)^n} \int_{\R^d} g(x)\, h(x)\,dx
\]
for any representatives $u=\widehat g\in PM$ and $a=\widehat h\in A$ with $g\in L^\infty$, $h\in L^1$. In the special case that $u \in L^1$, it follows from Plancherel's theorem that this agrees with the natural pairing of $u$ and $a$ given by $\langle u, a \rangle_{L^1-L^{\infty}} = \int u(x) a(x) dx = \int \hat g(x) \hat h(x) dx$.

Sometimes instead of $A(\R^d)$ we consider the closely related Fourier-Stieltjes algebra (see Chapter VI of \cite{katznelson2004introduction}):
\begin{definition}
The Fourier-Stieltjes algebra on $\R^d$, denoted by $B(\R^d)$ is given by the image of the space of complex Radon measures under the Fourier transform. It is equipped with the norm corresponding to the total variation of the measure:
\[ \|f\|_B = \|\mathcal F^{-1} f\|_{TV}. \]
\end{definition}


\subsubsection{Sobolev spaces and duality}
We will use Sobolev spaces corresponding to the $L^1$ and $L^{\infty}$ norms: we will consider negative Sobolev spaces which enlarge $L^1$ to contain some tempered distributions, and positive Sobolev spaces which restrict $L^{\infty}$ to functions with bounded derivatives.  
See the textbook by Adams \cite{adams2003sobolev} for an extensive reference on Sobolev spaces, including the below definitions and facts. 
\begin{definition}[Positive Sobolev space]
For $k\in\N$ set
\[
W^{k,\infty}(\R^d):=\{f\in L^\infty:\ \partial^\alpha f\in L^\infty\ \text{for all }|\alpha|\le k\},
\quad
\|f\|_{W^{k,\infty}}:=\sum_{|\alpha|\le k}\|\partial^\alpha f\|_{L^\infty}.
\]
Let $W^{k,\infty}_0(\R^d)$ be the $W^{k,\infty}$-closure of $C_0^\infty(\R^d)$, the space of smooth functions that vanish at infinity.
\end{definition}
Note that $W^{k,\infty}_0$  is a proper subspace: for example, it does not contain the constant function $1$.     
\begin{definition}[$C_b^\infty$ and its Fourier image]
Let
\[
C_b^\infty:=\{f\in C^\infty:\ \|\partial^\alpha f\|_{L^\infty}<\infty\ \forall \alpha\}
=\bigcap_{m\ge0}W^{m,\infty},
\]
be the space of smooth functions with bounded derivatives, which is equipped
with Fréchet seminorms \(p_\alpha(f)=\|\partial^\alpha f\|_\infty\).
Define
\[
\widehat{C_b^\infty}
:=\Bigl\{u\in\mathcal S':\ (i\xi)^\alpha u\in PM\ \text{for all multiindices }\alpha\Bigr\},
\]
with seminorms \(\|u\|_{PM,k}:=\sum_{|\alpha|\le k}\|(i\xi)^\alpha u\|_{PM}\).
Then \(\mathcal F:C_b^\infty\to\widehat{C_b^\infty}\) is a Fréchet isomorphism.
\end{definition}

\begin{definition}[Negative Sobolev space]
For $k\in\N$,
\[
W^{-k,1}(\R^d)
:=\Big\{\,f\in\mathcal S' : f=\sum_{|\alpha|\le k}\partial^\alpha f_\alpha,\ f_\alpha\in L^1(\R^d)\,\Big\},
\]
with norm $\ \|f\|_{W^{-k,1}}:=\inf\big\{\sum_{|\alpha|\le k}\|f_\alpha\|_{L^1}\big\}$ over such representations.
Equivalently,
\[
W^{-k,1}(\R^d)\ \cong\ \big(W^{k,\infty}_0(\R^d)\big)'\!,
\]
via $\langle \varphi, f\rangle_{W^{k,\infty}_0-W^{-k,1}}=\sum_{|\alpha|\le k}(-1)^{|\alpha|}\!\int f_\alpha\,\partial^\alpha\varphi$.
\end{definition}

\subsubsection{Fourier image of negative Sobolev spaces and \texorpdfstring{$A_{\mathrm{poly}}$}{A\_poly}}
\begin{definition}
We write $\widehat{W^{-k,1}}:=\mathcal F\big(W^{-k,1}\big)\subset\mathcal S'(\R^d)$ and give it the Banach norm
\[
\|u\|_{\widehat{W^{-k,1}}}:=\big\|\mathcal F^{-1}u\big\|_{W^{-k,1}}.
\]
\end{definition}

\begin{proposition}[Concrete description of $\widehat{W^{-k,1}}$]\label{prop:What=polyA}
For each $k\in\N$, we can explicitly write $\widehat{W^{-k,1}}$ as an $A(\R^d)$-module:
\[
\widehat{W^{-k,1}}
\;=\;\sum_{|\alpha|\le k}(i\xi)^\alpha\,A(\R^d).
\]
Moreover, pointwise multiplication satisfies
\[
\widehat{W^{-k,1}}\cdot \widehat{W^{-m,1}}\ \subset\ \widehat{W^{-(k+m),1}}.
\]
\end{proposition}

\begin{proof}[Proof sketch]
If $f=\sum_{|\alpha|\le k}\partial^\alpha g_\alpha$ with $g_\alpha\in L^1$, then
$\widehat f=\sum_{|\alpha|\le k}(i\xi)^\alpha\widehat{g_\alpha}\in\sum (i\xi)^\alpha A$.
Conversely, if $u=\sum_{|\alpha|\le k}(i\xi)^\alpha a_\alpha$ with $a_\alpha=\widehat{g_\alpha}\in A$, then
$\mathcal F^{-1}u=\sum_{|\alpha|\le k}\partial^\alpha g_\alpha\in W^{-k,1}$.
For the product: multiply two such finite sums and use that $A$ is a Banach algebra; degrees add.
\end{proof}

\begin{definition}[Inductive limits]
Define
\[
W^{-\infty,1}:=\bigcup_{k\ge 0}W^{-k,1}\quad\text{and}\quad
A_{\mathrm{poly}}:=\indlim_{k\to\infty}\ \widehat{W^{-k,1}}
=\bigcup_{k\ge0}\widehat{W^{-k,1}}.
\]
(We use the inductive-limit locally convex topology in both cases.)
We call $A_{\mathrm{poly}}$ the polynomially-weighted Fourier algebra.
\end{definition}

\begin{proposition}[Identification and algebra/module structure]\label{prop:Apoly=hatWinfty}
With these topologies,
\[
A_{\mathrm{poly}}\;=\;\widehat{W^{-\infty,1}},
\]
$A_{\mathrm{poly}}$ is a locally convex algebra under pointwise multiplication, and an $A$-module.
\end{proposition}

\begin{proof}[Proof sketch]
By definition, $A_{\mathrm{poly}}=\bigcup_k \widehat{W^{-k,1}}$, and $\widehat{\bigcup_k W^{-k,1}}=\bigcup_k \widehat{W^{-k,1}}$.
Algebra/module statements follow from Proposition~\ref{prop:What=polyA}.
\end{proof}

\begin{remark}[Equivalence with the Beurling–Fourier union]
Let $A_{\langlexi^{m}}:=\{u:\ \langlexi^{-m}u\in A\}$ which is sometimes called a Beurling-Fourier algebra.
Using $\partial^\alpha J_m\in L^1$ for $|\alpha|<m$ (Bessel kernels), one gets
$(i\xi)^\alpha\langlexi^{-m}\in A$ for $|\alpha|<m$, hence $\widehat{W^{-k,1}}\hookrightarrow A_{\langlexi^{m}}$ for $m>k$.
Conversely, $\langlexi^{2m}=(1+|\xi|^2)^m$ is a polynomial of degree $2m$, so $A_{\langlexi^{2m}}\hookrightarrow \widehat{W^{-(2m),1}}$.
Thus
\[
A_{\mathrm{poly}}=\bigcup_{k\ge0}\widehat{W^{-k,1}}
=\bigcup_{m\ge0}A_{\langlexi^{m}}.
\]
\end{remark}
\subsubsection{Relation to smooth functions slowly increasing at infinity (multipliers)}
$A_{poly}$ contains the well-studied class of \emph{multipliers}/multiplication operators, also known as the class of $C^{\infty}$ \emph{functions slowly increasing at infinity} (Definition 25.3 of \cite{treves2016topological}). These functions are Fourier transforms of \emph{convolutors}/convolution operators, also known as \emph{distributions rapidly decreasing at infinity}. The space of convolutors contains the compact tempered distributions, continuous functions rapidly decreasing at infinity, and other important classes --- see Chapter 30 of Treves \cite{treves2016topological} and also the textbook \cite{larsen2012introduction}. We review these concepts below.

\newcommand{\abs}[1]{\left|#1\right|}
\newcommand{\ip}[1]{(1+\abs{#1})} 
\begin{definition}[Convolutors \(\mathcal O'_C(\R^d)\)]
We denote by \(\mathcal O'_C(\R^d)\) the space of distributions \(T\in\mathcal{D}'(\R^d)\) with the following property:
for every integer \(h\ge 0\) there exist an integer \(m(h)\ge 0\) and a finite family
\(\{f_p\}_{p\in\N^n,\ |p|\le m(h)}\subset C^0(\R^d)\) such that
\[
  T=\sum_{|p|\le m(h)} \partial_x^p f_p,
\qquad\text{and}\qquad
  \lim_{|x|\to\infty} \ip{x}^{\,h}\, \abs{f_p(x)}=0
  \ \ \text{for all } p\in\N^n,\ |p|\le m(h).
\]
Equivalently, \(\mathcal O'_C(\R^d)\) is the space of those distributions for which the convolution
operator \(\varphi\mapsto T*\varphi\) maps \(\mathcal S(\R^d)\) continuously into itself (“convolution
operators on \(\mathcal S\)”).
\end{definition}

\begin{definition}[Multipliers \(\mathcal O_M(\R^d)\) and slow growth]
A function \(a\in C^\infty(\R^d)\) is called \emph{slowly increasing at infinity} if for every multi–index
\(\alpha\in\N^n\) there exists \(N_\alpha\in\N\) and \(C_\alpha>0\) such that
\[
  \abs{\partial^\alpha a(x)} \le C_\alpha\,\ip{x}^{\,N_\alpha}\qquad\text{for all }x\in\R^d .
\]
We write
\[
  \mathcal O_M(\R^d)
  := \bigl\{ a\in C^\infty(\R^d): \text{each } \partial^\alpha a \text{ has at most polynomial growth}\bigr\}.
\]
Equivalently, \(\mathcal O_M(\R^d)\) is the space of \emph{multipliers} of \(\mathcal S(\R^d)\): those
\(a\) for which \(a\varphi\in\mathcal S(\R^d)\) for all \(\varphi\in\mathcal S(\R^d)\), and the map
\(\varphi\mapsto a\varphi\) is continuous \(\mathcal S\to\mathcal S\).
\end{definition}

\begin{theorem}[Fourier transform: convolutors \(\leftrightarrow\) multipliers, Theorem 30.3 of \cite{treves2016topological}]
With the above Fourier convention, the Fourier transform is a bijection
\[
  \mathcal F:\ \mathcal O'_C(\R^d) \xrightarrow{\;\simeq\;} \mathcal O_M(\R^d),
\]
and it identifies convolution operators with multiplication operators in the sense that
for every \(T\in\mathcal O'_C(\R^d)\) and \(\varphi\in\mathcal S(\R^d)\),
\[
  \widehat{T*\varphi} = (\widehat T)\,\widehat\varphi, \qquad \widehat T \in \mathcal O_M(\R^d).
\]
Conversely, for every \(a\in\mathcal O_M(\R^d)\) there exists a unique \(T\in\mathcal O'_C(\R^d)\)
with \(\widehat T = a\), and multiplication \(\varphi\mapsto a\varphi\) corresponds under \(\mathcal F^{-1}\)
to convolution \(\varphi\mapsto T*\varphi\) on \(\mathcal S(\R^d)\).
\end{theorem}
\begin{corollary}
For all $n \ge 1$, $\mathcal{O}_M(\R^d) \subset A_{poly}(\R^d)$. 
\end{corollary}
\begin{proof}
Note that applying the definition of $\mathcal{O}'_C$ with $h > n$ implies that each $f_p \in L_1$, so indeed $\mathcal{O}_M(\R^d) \subset A_{poly}(\R^d)$. 
\end{proof}
\begin{remark}
We also easily see that $A_{poly}$ is a strictly larger space than the space of multipliers, because it is not contained in $C^{\infty}$.
\end{remark}

\section{Dimension-free polynomial approximation}\label{sec:Rd-analytic-lemmas} 


\subsection{Fourier representation of approximation error}
Let $\mu$ be a probability measure on $\R^d$ such that all moments exist.
Let $f \in A_{poly}(\R^d)$, so there exists some $k$ and $f_{\alpha} \in A(\R^d)$ such that 
\[ f(z) = \sum_{|\alpha| \le k} (iz)^{\alpha} f_{\alpha}. \]
Note that any such function $f$ has at most polynomial growth and therefore is guaranteed to be square-integrable under $\mu$:
\begin{lemma}
    $A_{poly}(\R^d) \subset L_2(\mu)$.
\end{lemma}
\begin{proof}
Consider $f$ as above, and observe by Cauchy-Schwarz that
\[ \E_{\mu} |f|^2 \le \left(\sum_{|\alpha| \le k} \E_{z \sim \mu} |z|^{2 \alpha}\right)\left(\sum_{|\alpha| \le k} \E_{\mu} |f_{\alpha}|^2\right) < \infty. \]
The first sum is bounded because all moments of $\mu$ exist, and the second is bounded because $f_{\alpha}$ is the inverse Fourier transform of a function in $L_1(\R^d)$, so it is uniformly bounded by H\"older's inequality. 
\end{proof}
Our analysis starts with the following
key lemma:
\begin{lemma}\label{lem:fourier-error-rep}
Let $r$ be the $L_2(\mu)$ orthogonal projection of $f$ onto any closed subspace and let
\[ \varphi(\xi) = \int e^{i \xi \cdot z} \bar r(z) \, d\mu(z). \] 
Then
\[ \|r\|_{\mu}^2 = \frac{1}{(2\pi)^n} \int  \sum_{|\alpha| \le k} \hat f_{\alpha}(\xi) \partial^{\alpha} \varphi(\xi)\, d\xi \]
where $\hat f_{\alpha} \in L^1$ is the Fourier transform of $f_{\alpha}$.
\end{lemma}
\begin{proof}
Since $f = r + (f - r)$ is an orthogonal decomposition, we know that
\begin{equation} \langle r, r \rangle_{\mu} = \langle f, r \rangle_{\mu} = \int f \bar r\, d\mu = \sum_{|\alpha| \le k}  \int f_{\alpha}(z) (i z)^{\alpha} \bar r(z) \, d\mu(z). \label{eqn:f-to-falpha}
\end{equation}
Each $f_{\alpha}$ is the inverse Fourier transform of a function in $L_1$ by assumption, so letting $\hat f_{\alpha}$ be the Fourier transform we have
\begin{equation} \int f_{\alpha}(z) (iz)^{\alpha} \bar r(z)\, d\mu(z) = \frac{1}{(2\pi)^n} \int \int \hat f_{\alpha}(\xi) e^{i \xi z} (iz)^{\alpha} \bar r(z) \, d\xi\, d\mu(z). \label{eq:to-fubini}
\end{equation}
Observe that the integrand of \eqref{eq:to-fubini} is absolutely integrable, because
\[ \int \int |\hat f_{\alpha}(\xi)| |z|^{\alpha} |r(z)| d\xi d\mu(z) = \|\hat f_{\alpha}\|_{L_1} \int  |z|^{\alpha} |r(z)| d\mu(z) \]
and by Cauchy-Schwarz and the assumption that all moments exist,
\[ \int |z|^{\alpha} |r(z)| d\mu(z) \le \sqrt{\int |z|^{2\alpha} d\mu(z)} \sqrt{\int |r(z)|^2 d\mu(z)} < \infty.  \]
Therefore, we are justified to apply Fubini's theorem to the right hand side of \eqref{eq:to-fubini} and obtain
\begin{align}  
\frac{1}{(2\pi)^n} \int \int \hat f_{\alpha}(\xi) e^{i \xi z} (iz)^{\alpha} \bar r(z) \, d\xi\, d\mu(z) 
&=  \frac{1}{(2\pi)^n} \int \hat f_{\alpha}(\xi) \int e^{i \xi z} (iz)^{\alpha} \bar r(z) \, d\mu(z)\, d\xi \\
&= \frac{1}{(2\pi)^n} \int \hat f_{\alpha}(\xi) \partial^{\alpha} \varphi(\xi)\, d\xi.
\end{align}
where in the last step we recognized the derivative of $\varphi$.
Therefore, returning to \eqref{eqn:f-to-falpha} we find that
\[  \langle r, r \rangle_{\mu} = \frac{1}{(2\pi)^n} \int  \sum_{|\alpha| \le k} \hat f_{\alpha}(\xi) \partial^{\alpha} \varphi(\xi)\, d\xi. \]
\end{proof}
\subsubsection{Fourier-Stieltjes variant}
For convenience, we state a variant of the above results when the function is represented as a polynomially-weighted sum of Fourier-Stieltjes functions. This is equivalent to the previous statement because of the standard fact\footnote{Let \(\Gamma\) be the fundamental solution of the Laplacian on \(\mathbb{R}^n\) so \(\Delta\Gamma=\delta_0\) and choose \(\chi\in C_c^\infty\) with \(\chi\equiv1\) near \(0\), set \(K:=\chi\nabla\Gamma\) and \(\phi:=-(\nabla \chi) \cdot \nabla \Gamma\); then \(\delta_0=\phi+\operatorname{div}K\) with \(\phi, K\in L^1_c\) by Theorem 3.3.2 of \cite{hörmander1983analysis}, hence for any finite measure \(\mu\),
\(\mu=(\phi*\mu)\,dx+\operatorname{div}(K*\mu)\) with \(\phi*\mu\in L^1\) and \(K*\mu\in L^1(\mathbb{R}^n;\mathbb{R}^n)\).} that every Radon measure can be written as the sum of a function in $L^1$ and the divergence of a vector-valued function in $L^1$. Therefore the function $f$ below still represents an element of the same space $A_{poly}$; this also follows from Theorem 30.3 of \cite{treves2016topological}.
\begin{lemma}\label{lem:fourier-rep-fs}
Let 
\[ f(z) = \sum_{|\alpha| \le k} (iz)^{\alpha} f_{\alpha} \]
where each $f_{\alpha} \in B(\R^d)$, so $f_{\alpha}(z) = \frac{1}{(2\pi)^n} \int e^{i \xi \cdot z} d\rho_{\alpha}$ and $\rho_{\alpha}$ is a complex Radon measure. Let $r$ be the $L_2(\mu)$ orthogonal projection of $f$ onto any closed subspace and let
\[ \varphi(\xi) = \int e^{i \xi \cdot z} \bar r(z) \, d\mu(z). \] 
Then
\[ \|r\|_{\mu}^2 = \frac{1}{(2\pi)^n}  \sum_{|\alpha| \le k} \int  \partial^{\alpha} \varphi(\xi)\, d\rho_{\alpha}(\xi) \]
\end{lemma}
\begin{proof}
The same proof applies --- see Theorem 7.26 of Folland's textbook \cite{folland1999real} for the suitable generalization of Fubini's theorem to Radon measures, which directly implies the analogous statement for complex Radon measures. 
\end{proof}

\subsection{Measure specific upper bounds}

Let $\mu$ be a continuous probability density on $\RR^{d}$ such that $\int |x^{\alpha}|\mu(x) dx <\infty$ for every multi-index $\alpha\in\NN^{d}$. 
Suppose that $f\in A_{poly}(\RR^{d})\cap L^{2}(\mu)$, and let $p_{D}$ be its orthogonal projection in $L^{2}(\mu)$ onto the space $\mathcal{P}_{\le D}$ of polynomials of total degree at most $D$ and $r_D = f - p_D$. 
Then \cref{lem:fourier-error-rep} and \cref{lem:fourier-rep-fs} provide a characterization of $\norm{r_D}_\mu^2$ in terms of $\hatf$ and $\varphi(\xi) = \cF[r_D \cdot \mu](\xi)$. 
In applications, the Fourier transform of $f$ is determined by the structure of learning targets. 
In many cases, the Fourier transform $\hatf$ decays in high-frequency components. 
Therefore, it suffices to control $\varphi(\xi)$ for $\xi$ in a bounded region, as the degree of the approximating polynomial increases. 

In what follows, we characterize the growth of $\varphi(\xi)$ under specific assumptions on $\mu$. 
We follow the techniques from complex analysis as in \cref{sec:complex-approx} and generalize them to $\RR^d$.
We start with a claim that $\varphi$ is a bounded function with an order-$D$ zero at the origin.  
\begin{lemma}[Moment--derivative vanishing]\label{lem:Rd-derivative-zero}
It holds that $\partial^{\alpha}_{\xi}\varphi(0)=0$ for every $|\alpha|\le D$. 
For any directions $v_1,v_2,\ldots,v_{k}\in \SS^{d-1}$ with $k\le D$, this implies that $\partial^k\varphi(\xi) [v_1,v_2,\ldots,v_k] =0$. 
Moreover, we have that for any $k\in \NN$ and $\xi\in\R^{d}$ that 
\begin{align}
   \sup_{\xi \in \RR^d}|\partial^k\varphi(\xi) [v_1,v_2,\ldots,v_k] |&= \norm{r}_\mu \, \Big\|{\prod_{j=1}^k \dotp{v_j}{x}}\Big\|_\mu < \infty.   
\end{align}
Additionally, $\varphi\in C^\infty(\RR^d)$. 
\end{lemma}
\begin{proof}[Proof of \cref{lem:Rd-derivative-zero}]
By the definition of $\varphi$, it follows that  
\begin{align}
  \partial^k \varphi(\xi)[v_1, \ldots, v_k] &= \int_{\RR^d} (-i)^k \prod_{j=1}^k \dotp{v_j}{x} e^{-i \xi \cdot x} r(x) \mu(dx).  
\end{align}
At $\xi = 0$, last integral vanishes whenever $k \le D$, since $\prod_{j=1}^k \dotp{v_j}{x}$ is a polynomial of degree less than or equal to $D$, and are orthogonal to $r$ in $L^2(\mu)$. 
For general $\xi\in \RR^d$, we apply the Cauchy-Schwarz inequality to obtain that 
\begin{align}
  \big|\partial^k \varphi(\xi)[v_1, \ldots, v_k]\big| &\le \int_{\RR^d} \Big| \prod_{j=1}^k \dotp{v_j}{x} \Big| |r(x)| \mu(dx) \le  \norm{r}_\mu \, \Bignorm{\prod_{j=1}^k \dotp{v_j}{x}}_\mu < \infty.  
\end{align}
This concludes the desired result.  
\end{proof}
This lemma establishes the global upper bound and the zero property of $\varphi(\xi)$.  
Suppose that $\hatf$ decays for large $\xi$, the boundedness of $\varphi(\xi)$ with \cref{lem:fourier-reinterpret} implies that $\norm{r_D}_\mu$ is controlled by the local growth of $\varphi(\xi)$ around the origin. 
In addition, the zero property helps to control the growth of $\varphi(\xi)$ locally around the origin.  
In particular, we can generalize \cref{lem:complex2,lem:complex} to $\RR^d$ by applying them along each direction of the input space $\RR^d$. 

\paragraph{Strictly sub-exponential bound.} 
We first introduce the strictly sub-exponential distributions on $\RR^d$. 
\begin{definition}[Strictly sub-exponential distribution]\label{def:strictly-subexp} 
Fix $r\ge 1$. 
A distribution $\mu$ on $\RR^d$ is \emph{$r$-strictly sub-exponential} if for any $u\in \SS^{d-1}$ and $t\in \RR$, it holds that $\norm{e^{\dotp{x}{u}t}}_\mu \le  A \exp((K|t|)^r)$ for some constants $A,K>0$.  
\end{definition}
In this case, we have the following generalization of \cref{lem:complex2}.  

\begin{lemma}\label{lem:complex2-Rd}\label{thm:Rd-pw-strict}
For an $r$-strictly sub-exponential distribution $\mu$ on $\RR^d$ and $f\in A_{poly}(\RR^d)$, let $\varphi(\xi) = \cF[r_D\cdot \mu](\xi)$. 
Then we have for any $\xi \in \RR^d$ with $D > r (K\|\xi\|_2)^r$, that 
\begin{align}
  |\varphi(\xi)| \le  A \norm{r}_\mu \, \Big(\frac{er}{D}\Big)^{D/r} (K \|\xi\|_2)^{D}.
\end{align}
In addition, for any $v_1, v_2, \ldots, v_k\in \SS^{d-1}$ with $k \le D$, it holds that 
\begin{align}
\big|  \partial^k \varphi(\xi)[v_1, v_2, \ldots, v_k] \big|&\le \sqrt{A} \, \norm{r}_\mu \kappa_{4}(v_1,v_2,\ldots, v_k) \, \Big(\frac{er}{D-k}\Big)^{(D-k)/r} (2K\|\xi\|_2)^{D-k}, 
\end{align}
where $\kappa_{4}(v_1,v_2,\ldots, v_k) = \Big\|{\Big(\prod_{j=1}^k \dotp{v_j}{x}\Big)^2}\Big\|_\mu^{1/2}$.  
\end{lemma}

\begin{proof}[Proof of \cref{lem:complex2-Rd}]
For the upper bound of $|\varphi(\xi)|$, we consider the one-dimensional function $\varphi_u(\zeta) = \varphi(u\zeta)$ for fixed $u\in \SS^{d-1}$. 
Then $\varphi_u$ is entire on $\CC$ and can be expressed as 
\begin{align}
  \varphi_u(\zeta) = \int_{\RR^d} e^{-i \zeta \dotp{u}{x}} r(x) \mu(dx).
\end{align}
Then $\varphi_u$ is entire on $\CC$ and satisfies
\begin{align}
  |\varphi_u(\zeta)| &\le \norm{r}_\mu \, \norm{e^{|\dotp{u}{x}| \cdot |\Re(\zeta)|}}_\mu \\ 
  &\le A\norm{r}_\mu \, \exp\big((K|\Re(\zeta)|)^r\big) \\ 
  &\le A\norm{r}_\mu \, \exp\big((K|\zeta|)^r\big).
\end{align}
By \cref{lem:Rd-derivative-zero}, $\varphi_u$ has an order-$D$ zero at the origin. 
Applying \cref{lem:complex2} to $\varphi_u$, when $D > r(K|\zeta|)^r$, we obtain
\begin{align}
  |\varphi_u(\zeta)| &\le A\norm{r}_\mu \, \Big(\frac{er}{D}\Big)^{D/r} (K|\zeta|)^D.
\end{align}
Setting $\zeta = \|\xi\|_2$ and $u = \xi/\|\xi\|_2$, we have $|\varphi(\xi)| = |\varphi_u(\|\xi\|_2)|$, last inequality implies that when $D > r (K\|\xi\|_2)^r$, $|\varphi(\xi)|  \le A\norm{r}_\mu \, \Big(\frac{er}{D}\Big)^{D/r} (K\|\xi\|_2)^D$. 
This completes the first part of the proof.  
For the directional derivatives, fix $v_1, \ldots, v_k \in \SS^{d-1}$ with $k \le D$ and define  
\begin{align}
  \varphi_u^{(k)}(\zeta) = \int_{\RR^d} \prod_{j=1}^k \dotp{v_j}{x} \cdot e^{-i\zeta\dotp{u}{x}} r(x) \mu(dx).
\end{align}
Similar to  \cref{lem:Rd-derivative-zero}, this function has an order at least $D-k$ zero at the origin. 
Using Cauchy-Schwarz inequality repeatedly yields that  
\begin{align}
  |\varphi_u^{(k)}(\zeta)| &\le \norm{r}_\mu \, \Bignorm{\prod_{j=1}^k \dotp{v_j}{x} \cdot e^{|\dotp{u}{x}| \cdot |\Re(\zeta)|}}_\mu \\
  &\le \norm{r}_\mu \, \kappa_{4}(v_1,\ldots,v_k) \cdot \norm{e^{2|\dotp{u}{x}| \cdot |\Re(\zeta)|}}_\mu^{1/2} \\
  &\le \sqrt{A} \, \norm{r}_\mu \, \kappa_{4}(v_1,\ldots,v_k) \, \exp\big((2K|\zeta|)^r / 2 \big).
\end{align}
Applying \cref{lem:complex2} when $D-k > r(2K|\zeta|)^r$, we obtain
\begin{align}
  |\varphi_u^{(k)}(\zeta)| &\le \sqrt{A} \, \norm{r}_\mu \, \kappa(v_1,\ldots,v_k) \, \Big(\frac{er}{D-k}\Big)^{(D-k)/r} (2K|\zeta|)^{D-k}.
\end{align}
This completes the proof by setting $\zeta = \|\xi\|_2$ and $u = \xi/\|\xi\|_2$.  
\end{proof}



In comparison to \cref{lem:complex2}, we also include the upper bound of the directional derivatives of $\varphi(\xi)$. 
This bound is useful for some target functions, where the Fourier transform $\hatf$, as a tempered distribution, contains some high-order derivatives that act on $\varphi(\xi)$.


\paragraph{Sub-exponential bound.} 
We say that a distribution $\mu$ on $\RR^d$ is \emph{sub-exponential} if, for some $K>0$ and any $u\in \SS^{d-1}$, it holds that $\norm{e^{|\dotp{x}{u}|/K}}_\mu \le e$.  
In this case, we have the following generalization of \cref{lem:complex}. 

\begin{lemma}\label{lem:complex-Rd}\label{thm:Rd-pw-subexp}  
Let $\mu$ be a sub-exponential distribution on $\RR^d$ and $f\in L^2(\mu)$.  
Let $\varphi(\xi) = \cF[r_D\cdot \mu](\xi)$, where $r_D$ is the residual function after projecting $f$ onto polynomial space of degree at most $D$. 
Then for any $\xi \in \RR^d$, it holds that 
\begin{align}
  |\varphi(\xi)| &\le \norm{r}_\mu \, \tanh\Big( \frac{K\pi \|\xi\|_2}{4}\Big)^{D}. 
\end{align}
In addition, for any $v_1, v_2, \ldots, v_k\in \SS^{d-1}$ with $k \le D$, it holds that 
\begin{align}
\big|  \partial^k \varphi(\xi)[v_1, v_2, \ldots, v_k] \big|&\le \norm{r}_\mu \, \kappa_{4} (v_1,v_2,\ldots, v_k) \,  \tanh \Big( \frac{K\pi \|\xi\|_2}{2}\Big)^{D-k}, 
\end{align}
where $\kappa_{4}(v_1,v_2,\ldots, v_k) = \Big\|{\Big(\prod_{j=1}^k \dotp{v_j}{x}\Big)^2}\Big\|_\mu^{1/2}$.    
\end{lemma}
\begin{proof}[Proof of \cref{lem:complex-Rd}]
For any fixed $u\in \SS^{d-1}$, we consider the one-dimensional function $\varphi_u(\zeta) = \varphi(u\zeta)$ for $\zeta \in \CC$. 
Then $\varphi_u$ is strip-analytic and
\begin{align}
  |\varphi_u(\zeta)|  &\le  \norm{r}_\mu \,  \norm{e^{|\dotp{u}{x}|\cdot |\Im(\zeta)|}}_\mu \le e\norm{r}_\mu,
\end{align}
when $|\Im(\zeta)| \le K$. 
By \cref{lem:Rd-derivative-zero}, $\varphi_u$ has an order-$D$ zero at the origin.
Applying \cref{lem:complex}, we obtain that 
\begin{align}
  |\varphi_u(\zeta)| &\le \norm{r}_\mu  \tanh( \frac{K\pi |\zeta |}{4})^D. 
\end{align}
Setting $\zeta = \|\xi\|_2$ and $u = \xi/\|\xi\|_2$, we have 
$$|\varphi(\xi)| = |\varphi_u(\|\xi\|_2)|\le \norm{r}_\mu  \tanh( \frac{K\pi \|\xi\|_2}{4})^D.$$  
This completes the proof for the first part. 
For the second part, we define 
\begin{align}
  \varphi_u^{(k)}(\zeta) = \int_{\RR^d} \prod_{j=1}^k \dotp{v_j}{x} \cdot e^{-i\zeta\dotp{u}{x}} r(x) \mu(dx).  
\end{align}
Then setting $\zeta = \|\xi\|_2$ and $u = \xi/\|\xi\|_2$ yields $\varphi_u^{(k)}(\zeta) = \partial^k \varphi(\xi)[v_1, v_2, \ldots, v_k]$. 
Similar to  \cref{lem:Rd-derivative-zero}, this function has an order at least $D-k$ zero at the origin. 
Cauchy-Schwarz inequality implies that
\begin{align}
  |\varphi_u^{(k)}(\zeta)| &\le \norm{r}_\mu \, \Bignorm{\prod_{j=1}^k \dotp{v_j}{x} \cdot e^{|\dotp{u}{x}| \cdot |\Im(\zeta)|}}_\mu \\ 
  &\le \norm{r}_\mu \, \kappa_{4}(v_1,\ldots,v_k) \cdot \norm{e^{2|\dotp{u}{x}| \cdot |\Im(\zeta)|}}_\mu^{1/2} \\  
  &\le e^{1/2} \, \norm{r}_\mu \, \kappa_{4}(v_1,\ldots,v_k), 
\end{align}
for $|\Im(\zeta)| \le K/2$. 
Invoking \cref{lem:complex} when $2|\Im(\zeta)|\le K$, we obtain that 
\begin{align}
  |\varphi_u^{(k)}(\zeta)| &\le e^{-1/2}\norm{r}_\mu \, \kappa_{4}(v_1,\ldots,v_k) \,  \tanh ( \frac{K\pi |\zeta|}{2})^{D-k}. 
\end{align}
Setting $\zeta = \|\xi\|_2$ and $u = \xi/\|\xi\|_2$ completes the proof.  
\end{proof}

\paragraph{General case of Carleman's condition.} As in the one-dimensional case, whenever we have a distribution such that all of its 1-dimensional projections satisfy Carleman's moment condition, we can mirror the argument from the sub-exponential case and appeal to Theorem~\ref{thm:qdc}.

\paragraph{Summary.} 
In combination with \cref{lem:fourier-error-rep,lem:fourier-rep-fs}, these two lemmas provide powerful tools to bound the polynomial approximation error for various learning targets and data distributions. 
In the sequel, we apply them to two specific examples: smoothed analysis of agnostic learning and agnostically learning neural networks.
In the first case, targets are smoothed with Gaussian kernels, whose Fourier transforms decay exponentially. 
In the second case, explicit structure of the neural networks with various activations can be well characterized in the Fourier transform.

\section{Application: Smoothed analysis of learning}\label{sec:app-smoothed-analysis}

We now use our theory of polynomial approximation for general function to analyze the sample complexity of smoothed analysis of agnostic learning in the classification problem \citep{chandrasekaran2024smoothed,spielman2004smoothed}. 
In this setting, the learner aims to find a hypothesis with classification error that matches the best smoothed classification error of the function in a given function class. 
The smoothed error means that the hypothesis is evaluated on the input data corrupted by Gaussian noise. 
Formally, we have the following setup.
 \begin{definition}[Smoothed agnostic learning] 
Fix $\eps,\sigma>0$ and $\delta \in (0,1)$ and let $\gamma_d$ be the standard Gaussian distribution on $\RR^d$. 
Suppose that $\mu$ is a distribution over $\RR^d \times\{\pm 1\}$ and $\cH$ is a hypothesis class where each element $h\in \cH$ is a function from $\RR^d$ to $\{\pm 1\}$. 
And let $\cD = \{(x_i,y_i)\}_{i\in [n]}$ be $n$ iid samples drawn from $\mu$. 
We say that an algorithm $\cA$ that takes $\cD$ as input and outputs a hypothesis $f:\RR^d\to \{\pm 1\}$ learns $\cH$ in the $\sigma$--smoothed agnostic setting if with probability at least $1-\delta$ over the draw of $\cD$, it holds that
\begin{align}
  \PP_\mu  (h(x)\neq y) \le \mathrm{opt}_{\sigma,\cH} + \eps, \quad \text{where}\quad \mathrm{opt}_{\sigma,\cH } = \inf_{f\in \cH} \EE_{z\sim \gamma_d}[\PP_\mu (f(x + \sigma z ) \neq y)]. \label{eq:opt-sigma}  
\end{align} 
\end{definition}
Following the settings of \cite{chandrasekaran2024smoothed}, we focus on learning the class of low-dimensional functions that is defined as
\begin{align}
   \cH(k) & = \{f :\RR^d\to \{\pm 1\}:h(x) = h(\mathrm{proj}_\cU  x) \text{ for some }k\text{--dimensional subspace }\cU\subset \RR^d\},
\end{align} 
for some integer $k\le d$. 

A critical step in constructing such a learning algorithm is to establish polynomial approximation for functions in $\cH(k)$ in the smoothed input distribution. 
Provided the polynomial approximations, polynomial regression is able to achieve the desired learning guarantees. 
In the setting of smoothed analysis, the effect of the input Gaussian noise can be casted into smoothing the target function. 
We define the Gaussian smoothing operator $T_\sigma :f\mapsto T_\sigma f$ that maps any function $f\in L^\infty(\RR^d)$ to 
\begin{align}
  T_\sigma f (x) \coloneqq \int_{\RR^d} f(x+ \sigma z) \cdot d\gamma_d(z) . 
\end{align}
In the frequency domain, this operation is equivalent to multiplying\footnote{Multiplication between a tempered distribution and $C^\infty$ function is well defined, see \cref{par:differentiation-multiplication}} the Fourier transform with a Gaussian kernel, i.e., 
\begin{align}
  \cF[T_\sigma f](\xi) = \cF[f](\xi) \cdot \exp\{-\sigma^2 \|\xi\|_2^2/2\}. 
\end{align} 
Since we are measuring the un-smoothed classification error for a candidate hypothesis, it is never worse to use the smoothed function $T_\sigma f$ to predict the label with real input, than using original function $f$ over smoothed input. 
Concretely, Jensen's inequality implies that 
\begin{align}
  \EE_{(x,y)\sim \mu }[|T_\sigma f (x) - y|] & = \EE_{ (x,y)\sim\mu }\big |\EE_{z\sim \gamma_d}[f(x+ \sigma z) - y ]\big| \\ 
  &\le  \EE_{z, (x,y)}[|f(x+ \sigma z) - y|]. \label{eq:jensen-inequality}
\end{align}
Suppose that $f_\mathrm{opt}$  attains $\mathrm{opt}_{\sigma,\cH}$ defined in \cref{eq:opt-sigma}, then the right hand side is exactly $2\mathrm{opt}_{\sigma,\cH}$.   
Once we find a polynomial $P$ such that $\EE_{x\sim \mu_x}[|T_\sigma f_{\mathrm{opt}}(x) - P(x)|]<\eps$, it holds that
\begin{align}
  \EE_{(x,y)}[|P(x) - y|] &\le \EE_{(x,y)}[|T_\sigma f_{\mathrm{opt}}(x) - y|] + \EE_{x }[|T_\sigma f_{\mathrm{opt}}(x) - P(x)|] \\  
  &\le 2 \mathrm{opt}_{\sigma,\cH} + \eps.  \label{eq:polynomial-regression-bound-reduction}
\end{align}
And using the technique of thresholding yields the desired hypothesis. 
We defer the details of the approximating polynomial and polynomial regression to \cref{sec:poly-approx-smoothed-targets} and \cref{sec:poly-regression-learns} respectively. 
Now we summarize the main result on the sample complexity of learning $\cH(k)$ in the smoothed agnostic setting.
\begin{theorem}\label{thm:smoothed-learning}\label{thm:poly-regression-strictly-subexp}\label{thm:poly-regression-strictly-exp}
  Fix $\sigma,\eps>0$ and $\delta\in (0,1)$. 
  There exists an algorithm that learns $\cH(k)$ in the $\sigma$-smoothed agnostic setting with excess error $\eps$ and confidence $1-\delta$ using  i.i.d. samples of
  \begin{enumerate}
    \item size $n = O\Big(\eps^{-2}  \big(  d^{ O((\frac{k\log (k/\sigma)+  \log(1/\eps)}{\sigma^2})^{r/2})} + \log(1/\delta)\big)\Big)$ under $r$-strictly sub-exponential input;
    \item size $n = O\Big(\eps^{-2}  \big( d^{(\eps^{-1} k^k)^{O(\sigma^{-1})}}+ \log(1/\delta)\big)\Big)$ under sub-exponential input.
  \end{enumerate}
\end{theorem}  
\begin{proof}
See \cref{proof:thm-smoothed-learning}.
\end{proof}

We compare our results with Theorem 3.9 and Proposition 3.8 in \cite{chandrasekaran2024smoothed} from several aspects.
Remarkably, our result does not require the function class to have finite Gaussian surface area, which is a key assumption in \cite{chandrasekaran2024smoothed}. 
Second, the degree of the approximating polynomial in \cref{thm:smoothed-learning} has better dependence on the intrinsic dimension $k$ and the accuracy $\eps$. 
In the Proposition 3.8 of \cite{chandrasekaran2024smoothed}, the degree of the polynomial approximation is $O\big((k/ (\sigma  \eps^2))^{O(r^3 )}\big)$ \footnote{Our definition of the strictly sub-exponential distribution is equivalent to theirs with relation $r= 1 +\alpha^{-1}$.}.
In contrast, our polynomial approximation requires only degree $\tO\Big(\big((k \log(1/\eps))/\sigma^2\big)^{r/2}\Big)$. 
We remove the polynomial dependence on $\eps^{-1}$ in the degree of the approximating polynomial and improve the exponent of $k$ from $O(r^3)$ to $O(r)$. 
Additionally, our technique applies to a broader family of input distributions including sub-exponential distributions, while \cite{chandrasekaran2024smoothed} only handles strictly sub-exponential distributions. 
Hence, we successfully extend the smoothed analysis of learning to a wider range of input distributions with improved sample complexity. 

We elaborate more on the first difference by looking into the polynomial approximation procedure.
Although both results rely on polynomial approximations of the smoothed target functions, the initial procedure depicted above deviates from \cite{chandrasekaran2024smoothed}. 
The analysis presented in \cite{chandrasekaran2024smoothed} first introduces $T_\rho f$ with small $\rho$ using the Gaussian surface area assumption over the function $f$. 
In this way, the function $T_\rho f$ approximates $f$ well in the sense that $\EE [|T_\rho f(x+ \sqrt{1-\rho^2} z) - f(x+z)|] < \eps/2.$ 
Then, they construct a randomized polynomial $P_z(x)$ that approximates  $T_\rho  f$ over the smoothed input distribution, i.e.,  $\EE_{z,x}[ |T_\rho f(x+ \sqrt{1-\rho^2}z) -  P_z(x)|] < \eps /2$. 
As a result, it holds that $\EE_{x,z} [|f(x+z)- P_z(x)|] < \eps$. 
And
\begin{align}
  \EE_{z,(x,y)}[| P_z(x) - y|] &\le \EE_{z,(x,y)}[|f(x+z) -y|] + \EE_{z,(x,y)}[|T_\rho f(x+ \sqrt{1-\rho^2} z) - P_z(x)|]. \label{eq:adam-inequality}
\end{align}
For an appropriate choice of $f$, \cref{eq:adam-inequality} implies that the corresponding polynomial $P_z$ achieves the smoothed optimality up to $\eps$.  
However, if we first derive a polynomial $P (x)$ to approximate $T_1  f(x)$ following \cref{eq:jensen-inequality}, then the second term in \cref{eq:adam-inequality} can no longer be controlled using Jensen's inequality.  
In fact, the approach in \cite{chandrasekaran2024smoothed} introduced some unnecessary randomness as the intermediate targets, and potentially incurs a higher degree to approximate the target function.

\subsection{Polynomial approximation of smoothed targets}\label{sec:poly-approx-smoothed-targets}

With the reduction in \cref{eq:polynomial-regression-bound-reduction},
it now suffices to find a polynomial approximation to the smoothed function $T_\sigma f$ using the results in \cref{sec:Rd-analytic-lemmas}. 

Since the target function $f\in \cH(k)$ only depends on a $k$-dimensional projection of the input, we can reduce the polynomial approximation problem in $\RR^d$ to that in $\RR^k$. 
Once we identified an approximating polynomial in $\RR^k$, we can lift it to $\RR^d$ by composing with the projection operator of the oracle subspace associated with $f$.  
This helps to provide better dependence that scales with the intrinsic dimension $k$ instead of the ambient dimension $d$. 

To apply the results in \cref{sec:Rd-analytic-lemmas}, we need to control the Fourier transform of the smoothed target function. 
Although the original target function $f$ is bounded and can be naturally regarded as a tempered distribution, structural properties of its Fourier transform are hard to characterize without further assumptions. 
Indeed, we can construct bounded functions like Dirac comb whose Fourier transforms does not decay at all. 
However, with the concentration assumptions on the input distribution, we can truncate the target function to a bounded domain without incurring much error.
The following lemma is a convenient tool to control the error induced by truncating the tail of a smoothed function for distributions with exponentially  decaying tail.  
\begin{lemma}\label{lem:gaussian-spatial-truncation}
Let $\mu$ be a probability measure on $\R^d$ satisfying that
\begin{align}
  \sup_{u\in\SS^{d-1 }}\PP_{x\sim\mu}\bigl(\,|\langle u,x \rangle |>t\,\bigr)
  \;\le\;A\,\exp\!\big(-({t} /{b})^{c}\bigl ), \label{eq:tail-bound}
\end{align}
for some $b>0$, $c>0$ and $A>0$ and all $t>0$. 
Let $f$ be a bounded function on $\R^d$ with $\|f\|_{L^\infty(\RR^d)}\le M$. 
Then for every $\varepsilon>0 $, if we set $R >2(\sigma + b) \cdot \big(\log (2A  / \eps^2) +d \log 5 \big)^{1/c'}$, where $c'=\min\{c,2\}$,  it holds that  that $\norm{T_\sigma f - T_\sigma (f\cdot \ind_{B_R})}_{\mu}\le M \varepsilon$. 
\end{lemma}

\begin{proof}
See \cref{proof:gaussian-spatial-truncation}. 
\end{proof}

This lemma indicates that we can resort to approximating the smoothed version of the truncated target $T_\sigma (f\cdot \ind_{B_R})$ with properly chosen $R$. 
Since $f\cdot \ind_{B_R}$ is compactly supported and bounded, its Fourier transform is bounded and decays as suggested by Riemann-Lebesgue lemma. 
However, the decay rate is guaranteed in general since the original function is not necessarily smooth. 
Fortunately, the Gaussian smoothing further guarantees that the Fourier transform of $T_\sigma (f\cdot \ind_{B_R})$ decays exponentially fast. 
Similar the one-dimension result in \cref{lem:fourier-apx-bound}, we have that
the integral of the Fourier transform outside a ball can be bounded with an exponentially decaying term in $\Omega$. 
On the other hand, choosing appropriate $D$ leads to an upper bound on $|\varphi(\xi)|$ for $\xi \in B_D$, and leads to the following polynomial approximation result. 

\begin{proposition}\label{prop:poly-approx} 
Suppose that $\eps,\sigma>0$. 
Let $\mu$ be a probability measure on $\RR^d$. 
Then, for any functions $f\in \cH(k): \RR^d\to \{\pm 1\}$. 
there exists a polynomial $P$ with total degree 
\begin{enumerate}
  \item $O\Big(\big(  \frac{{\log (\eps^{-1})  + k\log(1/\sigma)}} {\sigma^2}\big)^{r/2}\Big)$ under $r$-strictly sub-exponential $\mu$ with $r\le 2$;
  \item  $O\Big(\big( \frac {\log (\eps^{-1})  + k\log(k/\sigma)}{\sigma^2}\big)^{r/2}\Big)$ under $r$-strictly sub-exponential $\mu$ with $r>2$; 
  \item $O\big(  \big(\eps^{-1} (k/\sigma)^k\big)^{O(1/\sigma)}\big)$ under sub-exponential $\mu$,
\end{enumerate}
such that $\norm{T_\sigma f - P}_{L^2(\mu)} \le \eps$.  
\end{proposition}

\begin{proof}
  See \cref{proof:poly-approx}.  
\end{proof}

\subsection{Polynomial regression learns smoothed targets} \label{sec:poly-regression-learns}

Once we know that there is a polynomial that approximates the smoothed target function, the polynomial thresholding functions (PTFs) are adequate to learn the targets in the smoothed setting, once the degree is sufficiently large. 
In the sequel, we denote $\cP_D$ as the subspace of all polynomials over $\RR^d$ with degree at most $D$. 

\begin{algorithm}
\caption{Polynomial Regression for Smoothed Learning}\label{alg:poly-regression}
\begin{algorithmic}[1]
\Require Training data $\cD = \{(x_i, y_i)\}_{i=1}^n$ iid sampled from distribution $\mu$, degree $D$. 

\State Solve the $L^1$  polynomial regression problem: $\hat P_\cD  = \argmin_{ P\in \cP_D } n^{-1 }\sum_{i=1}^n |P(x_i) - y_i|$. 
\State Solve the  one-dimensional optimization problem: $\hat t_\cD = \argmin_{t\in [-1,1]} (2n)^{-1 }\sum_{i=1}^n |\mathrm{sign}(P(x_i)- t) - y_i|$.
\State \Return The polynomial threshold classifier $\hath_\cD(x) = \mathrm{sign}(\hatP_\cD(x) - \hat t)$. 
\end{algorithmic}
\end{algorithm}

\begin{lemma}\label{lem:poly-regression-learns}
Fix $\sigma>0$, $\eps\in (0,1)$. 
Let $\mu$ be a distribution over $\mathbb{R}^d \times \{\pm 1\}$. \
If for any target $f\in \cH(k)$, there exists a polynomial $P$ of degree $D_\eps$ such that $\|T_\sigma f - P\|_{L^2(\mu)} \le \eps/4$, 
then there exists a procedure that takes $n= O\big(\eps^{-2}  d^{D_\eps}\log(1/\delta)\big)$ training samples, operates in $\mathrm{poly}(n,d)$ time, and outputs a polynomial threshold classifier $\hath$ such that with probability at least $1-\delta$, 
\begin{align}
\mathbb{P}_{(x,y)\sim\mu}(h(x) \neq y) &\leq \mathrm{opt}_{\sigma,\cH(k)} + \eps. \label{eq:poly-regression-learns} 
\end{align}
\end{lemma}
\begin{proof}
See \cref{app:poly-regression-proofs}.
\end{proof}
This result essentially follows from the technique of $L^1$-polynomial regression with a standard replication technique (as in e.g. \cite{kalai2008agnostically}). 
Combining \cref{prop:poly-approx} and \cref{lem:poly-regression-learns} yields the main result stated in \cref{thm:smoothed-learning}.

\section{Application: Distribution-free learning of intersections of halfspaces}\label{sec:khalfspace}
 In this section we illustrate an interesting application of the polynomial approximation results combined with dimension reduction, inspired by the previous work of Chandrasekaran et al \cite{chandrasekaran2024smoothed}. First we start with a definition of an intersection of halfspaces with margin:

\begin{definition}
An intersections of $K$ halfspaces with $\gamma$ margin with respect to a (possibly infinite) collection of vectors $\mathcal X$ contained in the unit ball is specified by vectors $w_1,\ldots,w_k$ in the unit ball $B$. These vectors must satisfy the requirement that $\mathcal X \subset B_+ \cup B_-$, where the region $B_+$ is defined by
\[ B_+ = \{ x \in B: \forall i \in [k], \langle w_i, x \rangle \ge 0 \} \]
and similarly the region containing negative examples is
\[ B_- = \{ x \in B : \exists i \in [k], \langle w_i, x \rangle \le -\gamma \}. \]
\end{definition}

Fix a collection of vectors $w_1,\ldots,w_K \in B$, let $\gamma > 0$, and define $B_+,B_-$ to be an intersection of halfspaces as above. 
Suppose there is an arbitrary distribution $\mathcal P$ over the unit ball $B$ in $\mathbb{R}^n$ which is supported almost surely in $B_+ \cup B_-$,
 and let $x^{(1)},\ldots,x^{(N)}$ be i.i.d. samples from $\mathcal P$. Let $\mathcal X = \{x^{(1)},\ldots,x^{(N)}\}$ denote the training set.
Suppose also that for every point $x_i$ we also observe a label $y_i = 1(x_i \in B_+)$. We show how to learn an improper predictor of the label from the training set which will generalize well, i.e. predicts the correct label with probability $1 - \epsilon$ for a fresh sample from $\mathcal P$. 

\begin{theorem}\label{thm:khalfspace}
There exists a polynomial time algorithm (see below) 
for the problem of learning an intersection of $K$ halfspaces with margin $\gamma$ obtaining the following guarantee.
For any $\delta, \varepsilon > 0$, with probability at least $1 - \delta$ the algorithm outputs a classifier with accuracy $1 - \varepsilon$ using 
\[ N = e^{\Theta\left(\log^{1.5}(2K/\epsilon)\log^2(\log(K/\epsilon)/\gamma^2)/\gamma^3 \right)} \log(1/\delta) \]
samples and runtime $poly(N)$. 
\end{theorem}
\begin{remark}[Comparison to Corollary 1.6 of \cite{chandrasekaran2024smoothed}]
In the previous work of \cite{chandrasekaran2024smoothed}, they analyzed a similar algorithm: they used a randomized projection to $\tilde m$ dimensions, where $\tilde m$ is somewhat larger than our choice of $m$, perform polynomial regression, and use replication to boost the success probability. For the proof, they reduced to their smoothed analysis setting using the same setting of $\sigma$, and then applied their Theorem C.7 which leads to a sample complexity of 
\[ \tilde N = (\tilde m)^{O((\Gamma/\epsilon)^4\log(1/\epsilon)/\sigma^2)} \log(1/\delta) = e^{O(\log^2(K)\log(K/\epsilon)\log(1/\epsilon)\log(\tilde m)/\gamma^2\epsilon^4)} \log(1/\delta) \]
since their parameter $\Gamma = O(\sqrt{\log K})$ by Lemma A.2 of \cite{chandrasekaran2024smoothed}, and we can find from the proof of Theorem C.7 that $\log(\tilde m) = O(\log(K/\epsilon\gamma))$. 

To summarize, our guarantee is of the form $\exp(\mathrm{poly}(\log k, 1/\gamma, \log(1/\epsilon)))$ compared to the previous guarantee of the form $\exp(\mathrm{poly}(\log k, 1/\gamma, 1/\epsilon))$. Our result has an improved dependence on the error parameter $\epsilon$, as the runtime and sample complexity are quasipolynomial in $1/\epsilon$ instead of exponential. 
\end{remark}
\subsection{Preliminaries}
In our argument we use some standard results which can be found in the textbook \cite{vershynin2018high}, whose notation we follow. We only need to refer to these notations and facts in this section.

\paragraph{Facts about sub-exponential  random variables.}  
For a random variable $X$, we define the sub-exponential  norm
\[ \|X\|_{\psi_1} = \inf \{ K > 0 : \mathbb E \exp(|X|/K) \le 2 \}\]
and sub-Gaussian norm
\[ \|X\|_{\psi_2} = \sqrt{\|X^2\|_{\psi_1}}. \]
The product of sub-Gaussian random variables is always sub-exponential ,
regardless of dependence/independence.
\begin{lemma}[Lemma 2.8.6 of \cite{vershynin2018high}]\label{lem:subg_prod}
If $X$ and $Y$ are sub-Gaussian then $XY$ is sub-exponential  with
\[ \|XY\|_{\psi_1} \le \|X\|_{\psi_2} \|Y\|_{\psi_2}. \]
\end{lemma}
The sub-exponential norm also satisfies a centering inequality.
\begin{lemma}[Equation 2.26 of \cite{vershynin2018high}]\label{lem:sube_center}
For any sub-exponential  random variable $X$,
\[ \|X - \mathbb E X\|_{\psi_1} \lesssim \|X\|_{\psi_1}. \]
\end{lemma}

In the dimension reduction argument we use the following sub-exponential  Bernstein inequality.
\begin{theorem}[Theorem 2.9.1 of \cite{vershynin2018high}]\label{thm:bernstein-subexp}
There exists an absolute constant $c > 0$ such that the following result holds.
Let $X_1,\ldots,X_n$ be independent mean-zero sub-exponential  random variables. Then
\[ \Pr(\left|\sum_{i = 1}^n X_i\right| \ge t) \le 2\exp\left(-c \min\left(\frac{t^2}{\sum_{i = 1}^n \|X_i\|_{\psi_1}^2}, \frac{t}{\max_i \|X_i\|_{\psi_1}}\right)\right)  \]
\end{theorem}

\paragraph{Johnson-Lindenstrauss (JL) lemma.} Our dimension reduction analysis uses ideas similar to a standard proof of JL. 
\begin{lemma}[Theorem 5.3.1 and Exercise 5.14 of \cite{vershynin2018high}]\label{lem:jl}
Let $x_1,\ldots,x_N \in \mathbb{R}^n$ be arbitrary points, and let $\epsilon > 0$. 
Let $A$ be an $m \times n$ random matrix with i.i.d. Rademacher entries (i.e., $Uni \{\pm 1\}$) and let $Q = A/\sqrt{m}$. Then if $m = \Omega(\epsilon^{-2} \log(N))$, it holds with probability at least $1 - 2\exp(-\Omega(\epsilon^2 m))$ that
\[ (1 - \epsilon)\|x_i - x_j\|_2 \le \|Q(x_i - x_j)\| \le (1 + \epsilon)\|x_i - x_j\| \]
for all $1\le i \le j \le N$.
\end{lemma}

\subsection{Algorithm and analysis} 
A key point is that we can estimate the inner products using dimension reduction (as in, e.g., \cite{arriaga2006algorithmic,chandrasekaran2024smoothed}). Let $A$ be an $m \times n$ random matrix with i.i.d. Rademacher entries (i.e. $Uni \{\pm 1\}$) and let $Q = A/\sqrt{m}$.
Note that by the Johnson-Lindenstrauss lemma (Theorem~\ref{lem:jl}), as long as $m \gg \log(k)$ then 
\begin{equation} \|Q w_i\|_2 \le 2 \label{eqn:Qw-bdd}
\end{equation}
for all $i \in [k]$ with probability at least $99\%$. We also have the following guarantees which cover large portions, but not the entirety, of the training set.
\begin{proposition}\label{prop:most-dimred-good}
Suppose that  $m \gg \log(k/\epsilon)/\gamma^2$, then with probability at least $99\%$ over the randomness of $Q$, $1 - \epsilon$ proportion of the points $Q x$ satisfy that for all $i \in [K]$
\begin{equation}\label{eqn:dimred-inner}
|\langle Q x, Q w_i \rangle - \langle x, w_i \rangle| \le \gamma/10.
\end{equation}
Furthermore, if $\xi \sim N(0,\Sigma)$ is independent noise for $\Sigma \in \mathbb{R}^{K \times K}$ a covariance matrix such that for all $i \in [K]$,  $\Sigma_{ii}^2 \le \sigma^2 := \gamma^2/100\log(2K/\epsilon)$, then under the same event on $Q$, we have that for $1 - \epsilon$ proportion of the points $Q x$ that
\[ \Pr(\forall i \in [K] : |\langle Q x, Q w_i \rangle + \xi - \langle x, w_i \rangle| \le \gamma/5 \mid Q) \ge 1 - \epsilon. \]
\end{proposition}
\begin{proof}
First we prove the result without the added Gaussian noise.
For any $x \in \mathcal X$, $i \in [m]$, and $k \in [K]$,
\[ \left(\sum_j Q_{ij} x_j\right) \left(\sum_j Q_{ij} (w_k)_j\right) =  \sum_{j,k} Q_{ij} x_j Q_{1k} w_k  \]
is a sub-exponential random variable with mean $\langle x, w_i \rangle$ and $\|\cdot\|_{\psi_1}$-norm $O(1)$ (by the fact that a product of sub-Gaussians is sub-exponential from Lemma~\ref{lem:sube_center}). So combined with the centering inequality (Lemma~\ref{lem:sube_center}) and Bernstein's inequality for sub-exponential  random variables (Theorem~\ref{thm:bernstein-subexp}), we have that for a fixed $x$ and $k \in [K]$
\[ \Pr\left(\left|\langle Q x, Q w_k \rangle  - \langle x, w_k \rangle\right| \ge t\right) \le 2\exp(-\Omega(\min(t^2 m,t m))) \]
Hence taking $t = \gamma/10$ and under our condition on $m$,  we can ensure by union bounding over $k \in [K]$ and applying linearity of expectation that
\[ \frac{1}{|X|} \mathbb E \#\{ x \in \mathcal X : \eqref{eqn:dimred-inner} \text{ does not hold} \} < 0.0001 \epsilon.  \]
The first conclusion then follows from Markov's inequality.

Now, with the additional Gaussian noise we have by the union bound and standard Gaussian tail bound that
\[ \Pr(\max_i |\xi_i| \ge t) \le 2K\exp(-t^2/\sigma^2) = \epsilon \]
where we again took $t = \gamma/10$, so the second conclusion follows from the first one by applying the triangle inequality and the union bound. 
\end{proof}
\begin{proposition}\label{prop:most-qx-bdd}
Suppose that $m \gg \log(1/\epsilon)$. Then with
 with probability at least $99\%$ over the randomness of $Q$, $1 - \epsilon$ of the $\|Q x\|_2$ are at most $2$.
\end{proposition}

Therefore, as long as we satisfy the conditions on $m$ above, with probability at least $97\%$ over $Q$, the events corresponding to Equation~\eqref{eqn:Qw-bdd}, Proposition~\ref{prop:most-dimred-good}, and Proposition~\ref{prop:most-qx-bdd} hold simultaneously. Note that under these three events, we can classify $1 - 2\epsilon$ fraction of points successfully based on their $Q$-embeddings by computing  $\mathrm{sign}(\langle Q x, Q w_i \rangle)$ for all $i \in [K]$. 

Also, from the second guarantee of Proposition B.2 and the fact that $\|Q w_i\|_2 \le 2$ under \eqref{eqn:Qw-bdd}, with probability at least $95\%$ if we also add independent Gaussian noise $N(0,\gamma^2/\log(k/\epsilon) I)$ to $Q x$, then in expectation over the randomness of the Gaussian noise\footnote{Note that this Gaussian noise is never actually sampled in the algorithm --- it is only used for the purpose of analysis.} , we can classify at least $1 - 3\epsilon$ fraction of points successfully based on their noised embedding. This last observation is useful because it means we can apply polynomial approximation results which hold in the smoothed analysis of learning setting. 

This leads us to the following learning algorithm (which we will combine with a standard reduction to boost the probability of success):
\begin{enumerate}
    \item Sample $Q = A/\sqrt{m}$ as above. Here $m = \Theta(\log(k/\epsilon)/\gamma^2)$.
    \item Restrict the training set to those $x$ with $\|Q x\|_2 \le 2$.
    \item Run $L_1$ polynomial regression of degree $D$ on the remaining dataset $(Q x_i, y_i)_i$ to obtain polynomial $f_D$.
    \item Output $x \mapsto \mathrm{sign}(f_D(\langle Q x, y))$ as the final classifier. 
\end{enumerate}
\begin{proof}[Proof of Theorem~\ref{thm:khalfspace}]
By Theorem~\ref{thm:poly-regression-strictly-subexp} it we let $\sigma = \gamma/\sqrt{\log(k/\epsilon)}$ and set
\[ D = \Theta\left(\sigma^{-1}(\log(1/\epsilon) + m\log(m\log(1/\epsilon)))\right) = \Theta\left(\log^{1.5}(2K/\epsilon)\log(\log(K/\epsilon)/\gamma^2)/\gamma^3 \right)\]
then the resulting classifier with 
\[ \Theta(m^D) =  e^{\Theta\left(\log^{1.5}(2K/\epsilon)\log^2(\log(K/\epsilon)/\gamma^2)/\gamma^3 \right)}\]
samples learns successfully a classifier with accuracy $\mathrm{OPT} - O(\epsilon)$ and final success probability at least $90\%$. The runtime is dominated by the polynomial regression step which takes time $poly(m^D)$. 

Using the standard reduction\footnote{I.e., splitting the training set into several folds, running the algorithm on each fold, and outputting the predictor which performs the best on a small holdout set. This is the standard replication trick as in e.g. \cite{kalai2008agnostically}.} to boost the success probability to $1 - \delta$ with factor $\log(1/\delta)$ samples, and defining $\epsilon = \varepsilon/3$, we obtain the  result.
\end{proof}

\section{Application: Polynomial approximation for general smooth functions} \label{app:learning-smooth-function}

In the classical approximation theory, any periodic $C^k$ function can be well approximated by trigonometric polynomials and algebraic polynomials on a bounded interval. 
This is a consequence of the polynomially decaying Fourier coefficients of the smooth target function.
In our framework of polynomial approximation, the decay of Fourier transform determines the "frequency cutoff" and further relates to the degree of approximating polynomials in the distributional sense.
Remarkably, our theory extends this classical result on a bounded interval to the general unbounded space with exponentially concentrated distribution, where the target function is not necessarily periodic.  
Formally, we consider the following class of Lipschitz functions: 
\begin{align}
     \mathsf{Lip}^{k,M}(\RR) = \{ f \in C^{k}(\RR): \RR\to \RR \mid \big|f^{(k)}(x) \big|\le M,  \text{ for all } x \in \RR\}. \label{def:lip_kM}
\end{align}
For any function in this class, the Fourier transform is generally not guaranteed to decay in frequency. 
Therefore, we introduce some intermediate targets that approximate the original network well and are more amenable to Fourier analysis. 
Under exponentially concentrated input, we can truncate the tail of the input to focus on a bounded interval with controlled error.  
Based on this, Jackson's theorem \citep{devore1993constructive} asserts that we can approximate any periodic smooth function pointwisely with trigonometric polynomials over a bounded interval. 
Since trigonometric polynomials are compactly supported in the frequency domain, our theory directly implies an approximating polynomial in $L^2(\mu)$, thus leading to the following theorem. 

\begin{theorem}\label{thm:poly-approx-lip-network}  
Let $\mu$ be a distribution on $\RR^d$.  
For any target function $f\in \mathsf{Lip}^{k,M} (\RR)$, there exists a polynomial $p$ of degree 
\begin{enumerate}
    \item  $O\Big((M/\eps)^{r/k} \cdot \big(\log(M/\eps) + (1-r^{-1})k\log k \big)^{r-1} \Big)$ under $r$-strictly sub-exponential $\mu$; 
    \item $O\Big(\exp \big\{O\big( (M/\eps)^{1/k}\big)\big\}\Big)$ under sub-exponential $\mu$. 
\end{enumerate}
such that $\norm{f - p}_{\mu} \le \eps$.  
\end{theorem}
\begin{proof}
See \cref{proof:thm:poly-approx-lip-network}. 
\end{proof} 

Our proof first applies some polynomial modifications of order less than $k$ to the target function, in order to use Jackson's theorem on a bounded interval. 
The critical step of our pipeline is to approximate the resulting trigonometric polynomial using our Fourier-based polynomial approximation. 
Indeed, classical Jackson's theorem already yields an algebraic polynomial approximation on a bounded interval as a byproduct.  
However, if we use the approximating polynomial from Jackson's theorem directly, it tends to blow up outside of the interval of approximation which leads to large $L^2(\mu)$ error (see the end of Section~\ref{subsec:discuss}). 
On the flip side, trigonometric polynomials are bounded in absolute value and have a compactly supported Fourier transform, thus being an accessible target for truncation and Fourier-based approximation. 

Below, we discuss connections between \cref{thm:poly-approx-lip-network} and some existing literature on weighted polynomial approximation. In particular, this literature includes converse results showing the sharpness of our guarantees. 
We then provide a proof sketch that illustrates each step in proving \cref{thm:poly-approx-lip-network}, including the statements of some intermediate results, in \cref{subsec:sketch-proof-lip}.

\subsection{Connections to existing results}\label{subsec:discuss}
\cref{thm:poly-approx-lip-network} is connected to prior results in an elegant way. 
First, we note that when $r=1$ (i.e., the support of $\mu_x$ is bounded) \cref{thm:poly-approx-lip-network} is a direct consequence of Jackson's theorem for polynomial approximation on a bounded interval. 
Beyond the bounded case, researchers in weighted approximation theory have studied in some detail Jackson type approximation under explicit forms of strictly sub-exponential measures --- see the survey of \citep{lubinsky2007survey}. This literature also includes some nice converse results.
We give some context and comparison with our results below.

\paragraph{Setting for comparison.} In what follows, we consider the case of $k$ times differentiable function $f$ with $|f^{(k)}|\le M$. In the approximation theory literature, researchers also study Jackson-Favard inequalities where the $L_{\infty}$ constraint on the derivative $f^{(k)}$ is replaced by an $L_p$ bound, in which case the function does not have to be differentiable everywhere. Some notable subtleties of such results were studied in \cite{lubinsky2006weights}.  
For simplicity, we will stick to the setting of strict $k$-differentiability. More generally, researchers in approximation theory have extensively investigated appropriate notions of ``moduli of continuity'' which determine the rate of approximation --- see the survey \cite{lubinsky2007survey}.

\paragraph{Strictly sub-exponential case.} 
In the above setting, Corollary 3.2 in \cite{lubinsky2007survey} gives a Jackson type inequality for weights of the form $e^{-|x|^{\alpha}}$ with $\alpha > 1$:
\begin{align}\label{eqn:lubinsky-comparison}
\inf_{g\in \cP_{\le D} } \norm{(g-f)w_\alpha}_{L^p( \RR)} = O\Big(\big({D^{(1/\alpha)-1}}\big)^{k}\Big),
\end{align}
Here, the underlying weight is the Freud\footnote{The class of Freud weights consists of $W$ of the form $e^{-Q}$ where $Q$ is even, $Q'$ is strictly positive on $(0,\infty)$, and some additional technical conditions on the derivative of $Q$ are satisfied. See Definition 3.3 of \cite{lubinsky2007survey}. Many of the results known for $w_{\alpha}$ with $\alpha > 1$ in fact are also proven for the class of Freud weights. Note that $w_{1}$ is not in the class.}  type weight $w_\alpha(x)= e^{-|x|^{\alpha}}$ with $\alpha>1$, and we used that the $D$-th Mhaskar-Rakhmanov-Saff number for this weight is proportional to $D^{1/\alpha}$ (see Equation 3.5 in \cite{lubinsky2007survey}). 
To match notation with our results, Theorem~\ref{thm:poly-approx-lip-network} shows that the $L_2$-approximation error is upper bounded by $O(D^{-k/r})$ for an arbitrary $r$-strictly sub-exponential distribution, and the measure $w_\alpha$ is $r$-strictly sub-exponential with $r=\alpha/(\alpha -1)$. So we  achieve the same conclusion of \eqref{eqn:lubinsky-comparison} when $p = 2$ (equivalently, by Jensen's inequality, when $p \in [1,2]$).


\paragraph{Approximation under single vs double-tailed exponential measures.}  
Under the \emph{single-sided} exponential type density $w_\alpha (x) = x^{\alpha} e^{-x} \ind\{x\ge 0\}$ for $\alpha\ge 0$, \cite{joo1988answer} provides the following Jackson type inequality:
\begin{align}
     \inf_{g\in \cP_{\le D} } \norm{(g-f) w_{\alpha} }_{L^p ( \RR)} \le c_{p,r} D^{-k/2}.
\end{align}
for any $1\le p\le \infty$. 
This result was further improved by \cite{mastroianni2008vallee} to relax the condition on $\alpha$. 
This raises a natural question of whether a ${1}/{\mathrm{poly}(D)}$ rate approximation is possible for general sub-exponential measures. However, it turns out such a result does not even hold for the Laplace/double-tailed exponential distribution, as we will discuss next. 

Under the double-tailed exponential measure $d\mu(x) = e^{-|x|}dx$,
\cite{freud1978approximation} established that 
\begin{align}\label{eqn:freud-comparison2}
        \inf_{g\in \cP_{\le D} } \norm{(g-f)w_1}_{L^1(\RR)} &\lesssim  \frac{1}{\log^k(D) }. 
\end{align}
The results of \cite{lubinsky2006jackson} yield an extension of this $O(\log^{-k}(D))$ bound to the general case with $p\ge 1$, and
Remark 1 in \cite{ditzian1987polynomial} shows that this rate is tight. 
In comparison, applying the sub-exponential result in \cref{thm:poly-approx-lip-network} yields a $L^2(\mu)$ approximation guarantee of $O(\log^{-k}(D))$ that matches up to constants the $L_2$ version of \eqref{eqn:freud-comparison2}. See \cite{lubinsky2006jackson} and \cite{bizeul2025polynomial} for some more precise results about approximation under the double-tailed exponential measure.  

In summary, for many natural sub-exponential distributions the rates of approximation are truly of the form $1/\mathrm{polylog}(D)$, and our results match this dependence over \emph{all} sub-exponential distributions.

\paragraph{Comparison to strictly sub-exponential case in \cite{chandrasekaran2025learning}.} 
In the case of strictly sub-exponential distributions, the recent work \cite{chandrasekaran2025learning} used an argument based on Jackson's inequality to obtain a polynomial approximation result for Lipschitz functions.
In a bounded region, Jackson's theorem implies an algebraic polynomial approximation that uniformly approximates a Lipschitz function. 
In contrast to our approach, they directly truncate the approximating polynomials whose degree scales with the accuracy level $\eps$. 
In this vein, they require precisely adjusting the truncation level to balance the trade-off between the degree of the approximating polynomial and its tail moment. 
As a result, the degree of the polynomial to approximate one-dimensional Lipschitz functions in \cite{chandrasekaran2025learning}  
 is of order\footnote{See \cref{apdx:comparison_futher} for more details. } $\tO(\eps^{-(2r-1)})$ under $r$-strictly sub-exponential distribution. 
Using our technique, we are able to obtain $\tO(\eps^{-r})$ dependency under $r$-strictly sub-exponential distribution and extend the results to sub-exponential distributions. 
So our result yields a strictly better exponent whenever $r>1$. 




\subsection{Sketch of analysis}\label{subsec:sketch-proof-lip}

We begin with a polynomial modification of the target function $f$ that does not affect the polynomial approximation result of the original function. 
We define 
\begin{align}
    \tf(t) &= f(t) - \sum_{l=0}^k \frac{f^{(l)}(0)}{l!} t^l. \label{eq:taylor-jet}
\end{align}
Then, for any polynomial that approximates $\tf$ well, adding back the Taylor polynomial of degree $k$ at the origin yields a polynomial that approximates $f$ well. 
The modification only affects the $k$-th derivative by a constant, i.e.,  $\tf^{(k)}(t) = f^{(k)}(t) - f^{(k)}(0)$ for all $t\in \RR$. 
Therefore, $\tf \in \mathsf{Lip}^{k,2M}(\RR)$.  
On the other hand, we have that $\tf^{(l)} = 0$ for $l\le k$.

\subparagraph{Periodic extension.}
With some $R>0$ to be specified later, we can truncate $\tf$ to $[-R,R]$ and bound the truncation error properly. 
However, the truncated function does not extend to a periodic function on $\RR$, as its derivatives at $\pm R$ do not match in general.
This obstructs the use of Jackson's theorem. 
To address this issue, we introduce another polynomial that properly adjusts the boundary derivatives up to order $k$.  
\begin{lemma}[Bernoulli polynomial]\label{lem:bernoulli-smoothing} 
    For any fixed sequence $\{a_l\}_{l=0}^k \subset \RR$, there exists a polynomial $p_k:\RR\to \RR$ of degree $k+1$ such that $p_k^{(l)}(1) - p_k^{(l)}(0) = a_l$ for all $0\le l\le k$.
    Additionally, we have that $\sup_{x\in [0,1]}|p_k^{(k)}(x)| = |a_k|/2$ and  $|p_k(x)|\le (1+|x|)^k \sum_{0\le l\le k} \tfrac{2^{l+1} a_l}{(l+1)!}$ for any $x\in\RR$. 
\end{lemma}
\begin{proof}
See \cref{proof:lem:bernoulli-smoothing}.
\end{proof}

Given $\tf$, we choose $a_l = (2R)^{l}\big(\tf^{(l)}(R) - \tf^{(l)}(-R)\big)$ for $0\le l\le k$ and let $p_k$ be the corresponding polynomial in \cref{lem:bernoulli-smoothing}. 
Then we define
\begin{align}
    q_k(x):x\mapsto p_k\big((x+R)/(2R)\big), \label{eq:bern-qk}
\end{align}
which satisfies that 
\begin{align}
    q_k^{(l)}(R) -  q_k^{(l)}(-R) &= (2R)^{-l} \big(p_k^{(l)}(1) - p_k^{(l)}(0)\big)   = \tf^{(l)}(R) - \tf^{(l)}(-R), \text{ for }0\le l\le k \\ 
    \sup_{x\in [-R,R]}|q_k^{(k)}(x)| &= (2R)^{-k} \cdot \sup_{x\in [0,1]} |p_k^{(k)}(x)| =  |\tf^{(k)}(R) - \tf^{(k)}(-R)|/2 \le M. \label{eq:qk-kprime-bound}
\end{align}  
Therefore, the function $\check{f}(x)  = \tf(x) - q_k(x)$ follows that $\check{f}^{(l)}(-R) = \check{f}^{(l)}(R)$ for all $0\le l\le k$, and $|\check{f}^{(k)}|$ is bounded by $3M$ on $[-R,R]$.

\subparagraph{Approximation from trigonometric to algebraic polynomial.} 

A trigonometric polynomial of degree $D$ on $[0,2\pi]$ is defined as  
\begin{align}
    T_D(t) &= \sum_{m\le D} a_m \cos(mt) + b_m \sin(mt), \text{ for } t\in [0,2\pi].   
\end{align}
Standard result of $L^2$ Fourier analysis provides a trigonometric polynomial approximation for any smooth periodic function with Dirichlet kernel. 
In addition, the Fej\'er kernel extends the approximation to $L^p$ for any $1\le p<\infty$.  
However, they do not directly imply $L^2(\mu)$ approximation for a general probability distribution $\mu$ over $\RR$, since $\mu$ does not necessarily have a bounded density. 
Thus, we resort to Jackson's theorem for the $L^\infty$ case. 
When $p=\infty$, Jackson's theorem \cite{jackson1930theory} provides a classical result on the approximation rate of trigonometric polynomials for smooth periodic functions, as stated below.   

\begin{theorem}[Corollary 2.4 of \cite{devore1993constructive}]\label{thm:jackson} 
Fix any $D\ge 1$. 
For any $f\in C^k([0,2\pi])$ such that $f^{(l)}(0) = f^{(l)}(2\pi)$ for all $l\le k$ and $\norm{f^{(k)}}_\infty \le M$, there exists a trigonometric polynomial $T_D$ of degree $D$ on $[0,2\pi]$ such that 
\begin{align}
\norm{T_D - f}_{\infty} &\le   C_k  D^{-k} M , 
\end{align}  
where $\sup_k C_k < \infty$.
Moreover, we have that $\sum_{1\le m \le D} |a_m| +|b_m| \le 2^{k+2}M/\pi \, \sum_{1\le m \le D} m^{-k}$. 
\end{theorem} 
Since the coefficient bound above is not explicitly stated in the literature, we reprove it here for completeness. 
\begin{proof}
See \cref{proof:jackson-coefficients}.
\end{proof}

For any function $f \in C^k([-R,R])$ with $\norm{f^{(k)}}_\infty \le M$, we can rescale the target to get the following corollary.

\begin{corollary}\label{cor:jackson-rescale}
For any $f\in C^k([-R,R])$ such that $f^{(l)}(-R) = f^{(l)}(R)$ for all $l\le k$ and $\norm{f^{(k)}}_\infty \le M$, there exists a trigonometric polynomial of the form 
\begin{align}
    T_D(x) &= \sum_{m\le D} a_m \cos\Big(\frac{m\pi (x+R)}{R}\Big) + b_m \sin\Big(\frac{m\pi (x+R)}{R}\Big), \text{ for } x\in [-R,R],  
\end{align}
such that $\norm{T_D - f}_\infty \le C_k D^{-k} M (R/\pi)^k$. 
Additionally, it holds that 
$$\sum_{1\le m\le D} |a_m| + |b_m| \le  \frac{2^{k+2}\,MR^k}{\pi^{k+1}} \sum_{1\le m \le D} m^{-k}.$$  
\end{corollary}

\begin{proof}
Note that the function $g(t)=f(Rt/\pi - R)$ for $t\in [0,2\pi]$ satisfies all the conditions in \cref{thm:jackson} with $\norm{g^{(k)}}_\infty \le M (R/\pi)^k$. 
Then \cref{thm:jackson} yields that there exists a trigonometric polynomial $\tilde{T}_D$ of degree $D$ on $[0,2\pi]$ such that $\norm{\tilde{T}_D - g}_\infty \le C_k D^{-k} M (R/\pi)^k$.  
Defining $T_D(x) = \tilde{T}_D(\pi (x+R)/R)$ for $x\in [-R,R]$ completes the proof.  
\end{proof}

We use our theory of Fourier transform to approximate the trigonometric polynomial $T_D$ with algebraic polynomials in $L^2(\mu)$ norm. 
\begin{lemma}\label{lem:nn-trig-poly-approx}
Let $\mu$ be any probability distribution over $\RR$. 
Fix any $\eps>0$ and $R>0$. 
Suppose that $T_{D_0}$ is a trigonometric polynomial on $\RR$ of the form 
\begin{align}
    T_{D_0}(x) &= \sum_{m\le D_0} a_m \cos\Big(\frac{m\pi (x+R)}{R}\Big) + b_m \sin\Big(\frac{m\pi (x+R)}{R}\Big). 
\end{align}
Let $C_{D_0} = \sum_{1\le m\le D_0} |a_m| + |b_m|$. 
Then there exists an algebraic polynomial $p$ on $\RR$ with 
\begin{enumerate}
    \item degree $O\big(\log(C_{D_0} / \eps)^r \vee  (D_0 / R)\big)^r$ for $r$-strictly sub-exponential $\mu$;  
    \item degree $O\big(\exp\{O(D_0 /R)\} \cdot \log(C_{D_0}/\eps)\big)$ for sub-exponential $\mu$, 
\end{enumerate}
such that $\norm{T_{D_0} - p}_{\mu} \le \eps$. 
\end{lemma}

\begin{proof}
See \cref{proof:nn-trig-poly-approx}.
\end{proof}

\subparagraph{Combined polynomial approximation bound.} 
We now combine these results to approximate $f \in \mathsf{Lip}^{k,M}(\RR)$ using algebraic polynomials in $L^2(\mu)$ norm. 
Consider a truncation level $R>0$ to be specified later.  
For given $\tf$, \cref{lem:bernoulli-smoothing} modifies it to $\tf - q_k$ that satisfies the boundary condition in \cref{cor:jackson-rescale}. 
Then, \cref{cor:jackson-rescale} implies that there exists a trigonometric polynomial $T_D$ such that $\sup_{x\in[ -R,R]}|\tf - q_k - T_D|\le \eps$.   
Next, \cref{lem:nn-trig-poly-approx} implies that there exists a polynomial $p_{T}$ of appropriate degree such that $\norm{p_T - T_D}_\mu \le \eps$. 
We consider the final polynomial approximator as $p = p_T + q_k$, whose approximation error can be bounded as 
\begin{align}
    \norm{ \tf  - q_k - p_T}_\mu &\le \bignorm{(\tf - q_k - p_T)\cdot \ind\{|x|\le R\}}_\mu + \bignorm{(\tf - q_k - p_T)\cdot \ind\{|x|> R\}}_\mu \\ 
    &\le  \sup_{|x|\le R} |\tf(x) - q_k(x) - p_T(x)| + \bignorm{(\tf - q_k -T_D +T_D - p_T )\cdot \ind\{|x|> R\}}_\mu\\ 
    &\le \eps + \bignorm{(\tf - q_k)\cdot \ind\{|x|> R\}}_\mu + \bignorm{(T_D - p_T)\cdot \ind\{|x|> R\}}_\mu +\norm{T_D\ind\{|x|>R\}}  _\mu. \label{eq:three-terms}
\end{align}
Now it suffices to bound the last three terms. 
For the second term, the tail of $q_k$ is of the same order as $\tf$, as suggested by \cref{lem:bernoulli-smoothing}.
Then, we can use the Lipschitz property of $\tf$ with the concentration property of $\mu$ to bound the tail moment. 

On the other hand, the third term is well bounded because \cref{lem:nn-trig-poly-approx} ensures that 
\begin{align}
    \bignorm{(T_D - p_T)\cdot \ind\{|x|> R\}}_\mu &\le \norm{T_D - p_T}_\mu \le \eps.  
\end{align}
Finally, the fourth term $\norm{T_D}_\mu$ can be bounded using the boundedness of the trigonometric polynomial. 
Combining these bounds, we derive the desired result on polynomial approximation of $f$.  


\section{Application: Agnostically learning neural networks} \label{sec:app-learning-nn}
Beyond Gaussian-smoothed bounded functions, our theory also extends to target functions with structural assumptions. 
A typical example is the class of one-layer neural networks with explicitly specified activation functions. 
For explicit functions whose Fourier transform is typically well-understood and decays sufficiently, our theory provides a characterization of their polynomial approximability, thus enabling efficient learning algorithms.
Formally, our notion of learnability is defined as follows.

\begin{definition}[Agnostic learning]
Fix $\eps>0$ and $\delta\in (0,1)$. 
Let $\ell:\RR\times\RR\to \RR^+$ be a loss function.
Suppose that $\mu$ is a distribution over $\RR^d \times \RR$ and $\cH$ is a hypothesis class where each element $h\in \cH$ is a function from $\RR^d$ to $\RR$.  
And let $\cD = \{(x_i,y_i)\}_{i\in [n]}$ be a dataset with i.i.d. samples drawn from $\mu$. 
We say that an algorithm $\cA$ that takes $\cD$ as input and outputs a hypothesis $f:\RR^d\to \RR$ learns $\cH$ agnostically, if with probability at least $1-\delta$ over the draw of $\cD$, it holds that
\begin{align}
    \EE_{\mu}\big[\ell(y,h(x))\big] &\le \mathrm{opt}_{\cH} + \eps, \text{ where } \mathrm{opt}_{\cH} = \inf_{f\in \cH} \EE_{\mu}\big[\ell(y,f(x))\big].  
\end{align}
\end{definition}
Our target in this section is to agnostically learn the class of one-layer neural networks with various activation functions and distributions. 
Suppose that $\sigma:\RR\to \RR$ is an activation function.
We want to learn the class of the one-layer neural networks with an activation function $\sigma$, i.e.
\begin{align} 
    \cH_{\sigma} = \{ \psi: \RR^d \to \RR \mid \psi(x) = \frac{1}{k} \sum_{j=1}^k \sigma (w_j^\top x),\; w_j \in \SS^{d-1},j\le k \}.    
\end{align} 
The target in this section is to identify the minimal sample complexity needed to approximate functions in $\cH_{\sigma}$ within any desired accuracy. 
We approach this problem using the polynomial approximation framework established in \cref{sec:Rd-analytic-lemmas}. 
For specific forms of activation functions, we can prove that any function in $\cH_\sigma$ can be approximated by polynomials using the results in \cref{sec:Rd-analytic-lemmas} in various distributions. 



Provided the existence of an approximating polynomial, we show that polynomial regression is able to find a hypothesis that is close to the best approximating polynomial, thus learning $\cH_\sigma$ with desired guarantees. 
This step relies on bounding the generalization error of polynomial regression under the distribution $\mu$ that potentially incurs heavy-tailed errors,. 
We briefly discuss this step in a lemma below.
Let $\phi_D:\RR^d \to \RR^m$ be the monomial feature mapping of degree $D$ and $m = {\binom{d+D} {D}}$. 
We consider the following empirical risk minimization problem: 
\begin{align}
    \hatw=\argmin_{w\in \RR^m} \frac{1}{n} \sum_{i=1}^n \ell(y, \phi_D(x_i)^\top w),
\end{align}
where we assume that $\ell(y,y') = l(|y-y'|)$, where $l$ is non-negative and non-decreasing in $|y-y'|$, and 2-pseudo-Lipschitz, in the sense that  $|l(y) - l(y')| \le  \mathfrak{L} \,\max\{|y|,|y' |,  1\}\cdot | y - y'|$ for some constant $\mathfrak{L}>0$.   
This loss function encapsulates many common losses such as the squared loss and the absolute loss. 
Next lemma states that the polynomial $x\mapsto \phi_D(x)^\top \hatw$ achieves the desired accuracy with proper truncation.  
\begin{lemma}\label{lem:gen-error-polynomial-regression}
Suppose that $\{(x_i,y_i)\}_{i\le n}$ is an i.i.d. dataset drawn from the distribution $\mu$, such that $\EE_\mu[y^4]<\infty$ and $\EE_\mu[|x^\alpha|]<\infty$ for all multi-index $\alpha$. 
Let $\cH$ be a class of functions from $\RR^d$ to $\RR$ such that any $f\in \cH$ is in $L^2(\mu)$.  
Suppose that for every $f\in \cH$, there is a polynomial $p$ of degree $D$ such that $\EE[|f(x) - p(x)|^2]^{1/2} \le \eps$. 
Then there exists an algorithm that takes $O\Big( \eps^{-3} d^{D} \log(1/\eps) \cdot \log(1/\delta)\Big)$ samples and computes a hypothesis $\hath$ in $\mathrm{poly}(d,D,1/\eps)$ time such that, with probability at least $1-\delta$ over the dataset, it holds that $\EE[\ell (y, \hat h(x))] \le \mathrm{opt}_\cH + \eps$, where $\mathrm{opt}_\cH = \inf_{f\in \cH} \EE[\ell(y,f(x))]$. 
\end{lemma}

\begin{proof}
    See \cref{proof:gen-error-polynomial-regression}.  
\end{proof}


\newcommand{\sgm}{\mathsf{sigmoid}}
\newcommand{\relu}{\mathsf{ReLU}}
\newcommand{\pv}{\mathrm{p.v.}\,}
\newcommand{\fp}{\mathrm{f.p.}\,}

\paragraph{Learning Sigmoid networks} \label{sec:app-learning-nn-sigmoid}  
We begin with the setting of learning $\cH_\sigma$ with $\sigma(t) = \sgm(t) = {(1 + e^{-t})^{-1}}$. 
This function is smooth and real-analytic, and corresponding Fourier transform decays, as suggested by the following lemma.
\begin{lemma}\label{lem:sigmoid-fourier}
    The Fourier transform of $\sgm$ is given by the tempered distribution such that for every test function $\psi \in \cS(\RR)$,
    \begin{align}
        \dotp{\widehat{\sgm}}{\psi}  &= \pi \psi(0) +  \pv  \int_\RR \psi(\xi)\frac{\pi }{i\sinh(\pi\xi)}\, d\xi , 
    \end{align}
\end{lemma}
\begin{proof}
See \cref{proof:lem:sigmoid-fourier}.
\end{proof}
Notably, for sufficiently large $|\xi|$, the Fourier transform of the sigmoid function decays exponentially. 
That indicates that we can choose a logarithmic frequency cutoff in \cref{lem:fourier-apx-bound}.
This concludes the degree of the approximating polynomials. 
However, the proof of the polynomial approximation result for sigmoid function requires careful handling of the singularity at zero, which necessitates the more general characterization in \cref{lem:fourier-rep-fs}.  
As will be shown in the proof, the singularity at zero incurs first-order differentiation on the test functions within a bounded region, which does not affect the choice of frequency cutoff. 
In summary, this leads to the following polynomial approximation result. 
\begin{lemma} \label{lem:poly-approx-sigmoid-network}
Suppose that $\mu$ is a probability  distribution over $\RR^d$. 
Then, for any $f\in \cH_{\text{sigmoid}}$, there exists a polynomial $p$ of
\begin{enumerate}
    \item  degree $O\big(\log(1/\eps )^r\big)$ under $r$-strictly sub-exponential distribution $\mu$; 
    \item degree $O\big( \eps^{-O(1)}\big)$ under sub-exponential distribution $\mu$,
\end{enumerate}
such that $\|f - p\|_{\mu} \leq \eps$. 
\end{lemma}

\begin{proof}
See \cref{proof:lem:poly-approx-sigmoid-network}
\end{proof}

In combination with the generalization bound in \cref{lem:gen-error-polynomial-regression}, we obtain the following theorem on learning sigmoid networks.  

\begin{theorem}[Learning Sigmoid networks]\label{thm:sigmoid_samples} 
Fix $\eps>0$ and $\delta\in(0,1)$. 
Let $\ell(y,y') = (y-y')^2$ be the quadratic loss function.
Suppose that the $y$-marginal of $\mu$ has a finite fourth moment, i.e., $\EE[|y|^4] <\infty$.
Then there exists an algorithm that agnostically learns $\cH_\relu$ agnostically with accuracy $\eps$ and confidence $1-\delta$ using i.i.d. samples of size 
\begin{enumerate}
    \item \({O}\left(\eps^{-3} d^{O(\log(1/\eps)^r)} \log(1/\delta)\right)\) under \(r\)-strictly sub-exponential input $\mu_x$;
    \item \({O}\left(\eps^{-3} d^{O(\eps^{-O(1)})} \log(1/\delta)\right)\) under sub-exponential input $\mu_x$.
\end{enumerate}
\end{theorem}

\paragraph{Learning ReLU networks}

Similar to the sigmoid case, we analyze the Fourier transform of the ReLU activation function. 
\begin{lemma}\label{lem:relu-fourier}
    The Fourier transform of $\relu(t) = \max\{0,t\}$ is given by the tempered distribution such that for every test function $\psi \in \cS(\RR)$,
\begin{align}
    \dotp{\widehat{\relu}}{\psi } =  \dotp{\mathrm{f.p.}\,\frac{-1}{\xi^2}}{\psi} +   i\pi \psi'(0).
\end{align}
Here $\mathrm{f.p.}\, \frac{1}{\xi^2}$ is the Hadamard finite part distribution defined as
\begin{align}
    \dotp{\mathrm{f.p.}\,\frac{1}{\xi^2}}{\psi} = \lim_{\eps \to 0^+} \int_{|\xi|>\eps} \frac{\psi(\xi) - \psi(0)}{\xi^2} \, d\xi.
\end{align}
\end{lemma}

\begin{proof}
See \cref{proof:lem:relu-fourier}. 
\end{proof}

Similarly, we obtain the following polynomial approximation results. 

\begin{proposition}\label{prop:poly-approx-relu-network}
    Suppose that $\mu$ is a probability distribution over $\RR^d$. Then for any $f\in \cH_{\relu}$, there exists a polynomial $p$ of 
    \begin{enumerate}
        \item degree $O(\eps^{-r})$ under $r$-strictly sub-exponential input $\mu$;
        \item degree $O(\exp\{O(\eps^{-1})\})$ under sub-exponential input $\mu$, 
    \end{enumerate}
    such that $\norm{f - p}_\mu\le \eps$. 
\end{proposition}
\begin{proof}
See \cref{proof:prop:poly-approx-relu-network}. 
\end{proof}

Indeed, $\relu(t) = \max\{t,0\}$  belongs to the class of $\mathsf{Lip}^{1,1} (\RR)$.
Therefore, we can directly apply \cref{thm:poly-approx-lip-network} and obtain the same results, albeit with an additional logarithmic factor. 
Roughly speaking, the non-smoothness of the ReLU function at zero leads to a slowly decaying pattern in the frequency domain.
This indicates that the ideal frequency cutoff scales $\mathrm{poly}(1/\eps)$, which leads to significantly larger degree in the approximating polynomials.  


Combining the results above, we obtain the following theorem.
\begin{theorem}[Learning ReLU networks]\label{thm:relu_samples} 
Fix $\eps>0$ and $\delta\in(0,1)$. 
Let $\ell(y,y') = (y-y')^2$ be the quadratic loss function.
Suppose that the $y$-marginal of $\mu$ has a finite fourth moment, i.e., $\EE[|y|^4] <\infty$. 
Then there exists an algorithm that agnostically learns $\cH_\relu$ with accuracy $\eps$ and confidence $1-\delta$ using i.i.d. samples of size 
\begin{enumerate}
    \item \(\tilde{O}\left(\eps^{-3} d^{O((1/\eps)^r)} \log(1/\delta)\right)\) under \(r\)-strictly sub-exponential input $\mu_x$;
    \item \(\tilde{O}\left(\eps^{-3} d^{O(\exp{O(\eps^{-1})})} \log(1/\delta)\right)\) under sub-exponential input $\mu_x$. lip
\end{enumerate}
\end{theorem}

\paragraph{Learning general Lipschitz functions.}
In this part, we consider the activation function $\sigma$ from the general Lipschitz class  $\mathsf{Lip}^{k,M}(\RR)$, which is defined in \cref{def:lip_kM}. 
We directly state our main result, which is a direct implication of \cref{thm:poly-approx-lip-network}.

\begin{theorem}[Learning general Lipschitz networks]\label{thm:lip_samples} 
Fix $\eps>0$ and $\delta\in(0,1)$. 
Let $\sigma \in \mathsf{Lip}^{k,M}(\RR)$. 
Suppose that the $y$-marginal of $\mu$ has a finite fourth moment, i.e., $\EE[|y|^4] <\infty$. 
Let $\ell(y,y') = (y-y')^2$ be the quadratic loss function.
Then there exists an algorithm that agnostically learns $\cH_\sigma$ with accuracy $\eps$ and confidence $1-\delta$ using i.i.d. samples of size 
\begin{enumerate}
    \item $O\Big(\eps^{-3} d^{ O\left((M/\eps)^{r/k} \cdot \left(k\log k \right)^{r-1}\right)} \log(1/\delta)\Big)$ under \(r\)-strictly sub-exponential input $\mu_x$;
    \item $O\Big(\eps^{-3} d^{O\big(\exp \big\{O\big( (M/\eps)^{1/k}\big)\big\}\big)} \log(1/\delta)\Big)$ under sub-exponential input $\mu_x$.
\end{enumerate}
\end{theorem}

\begin{proof}
    See \cref{proof:thm:lip_samples}.
\end{proof}

\section{Nearly matching crytographic lower bounds for learning}\label{sec:crypto}

One way to think about our main result is that, based on existing SQ and cryptographic lower bounds, we already believe in some cases that ``low frequency'' of the target class is necessary for agnostic learning in high dimensions, so we prove related sufficiency results. In this section, we discuss more explicitly the relationship between our algorithmic results and two well-known conjectured hard problems: learning sparse parities with noise (SPWN), and learning with errors (LWE). 

\subsection{Boolean space and SPWN}
The problem of \emph{learning a $k$-sparse parity with noise} ($k$-SPWN) is a classical problem in learning, cryptography, and average-case complexity. It is usually phrased as the following hypothesis testing question.
\begin{itemize}
\item Let $\alpha \in (0,1)$ be a fixed (dimension-independent) constant and $k \ge 1$ an integer parameter.

\item The null hypothesis $H_0$ is that $(X_i,Y_i) \sim Uni \{\pm 1\}^{n + 1}$.

\item The alternative hypothesis $H_1$ is that for some unknown set $S \subset [n]$, we have i.i.d. samples $(X_i,Y_i)$ following $X_i \sim Uni \{\pm 1\}^n$ and an unknown set $S \subset [n]$,
\[ E[Y_i \mid X_i] = \alpha \prod_{j \in S} X_{ij} \]
where $Y_i$ is still valued in $\{\pm 1\}$.

\item The (simple-vs-composite) hypothesis testing problem is to distinguish whether a given set of samples is drawn from $H_0$ or from $H_1$.
\end{itemize}

\paragraph{Conjectured hardness: } Consider an algorithm with access to an unbounded number of examples, but limited running time. Using brute force, it is straightforward to prove there is an algorithm to recover $S$ with $\Theta(n^{k})$ samples. 
Currently there is no algorithm which can learn a sparse parity with noise in time $n^{o(k)}$ when $k = o(n)$. See \cite{valiant2015finding} for the current best algorithm, which runs in time $n^{ck}$ for $c < 1$, so it is better than brute force but does not contradict the conjecture that $n^{o(k)}$ time is impossible.
More precisely, the conjectured hardness of learning $k$-SPWN says that there is no algorithm which runs in time $n^{o(k)}$ and can test between $H_0$ and $H_1$ with sum of Type I and Type II error $0.9$ (or any number significantly less than $1$). 

\begin{remark}
Note that the assumption $\alpha < 1$ means there is a nontrivial amount of noise. If there is no noise, the analogous conjecture is false because the problem is solvable by using Gaussian elimination over $\mathbb F_2$.
\end{remark}

\paragraph{Single neuron interpretation: } When $k$ is even we can equivalently write the sparse parity with noise model $H_1$ as
\[ \E[Y_i \mid X_i] = \alpha \cos\left(\frac{\pi}{2} \sum_{j \in S} X_{ij} \right)\]
and when $k$ is odd similarly
\[ \E[Y_i \mid X_i] = \alpha \sin\left(\frac{\pi}{2} \sum_{j \in S} X_{ij}\right). \]
Since
\[ \sum_{i \in S} X_i = \sqrt{k} \langle 1_S/\sqrt{k}, X_i \rangle \]
the frequency along the signal direction in either case is $\gamma = \frac{\pi \sqrt{k}}{2}$. So the conjectured hardness for $k$-SPWN corresponds to the conjecture that there is \emph{no algorithm running in time $n^{o(\gamma^2)}$}.

\paragraph{Comparison to our results: } The uniform distribution on the hypercube is $1$-sub-Gaussian, so taking $r = 2$ and $\Omega = \gamma$ in Theorem~\ref{thm:pw-strict}, we see that by taking $D = \Theta_{\alpha}(\gamma^2)$ that we can find a degree $D$ polynomial within $L_2$ distance $\alpha/100$ of the true regression function $\alpha \cos\left(\frac{\pi}{2} \sum_{j \in S} X_{ij} \right)$, and then we can conclude that by performing polynomial regression we can solve the distinguishing problem with high probability in time $n^{O(D)}$. By basic Fourier analysis over $\{\pm 1\}^n$ (see, e.g., \cite{ODonnell2014}) the polynomial approximation result is tight up to constant factors --- no polynomial of degree less than $k$ has any correlation with a parity of size $k$. Under the conjecture hardness of learning sparse parities with noise, our algorithmic guarantee is also optimal up to the constant in the exponent.
\begin{remark}
Viewed as regression with a cosine neuron, the sparse parity with noise model is a ``well-specified'' model since the observed $Y_i$ differs from $\E[Y_i \mid X_i] = \alpha \cos\left(\frac{\pi}{2} \sum_{j \in S} X_{ij} \right)$ only by the mean-zero noise induced by sampling the label from the conditional law. So while the upper bound applies in the general agnostic learning setting, the lower bound holds even in this relatively nice setting where the conditional mean is realizable (also known as the probabilistic concept model \cite{kearns1994efficient}).
\end{remark}
\subsection{Gaussian space and LWE}
\paragraph{CLWE.} The following distribution was introduced in \cite{bruna2021continuous} as a continuous analogue of the well-studied \emph{Learning with Errors (LWE)} problem in cryptography. 

Let $\gamma, \beta \in \mathbb R$ and let $\mathcal S$ be a distribution over unit vectors in $\R^n$. The $CLWE(m,\mathcal S,\gamma,\beta)$ distribution is given by sampling $A_1,\ldots,A_m \sim N(0,I_n)$ independently, sampling $w \sim \mathcal S$, $E_1,\ldots,E_n \sim N(0,\beta^2)$ and outputting
\[ (A_i, B_i = \gamma \langle A_i, w \rangle + E_i \mod 1)_{i = 1}^m. \]
Note that $w$ is a cryptographic ``secret'' --- it is not directly revealed as part of the dataset. 

The search version of the CLWE problem asks for an algorithm to recover $w$ from samples, whereas the \emph{decision} version of CLWE asks for an algorithm which can distinguish samples from the CLWE distribution and from the null distribution where $B_i \sim Uni(0,1)$. 

It was observed in \cite{song2021cryptographic} that the hardness of CLWE (which was proven under quantum reduction from lattice problems in \cite{bruna2021continuous}) leads to hardness results for learning periodic neurons. Below we discuss improved quantitative results which follow from the more recent classical reduction from LWE to CLWE \cite{gupte2022continuous}. 

\paragraph{LWE and conjectured hardness.} LWE is the original discrete variant of the above problem. It is a very well-studied problem in cryptography since, unlike many other average-case problems, there is an explicit reduction from LWE to worst-case lattice problems, and these worst-case problems are believed to be hard for both classical and quantum computers. We do not need to use the precise definition of LWE here, but  it is analogously defined
over a ring $\mathbb Z/q\mathbb Z$, where the data is given by pairs $(a_i,b_i \approx \langle w, x \rangle + e_i \mod q)_{i = 1}^m$ with $a_i \sim Uni (\mathbb Z/q\mathbb Z)^n$, and $e_i$ is discrete noise (whose size is parameterized by a parameter $\sigma$). See \cite{bruna2021continuous,gupte2022continuous} for much more extensive discussion and references.

\begin{theorem}[Corollary 2 of \cite{gupte2022continuous}]
Suppose that $k,n,\ell,q$ are integers with $k\log_2(n/k) \gg \ell \log_2(q)$. 
There is a $poly(n)$ time reduction from LWE in dimension $\ell$ with $n$ samples, modulus $q$, and noise parameter $\sigma \gg 1$ to CLWE in dimension $n$ with $\gamma = O(\sqrt{k \log n})$, $\beta = O\left(\sigma\sqrt{k}/q\right)$ with $k$-sparse secret distribution $\mathcal S$.
\end{theorem}
As discussed after the statement of Corollary 2 of \cite{gupte2022continuous}, it is believed that LWE in dimension $\ell$ with modulus $\sigma = poly(\ell)$ and $q/\sigma = poly(\ell)$ cannot be solved by algorithms which run in time $2^{\ell^{1 - \delta}}$ (and have access to $m = 2^{\ell^{1 - \delta}}$ samples) for any $\delta > 0$. 

Taking $\gamma = \Theta(\sqrt{k \log(n)})$, $q, \sigma,$ and $k$ to be appropriate polynomials of $\ell$ so that $\sigma\sqrt{k}/q = O(1)$, and assuming $\log_2(n) \gg \log_2(k)$, the condition in the theorem can be rewritten as $\gamma^2 \gg \ell \log_2(\ell)$. So under the believed hardness of LWE which we just described, CLWE in dimension $\ell$ with (e.g.) $\beta = 0.001$ and $\gamma,\ell$ related as above takes time at least $\ell^{\gamma^{2 - 4\delta}}$ for any fixed $\delta > 0$.

\paragraph{Single neuron interpretation.} Because the map $x \mapsto e^{2\pi ix} = \cos(2\pi x) + i\sin(2\pi x)$ is a bijection for $x \in [0,1)$ and $1$-periodic, the dataset in CLWE can be equivalently encoded as samples $(A_i, \tilde{C}_i, \tilde{S}_i)_{i = 1}^m$ where 
\[ \tilde{C}_i = \cos(B_i) = \cos(\gamma\langle A_i, w \rangle + E_i) \]
and likewise
\[ \tilde{S}_i = \sin(B_i) = \sin(\gamma\langle A_i, w \rangle + E_i) \]

\paragraph{Comparison to our results.} In the hardness regime of LWE/CLWE outlined above with $\beta = 0.001$, we can solve the CLWE decision problem in a straightforward way via polynomial regression of degree $D$ with $D = O(\gamma^2)$, so $\ell^{O(\gamma^2)}$ runtime. So conjecturally, this algorithm has optimal runtime up to $\gamma^{o(1)}$ factors in the exponent.


\section{Preliminaries for universality}\label{sec:prelim3}
See \cite{janssen1990spaces,folland1999real,gel2013spaces,rudin1987real} for more detailed references.

\subsection{Paley--Wiener theorem and exponential--type growth}

Let
\[
  PW_{\Omega}
  \;:=\;
  \bigl\{\,f\in L^{2}(\R):
          \operatorname{supp}\widehat{f}\subset[-\Omega,\Omega]
    \bigr\},
  \qquad
  \Omega>0.
\]
where $\hat f$ is the Fourier transform of $f$.
Then

\begin{theorem}[Paley--Wiener \cite{rudin1987real}]\label{thm:PW}
For every $f\in PW_{\Omega}$ the Fourier inversion formula gives an
\emph{entire} extension
\(
  F(z):=\frac{1}{2\pi}\int_{-\Omega}^{\Omega}
         \widehat{f}(\xi)\,e^{i\xi z}\,d\xi
\)
satisfying the exponential--type estimate
\begin{equation}\label{eq:PW-growth}
  |F(z)|
  \;\le\;
  C\,e^{\Omega|\,\Im z|},
  \qquad z\in\C,
\end{equation}
for some constant $C$ depending only on $\|f\|_{L^{2}}$.  Conversely,
every entire function obeying~\eqref{eq:PW-growth} for some $\Omega$ is
the Fourier--Laplace transform of an $L^{2}$ function whose Fourier
support lies in $[-\Omega,\Omega]$.
\end{theorem}

\paragraph{Characterization in terms of power series expansion.}
Write $F(z)=\sum_{n=0}^{\infty}a_{n}z^{n}$.
Applying Cauchy’s estimate to~\eqref{eq:PW-growth} on the circle
$|z|=r$ with $r>0$ gives
\[
  |a_{n}|
  \;=\;
  \Bigl|\frac{1}{2\pi i}\int_{|z|=r}
          \frac{F(z)}{z^{\,n+1}}\,dz\Bigr|
  \;\le\;
  \frac{C}{r^{n}}\,
  e^{\Omega r}.
\]
Optimising in $r$ (set $r=n/\Omega$) yields the weighted
bound
\[
  |a_{n}|
  \;\le\;
  C\,
  \frac{e^{n}}{n^{n}}\,
  \Omega^{\,n}
\]
Equivalently by Stirling's approximation, 
\begin{equation}\label{eq:PW-coeff-asymp}
  \limsup_{n\to\infty}
  \Bigl(|a_{n}|\,n!\Bigr)^{1/n}
  \;\le\;
  \Omega,
\end{equation}

\noindent
Thus the exponential--type parameter $\Omega$ controls the
\emph{factorial–scaled} decay of the Taylor coefficients: the smaller
the spectral band $[-\Omega,\Omega]$, the faster the coefficients
$a_{n}$ tend to zero.  Conversely, a power--series with
$|a_{n}|=O(\Omega^{n}/n!)$ always represents an entire function of
exponential type at most $\Omega$, hence lies in the Paley--Wiener
image of $PW_{\Omega}$.

\subsection{Hermite functions}
Hermite functions are typically defined in terms of the \emph{physicist's} Hermite polynomials rather than the probabilists one, which differ by some important scaling factors. 

For $n\in\N$ we set
\begin{align}
  H_{n}(x)
  &\;:=\;(-1)^{n}\,e^{x^{2}}\frac{d^{n}}{dx^{n}}e^{-x^{2}}
      &&\text{(``physicist's'')} ,\\[4pt]
  \He_{n}(x)
  &\;:=\;(-1)^{n}\,e^{x^{2}/2}\frac{d^{n}}{dx^{n}}e^{-x^{2}/2}
      &&\text{(``probabilist's'')}.
\end{align}
They are related by the rescaling
\begin{equation}\label{eq:H_vs_He}
  H_{n}(x)\;=\;2^{\,n/2}\,\He_{n}\!\bigl(\sqrt{2}\,x\bigr),
  \qquad
  \He_{n}(x)\;=\;2^{-\,n/2}\,H_{n}\!\bigl(x/\sqrt{2}\bigr).
\end{equation}
The \emph{Hermite functions} are defined from the physicist's polynomials by
\begin{equation}\label{eq:hermite_fn_phys}
  \varphi_{n}(x)
  \;:=\;
  \frac{1}{\pi^{1/4}\,\sqrt{2^{\,n}n!}}\;
  H_{n}(x)\,e^{-x^{2}/2},
  \qquad n=0,1,2,\dots .
\end{equation}
With this normalisation the family $\{\varphi_{n}\}_{n\ge 0}$ is an orthonormal basis of $L^{2}(\R)$.

\paragraph{Expressing $\varphi_{n}$ with the probabilist's polynomials.}
Insert $x\mapsto x/\sqrt{2}$ in \eqref{eq:hermite_fn_phys} and use
\eqref{eq:H_vs_He} to get:
\begin{equation}\label{eq:hermite_fn_prob}
\varphi_{n}\!\Bigl(\tfrac{x}{\sqrt{2}}\Bigr)
  \;=\;
  \frac{1}{\pi^{1/4}\,\sqrt{2^{\,n}n!}}\;
  H_{n}\!\Bigl(\tfrac{x}{\sqrt{2}}\Bigr)\,
  e^{-x^{2}/4}
  \;=\;
  \frac{1}{\pi^{1/4}\sqrt{n!}}\;
  \He_{n}(x)\,e^{-x^{2}/4}.
\end{equation}

\paragraph{Generating function.} The probabilist's Hermite polynomials can equivalently be defined in terms of the following generating function identity. For $x \in \mathbb R$,
\[
 \sum_{n=0}^\infty \mathrm{He}_n(x)\,\frac{t^n}{n!}=e^{xt-\tfrac{t^2}{2}},
 \]
\subsection{Hermite expansions in \texorpdfstring{$L_2(N(0,1))$ and $L_2(\R)$}{L2(N(0,1)) and L2(R)}}

Let  
\[
  h_{n}(x)\;:=\;\frac{\He_{n}(x)}{\sqrt{n!}},
  \qquad n=0,1,2,\dots ,
\]
denote the orthonormal version of the probabilist's Hermite polynomials in $L_2(N(0,1))$. 
In what follows, it is crucial to correctly distinguish and understand the relationship between Hermite polynomial expansion in $L_2(N(0,1))$ versus the Hermite function expansion in $L_2(\mathbb R)$. 

\begin{lemma}\label{lem:hermite-gaussian-r}
Assume that
\[
  f(x)
  \;=\;
  \sum_{n=0}^{\infty} a_{n}\,h_{n}(x),
  \qquad
  a_{n}\in\C,
\]
is the Gaussian–Hermite expansion of
\(f\in L_{2}\!\bigl(N(0,1)\bigr)\).
Defining
$r(x) := \frac{1}{\pi^{1/4}} e^{-x^2/2} f(\sqrt{2} x)$,
we have that $r \in L_2(\R)$ and
\[ r(x) = \sum_{n=0}^{\infty} a_{n}\,\varphi_{n}(x). \]
\end{lemma}
\begin{proof}
From~\eqref{eq:hermite_fn_prob} we have the pointwise identity  
\[ \varphi_{n}\!\Bigl(\tfrac{x}{\sqrt{2}}\Bigr)
  \;=\;
  \frac{1}{\pi^{1/4}\sqrt{n!}}\;
  \He_{n}(x)\,e^{-x^{2}/4} = \frac{1}{\pi^{1/4}}\;
  h_n(x)\,e^{-x^{2}/4}\]
Therefore
\[
  g(x)
  \;:=\;
  \frac{1}{\pi^{1/4}}\,
  e^{-x^{2}/4}\,f(x)
  \;=\;
  \sum_{n=0}^{\infty} a_{n}\,\varphi_{n}(x/\sqrt{2}).
\]
So
\[ r(x) := g(\sqrt{2} x) = \frac{1}{\pi^{1/4}} e^{-x^2/2} f(\sqrt{2} x) =  \sum_{n=0}^{\infty} a_{n}\,\varphi_{n}(x). \]
and we see that its expansion into Hermite functions has the same coefficients as the expansion of $f$ into orthonormal Hermite polynomials.
\end{proof}

\subsection{Bargmann transform}
The following transformation is closely related to the ``holomorphic representation'' for quantum mechanics. See Chapter 14.4 of \cite{hall2013quantum} as well as, e.g.,  \cite{berger1986toeplitz,bargmann1962remarks}.
\paragraph{Definition.}
For $f\in L^{2}(\R)$ the \emph{Bargmann transform}
\[
  \mathcal{B}:L^{2}(\R)\longrightarrow
  \mathcal{F}^{2}(\C)
  \;=\;
  \bigl\{
    F\text{ entire}:
    \|F\|_{\mathcal{F}^{2}}^{2}
    :=\pi^{-1}\!\int_{\C}|F(z)|^{2}\,e^{-|z|^{2}}\,dz
    <\infty
  \bigr\}
\]
is given by
\begin{equation}\label{eq:Bargmann-def}
  (\mathcal{B}f)(z)
  \;:=\;
  \pi^{-1/4}
  \int_{\R}
    \exp\!\bigl(
      -\tfrac12x^{2}
      +\sqrt{2}\,x\,z
      -\tfrac12z^{2}
    \bigr)\,
    f(x)\,dx,
  \qquad z\in\C.
\end{equation}
The map $\mathcal{B}$ is a unitary isomorphism
$L^{2}(\R)\xrightarrow{\;\cong\;}\mathcal{F}^{2}(\C)$ and can be viewed as a special case of the short-time Fourier transform \cite{grochenig2001foundations}.
$\mathcal{F}^{2}(\C)$ is called the Segal-Bargmann space or Bargmann-Fock space. 

\paragraph{Action on the Hermite basis.}
Let $\{\varphi_{n}\}_{n\ge0}$ be the orthonormal Hermite functions from
\eqref{eq:hermite_fn_phys}. 
Then
\begin{equation}\label{eq:Bargmann-on-Hermite}
    (\mathcal{B}\varphi_{n})(z)
    \;=\;
    \frac{z^{n}}{\sqrt{n!}},
    \qquad n=0,1,2,\dots
\end{equation}
so $\mathcal{B}$ sends the $\varphi_{n}$’s onto the monomial
orthonormal basis of the Fock space.

\emph{Sketch of derivation.} 
Insert the generating function
$\displaystyle
  \sum_{n=0}^{\infty}\frac{H_{n}(x)\,t^{n}}{n!}
  =\exp(-t^{2}+2tx)
$
in~\eqref{eq:Bargmann-def} with
$t=z/\sqrt{2}$ and compare coefficients of $t^{n}$; the Gaussian integral
\(
  \int_{\R}\exp\!\bigl(-\tfrac12x^{2}+2tx\bigr)\,dx
  =\sqrt{2\pi}\,e^{2t^{2}}
\)
produces exactly the factor $z^{n}/\sqrt{n!}$ after the normalising
constants are collected. 

\begin{remark}
Because both $\{\varphi_{n}\}$ and
$\{z^{n}/\sqrt{n!}\}$ are orthonormal bases in their respective Hilbert
spaces,~\eqref{eq:Bargmann-on-Hermite} shows that
$\mathcal{B}$ is indeed unitary.
\end{remark}
\subsection{Gelfand--Shilov spaces}
We now recall the definition and some relevant properties of the Gelfand--Shilov spaces. They can be motivated by and in various ways parallel the Paley-Wiener spaces, where the requirement that the function is bandlimited has been slightly relaxed. See Chapter IV of the monograph \cite{gel2013spaces} for detailed explanation as well as the proofs of the fundamental results. 

Recall from before that $\mathcal S$ denotes the Schwartz space of test functions.
\begin{definition}[\cite{gel2013spaces,kashpirovsky1978equality,van1987functional}]
For $\alpha,\beta > 0$, we define the one-sided Gelfand--Shilov spaces as
\[ \mathcal S_{\alpha} = \{ \phi \in \mathcal{S} : \exists a > 0, D > 0 : |\phi(x)| \le D \exp(-\alpha a |x|^{1/\alpha}) \} \]
and
\[ \mathcal S^{\beta} = \{ \phi \in \mathcal{S} : \exists b > 0, D > 0 : |\hat \phi(x)| \le D \exp(-\beta b |x|^{1/\beta}) \} \]
where $\hat \phi$ is the Fourier transform of $\phi$. For $\alpha,\beta > 0$. We can also define the two-sided Gelfand--Shilov space by $\mathcal S_{\alpha}^{\beta} = \mathcal S_{\alpha} \cap \mathcal S^{\beta}$ \cite{kashpirovsky1978equality}. 
\end{definition}
\begin{remark}
The two-sided Gelfand--Shilov space $\mathcal S_{\alpha}^{\beta}$ is trivial, only containing zero, if $\alpha + \beta < 1$ \cite{gel2013spaces}. This is a type of uncertainty principle (see, e.g., \cite{hall2013quantum,grochenig2001foundations,folland1997uncertainty}). 
\end{remark}
There are several equivalent ways to define the Gelfand--Shilov spaces. For example, the initial definition in \cite{gel2013spaces} is in terms of the decrease of the derivatives of $\phi$ at infinity. 
From the above definition, it is clear that the Fourier transform $\mathcal F$ is a bijection between $\mathcal S_{\alpha}$ and $\mathcal S^{\alpha}$, and that it bijectively maps $\mathcal S_{\alpha}^{\alpha}$ to itself.

When $\beta < 1$, a different and useful characterization of $\mathcal S^{\beta}$ is in terms of extensions to entire functions. 
\begin{theorem}[Chapter IV of \cite{gel2013spaces}, see IV.7.1]\label{thm:gs-entire}
Let $\beta\in (0,1)$.  A Schwartz function $\phi\in\mathcal S(\R)$ belongs to 
\(\displaystyle \mathcal S^{\beta}\) 
if and only if it admits an extension to an entire function on~\(\C\) and there exist constants \(C,a>0\) such that
\[
  \bigl|\phi(z)\bigr|
  \;\le\;
  C\,
  \exp\!\bigl(a\,|\Im z|^{\,\frac{1}{1-\beta}}\bigr)
  \quad
  \text{for all }z\in\C.
\]
In other words,
\[
  \mathcal S^{\beta}
  \;=\;
  \Bigl\{
    \phi|_{\R} \in\mathcal S(\R)\;:\;
    \phi \;\text{entire and }
    \exists\,a,C>0:\;
    |\phi(z)|\le C\,e^{a|\Im z|^{\!1/(1-\beta)}}
  \Bigr\}.
\]
\end{theorem}
Note that when $\beta > 1$, the elements of the Gelfand-Shilov space are no longer guaranteed to be analytic.
The intuition behind the above result is similar to the Paley-Weiner theorem: if the Fourier transformation $\hat \phi$ decays rapidly, then by using the Fourier inversion formula and plugging in $z = x + iy \in \mathbb C$,
\[ \phi(z) = \frac{1}{2\pi} \int \hat \phi(\xi) e^{i z \xi} d\xi =  \frac{1}{2\pi} \int \hat \phi(\xi) e^{i x \xi} e^{-y \xi} d\xi \]
we can extend $\phi$ to a complex analytic function with growth rate at most $e^{O(y^{1/(1 - \beta)})}$. The converse direction is harder and builds on the Phragm\'en-Lindel\"of principle \cite{gel2013spaces}.
\begin{theorem}[Chapter IV of \cite{gel2013spaces}, see IV.7.1]\label{thm:saa}
For any $\alpha > 0$ and $\beta \in (0,1)$,
\[ \mathcal S_{\alpha}^{\beta} = \{ \phi|_{\mathbb R} : \phi \text{ entire and } \exists\, a,b,c > 0 \text{ s.t. } |\phi(x + iy)| \le c \exp(-a |x|^{1/\alpha}  + b |y|^{1/(1 - \beta)}) \}. \]
\end{theorem}

\paragraph{Additional results from \cite{gel2013spaces}.} We will use the following key complex-analytic estimates from Chapter IV.7 of \cite{gel2013spaces}:
\begin{theorem}[Theorems IV.7.1 of \cite{gel2013spaces}]
If an entire function $f(z)$ satisfies the inequality
\[ |f(z)| \le C_1 \exp(b |z|^p) \qquad \forall z \in \mathbb C \]
for some $p > 0$ and $C_1 > 0$, and for some $a \ne 0$ and $h \in (0,p]$ satisfies
\[  |f(x)| \le C_2 \exp(a |x|^h) \qquad \forall x \in \mathbb R \]
then for any $a' > a$, there exists $K_1 > 0$ and $\mu \ge 1 - (p - h)$ such that
\[ |f(x + iy)| \le \max(C_1,C_2) \exp(a' |x|^h) \]
whenever
\begin{equation}\label{eqn:y-domain}
|y| \le K_1(1 + |x|)^{\mu}. 
\end{equation}
\end{theorem}
\begin{theorem}[Theorem IV.7.2 of \cite{gel2013spaces}]\label{cor:extension-estimate}
If an entire function $f(z)$ satisfies the inequality
\[ |f(z)| \le C_1 \exp(b |z|^p) \qquad \forall z \in \mathbb C \]
and there exists $K_1 > 0$, $\mu \in (0,1)$, $a' < 0$, and $h \in (0,p]$ such that 
\[ |f(x + iy)| \le C_2 \exp(a' |x|^h) \]
whenever $x,y$ satisfies the constraint from \eqref{eqn:y-domain}, then there exists constants $C_3,b' > 0$ such that
\[ |f(z)| \le C_3 \exp(a' |x|^h + b' |y|^{p/\mu}) \qquad \forall z \in \mathbb C \]
\end{theorem}
Combining these results immediately yields the following corollary:
\begin{corollary}
Suppose an entire function $f(z)$ satisfies the inequality
\[ |f(z)| \le C_1 \exp(b |z|^p) \qquad \forall z \in \mathbb C \]
for some $p, C_1, b > 0$, and for some $a < 0$, $C_2 \ge 0$, and $h \in (p - 1,p]$ satisfies
\[  |f(x)| \le C_2 \exp(a |x|^h) \qquad \forall x \in \mathbb R. \]
Then for any $a' > a$,  there exist constants $C_3, b' > 0$ such that
\[ |f(z)| \le C_3 \exp(a' |x|^h + b' |y|^{p/\mu}) \qquad \forall z \in \mathbb C \]
where
\[ \mu = 1 - (p - h) \in (0,1]. \]
\end{corollary}
If $\mu < p$ then combined with Theorem~\ref{thm:saa}, the conclusion tells us that $f$ lies in the two-sided Gelfand-Shilov space $\mathcal S_{\alpha}^{\beta}$ with $\alpha = 1/h$ and $\beta = 1 - \mu/p \in (0,1)$.
\section{Universality in Paley--Wiener classes}\label{sec:pw-1d}
\subsection{Paley--Wiener Theorem under the Bargmann Transform}

\begin{theorem}\label{thm:HPW}
Fix $\Omega>0$, suppose that $f \in L_2(\mathbb R)$, and let
\[
  f(x)=\sum_{n=0}^{\infty}a_{n}\,h_{n}(x)
\]
in $L^{2}\!\bigl(N(0,1)\bigr)$.  
The following are equivalent:
\begin{enumerate}
  \item[(i)]   $\operatorname{supp}\widehat f\subset[-\Omega,\Omega]$.
  \item[(ii)]  There exists a constant $M>0$ such that for all $n \ge 0$,
               \[
                 |a_{n}|\;\le\;M\,\frac{\Omega^{\,n}}{\sqrt{n!}}.
               \]
\end{enumerate}
\end{theorem}
For the proof and as in the previous sections, we define the auxiliary functions
\[
  g(x)=\pi^{-1/4}\,e^{-x^{2}/4}\,f(x),
  \qquad
  r(x)=g(\sqrt2\,x)
      =\pi^{-1/4}\,e^{-x^{2}/2}\,f(\sqrt2\,x),
\]
so that  
\(r=\sum_{n\ge0}a_{n}\varphi_{n}\)  
and, by \eqref{eq:Bargmann-on-Hermite},
\begin{equation}\label{eq:Barg_coeff_id}
  (\mathcal{B}r)(z)=\sum_{n=0}^{\infty}a_{n}\frac{z^{n}}{\sqrt{n!}}
  =:\sum_{n=0}^{\infty}c_{n}z^{n},
  \qquad c_{n}:=\frac{a_{n}}{\sqrt{n!}}.
\end{equation}

The theorem is proved by comparing the power series formula for $\mathcal{B}r$ with its Fourier interpretation:
\begin{lemma}
With the notation above,
\[ (\mathcal{B} r)(z) = \int
      \widehat f(\xi)\,e^{-\xi^{2}/2}\,e^{i\xi z}\,d\xi. \]
\end{lemma}
\begin{proof}
Starting from the defining integral \eqref{eq:Bargmann-def} and inserting
$r(x)=\pi^{-1/4}e^{-x^{2}/2}f(\sqrt2\,x)$ we obtain
\[
  (\mathcal{B}r)(z)
  =\frac{e^{-z^{2}/2}}{\pi^{1/2}}
   \int_{\R}e^{-x^{2}+\sqrt2\,xz}\,f(\sqrt2\,x)\,dx
  =\frac{e^{-z^{2}/2}}{(2\pi)^{1/2}}
   \int_{\R}e^{-y^{2}/2+yz}\,f(y)\,dy,
\]
after the change of variables $y=\sqrt2\,x$.  
Insert the Fourier inversion formula for $f$ and apply Fubini's theorem
to get
\begin{align}
(\mathcal{B}r)(z)
  &=\frac{e^{-z^{2}/2}}{(2\pi)^{1/2}}
    \int \widehat f(\xi)
      \Bigl(\int_{\R}e^{-y^{2}/2+(z+i\xi)y}\,dy\Bigr)d\xi\notag\\
  &= e^{-z^{2}/2}
    \int \widehat f(\xi)\,
      e^{\tfrac12(z+i\xi)^{2}}\,d\xi \\
  &= \int
      \widehat f(\xi)\,e^{-\xi^{2}/2}\,e^{i\xi z}\,d\xi
\end{align}
where in the second line we used the formula for the Gaussian MGF. 
\end{proof}
\begin{proof}[Proof of Theorem \ref{thm:HPW}]

\bigskip
\noindent\textbf{(i)$\;\Rightarrow\;$(ii).}
Assume $\supp\widehat f\subset[-\Omega,\Omega]$.  
Then
\[
  \bigl|(\mathcal{B}r)(x+iy)\bigr|
  \;\le\;
  \|\widehat f\,e^{-\xi^{2}/2}\|_{L^{1}}\,
  e^{\Omega|y|}
  \;=\;C_{1}\,e^{\Omega|y|},
  \qquad x,y\in\R.
\]

That is, $(\mathcal{B}r)$ is an entire function of exponential type
$\le\Omega$.  Apply Cauchy’s estimate on the circle $|z|=r$ to the
coefficients $c_{n}$ from~\eqref{eq:Barg_coeff_id},
\[
  |c_{n}|
  \;\le\;\frac{C_{1}}{r^{n}}\,e^{\Omega r}.
\]
Minimising the right–hand side at $r=n/\Omega$ gives
\(
  |c_{n}|\le C_{1}\,(\Omega e/n)^{n}.
\)
Multiplying by $\sqrt{n!}$ yields
\[
  |a_{n}|
  =\sqrt{n!}\,|c_{n}|
  \;\le\;
  C_{2}\,\frac{\Omega^{\,n}}{\sqrt{n!}},
  \qquad n\ge0,
\]
which is exactly (ii).  

\bigskip
\noindent\textbf{(ii)$\;\Rightarrow\;$(i).}
Conversely, suppose
\(
  |a_{n}|\le M\,\Omega^{\,n}/\sqrt{n!}
\)
for all $n$.  
Then the coefficients \(c_{n}=a_{n}/\sqrt{n!}\) satisfy
\(
  |c_{n}|\le M\,\Omega^{\,n}/n!.
\)
The coefficient bound implies (by the standard argument used earlier) that
$\mathcal{B} r$ is an entire function of exponential type $\le\Omega$ and
$(\mathcal B r)|_{\R}\in L^{2}(\R)$.  
Recall from the lemma  that 
\[
(\mathcal{B}r)(z)
       = \int_{\R}\widehat f(\xi)\,
         e^{-\xi^{2}/2}\,e^{i\xi z}\,d\xi.
\]
By the Paley-Wiener theorem and the fact that $e^{-\xi^2/2}$ does not vanish on the real line, we must
have that $\widehat f$ is supported on $[-\Omega,\Omega]$.
Thus (i) holds, completing the proof of the equivalence.
\end{proof}
\subsection{Consequences for universality}
\begin{theorem}[Universality in Paley--Wiener classes]\label{thm:pw-universality}
Suppose that $f \in L_2(\mathbb R)$ and we have the orthogonal $L_2(N(0,1))$ expansion 
\[
  f(x)=\sum_{n=0}^{\infty}a_{n}\,h_{n}(x)
\]
For any $\Omega > 0$, the following are equivalent:
\begin{enumerate}
    \item $f \in PW_{\Omega}$. 
    \item There exist $M > 0$ such that $|a_n| \le M \frac{\Omega^n}{\sqrt{n!}}$.
    \item There exists a constant $C > 0$ and a sequence of polynomials $p_m$ of degree $m$ such that
    for all $m \ge 1$,
    \[ \|f - p_m\|_{L_2(N(0,1))} \le C \frac{(\Omega\sqrt{e})^m}{m^{m/2}} \]
    \item There exists a constant $C > 0$ such that for any probability measure $\mu$ such that $X \sim \mu$ is $(K \sqrt{2})$-sub-Gaussian\footnote{Here we mean that $\mathbb E e^{\lambda X} \le Ae^{K^2 \lambda^2}$ for some $A > 0$, to compare with our definition of $r$-strictly sub-exponential distributions.}, there exists a sequence of polynomials $p_m$ of degree $m$ such that
    for all $m \ge 1$,
    \[ \|f - p_m\|_{L_2(\mu)} \le C \frac{(\Omega K \sqrt{2e})^m}{m^{m/2}} \]
    \item There exists a constant $C > 0$ such that for any probability measure $\mu$ which is $r$-strictly-sub-exponential with scale parameter $K$, there exists a sequence of polynomials $p_m$ of degree $m$ such that for all $m \ge 1$,
    \[ \|f - p_m\|_{L_2(\mu)} \le C \frac{(\Omega K (er)^{1/r})^m}{m^{m/r}} \]
\end{enumerate}
\end{theorem}
\begin{proof}
The equivalence of the first two follows from the Paley-Wiener theorem under the Bargmann transform described above. 

Given the second condition, we have that
\[ \|f - \sum_{n = 0}^m a_n h_n(x)\|_2^2 = \sum_{n = m + 1}^{\infty} |a_{n + 1}|^2 \le M^2 \sum_{n = m + 1}^{\infty} \frac{\Omega^{2n}}{n!} \le  M^2 \sum_{n = m + 1}^{\infty} \left(\frac{e \Omega^{2}}{n}\right)^n.  \]
For sufficiently large $m$, the right hand side is at most $[e\Omega^2/m]^m$, which proves the third condition.

Given the third condition, the second condition follows because the inequality must hold for the optimal polynomial $p_m = \sum_{n = 0}^m a_n h_m$, hence
$|a_n|^2 \le \sum_{k = n}^{\infty} |a_k|^2 \le C (\Omega \sqrt{e})^{2n}/n^n$. So we can recover the second condition for any $\Omega' > \Omega$ by using Stirling's formula. Using the equivalence between the first and second conditions, this implies that $f \in \bigcap_{\Omega' > \Omega} PW_{\Omega'} = PW_{\Omega}$, which proves the result.

Given the fourth condition, the third condition is a special case since the Gaussian distribution is $1$-sub-Gaussian. Similarly, the fourth condition is the special of the fifth condition when $r = 2$.

Finally, given the first condition, the fifth condition follows from our approximation result, Theorem~\ref{thm:pw-strict}.
\end{proof}
\begin{remark}
If the probability measure $\mu$ is sub-exponential, there exists polynomials of degree $m$ with $L_2$ squared error $A c^m$ for some $c < 1$ by Theorem~\ref{thm:pw-subexp}. In the next section we show that this cannot be improved to super-exponential convergence.
\end{remark}
\subsection{Examples}
We give some examples of Hermite expansions to illustrate the connection with the Paley--Wiener class.
While $\cos(x)$ is not in $L_2(\mathbb R)$, it is not hard to show that it can be approximated in $L_2(N(0,1))$ arbitrarily well by uniformly band-limited functions (or see Section~\ref{sec:Rd-analytic-lemmas} which can handle this case directly). 
For the function $\cos(x)$ over the distribution $N(0,1)$, it can be computed that its $L^2$ expansion into orthonormal Hermite polynomials is
\[ \cos(x) = e^{-1/2} \sum_k (-1)^k \psi_{2k}/\sqrt{2k!} \]
and so we can see that to get $\epsilon$ error, truncating to $\log(1/\epsilon)/\log\log(1/\epsilon)$ error is sharp. Our result implies the same approximation rate can be obtained over all strictly sub-exponential distributions.

We can also use sinc function and similarly find
\[
\mathrm{sinc}(x)
=\sum_{k=0}^{\infty}a_{2k}\,\psi_{2k}(x),
\qquad
a_{2k+1}=0,
\]
\[
a_{2k}
=\frac{1}{\sqrt{(2k)!}}\,
\int_{-\infty}^{\infty}\!\frac{\sin x}{x}\,H_{2k}(x)\,\gamma(dx)
\;=\;(-1)^k\;\frac{2^k}{\sqrt{2\,(2k)!}}\;
\gamma\!\Bigl(k+\tfrac12,\tfrac12\Bigr),
\]
\[
\psi_{n}(x)=\frac{H_{n}(x)}{\sqrt{n!}},
\quad
\gamma(dx)=\frac{e^{-x^2/2}}{\sqrt{2\pi}}\,dx,
\]
Again the coefficients decay roughly like $k^k$, since it is a bandlimited function.

\paragraph{No super-exponential convergence for Laplace (symmetric exponential) distribution.}
Consider the function $\cos(\Omega x)$. Using that a degree $D$ polynomial has at most $D$ real zeros, and $\cos(\Omega x)$ has zeros at all points $(\pi/2 + m\pi i)/\Omega$, we see small $L_1$ polynomial approximation is impossible on an interval of size bigger than $\Theta(D/\Omega)$. So for sub-exponential but not strictly sub-exponential tail decay, we can not take $D$ of smaller order than $\log(1/\epsilon)$. 

\section{Universality in Gelfand--Shilov spaces}\label{sec:gs}





\subsection{Gelfand--Shilov Spaces and Hermite Expansion}
The following result of Van Eijndhoven characterizes symmetric Gelfand--Shilov spaces in terms of exponential decay of Hermite expansions in $L_2(\mathbb R)$.
\begin{theorem}[Page 141 of \cite{van1987functional}]\label{thm:van1987}
For any $\alpha \ge 1/2$, we have that $g \in \mathcal S^{\alpha}_{\alpha}$ if and only if $g \in L_2(\mathbb R)$ and 
the coefficients $c_n$ of the Hermite (function) expansion
\[ g = \sum_n c_n \varphi_n \]
decay at an exponential rate in $n^{1/2\alpha}$, i.e. there exists $t > 0$ such that $|c_n| = O(\exp(-t n^{1/2\alpha}))$.
\end{theorem}
This result in the case $\alpha = 1/2$ has the following interesting consequence. Note that in the corollary below, the function $f$ is \emph{not} necessarily required to live in Schwartz or Gelfand--Shilov spaces. Instead it is more like a multiplier (see Chapter IV.7 of \cite{gel2013spaces}).
\begin{corollary}\label{corr:gs-fun}
Let $f \in L^2(N(0,1))$ have a Hermite polynomial expansion
 \[
   f(x)=\sum_{n=0}^{\infty}a_{n}\,h_{n}(x),
   \qquad h_{n}(x)=\frac{\He_{n}(x)}{\sqrt{n!}}.
 \]
Then the following are equivalent:
\begin{enumerate}
    \item $f$ extends to an entire function on $\mathbb C$ satisfying
    \[ |f(x + iy)| \le C \exp((1/2 - a)|x|^2/2 + b|y|^2/2) \]
    for some $a,b,C > 0$.
    \item The Hermite coefficients $a_n$ decay at an exponential rate, i.e. there exists $t > 0$ such that
    \[ |a_n| = O(\exp(-tn)). \]
\end{enumerate}
\end{corollary}
\begin{proof}
As in Lemma~\ref{lem:hermite-gaussian-r}, define
\[ r(x) := \frac{1}{\pi^{1/4}} e^{-x^2/2} f(\sqrt{2} x) = \sum_{n=0}^{\infty} a_{n}\,\varphi_{n}(x). \]
From the previous theorem, we see that $r \in \mathcal S_{1/2}^{1/2}$ iff $a_n$ decays at an exponential rate. 
By Theorem~\ref{thm:saa}, $r \in \mathcal S_{1/2}^{1/2}$ iff it admits an entire extension and there exists $a,b,c > 0$ such that
\[ |r(x + iy)| \le c \exp(-a |x|^{2}  + b |y|^{2}).\]
This is equivalent to requiring that $f$ admits an entire extension, given by $f(\sqrt{2} z) = \pi^{1/4} e^{z^2/2} r(z)$, satisfying for some $b' > 0$
\[ |f(\sqrt{2}(x + iy))| \le c \pi^{1/4} \exp((1/2 - a)|x|^{2} + b' |y|^{2}), \]
equivalently
\[ |f(x + iy)| \le c \pi^{1/4} \exp((1/2 - a)|x|^2/2 + b'|y|^2/2). \]
\end{proof}
\subsection{Consequences for universality}
In this section we derive some results for Gelfand--Shilov spaces $\mathcal S_{1 - \beta}^{\beta}$ with $\beta \in [1/2,1)$.
To help illustrate the ideas, we focus on perhaps the most interesting case of $\beta = 1/2$ first and then state the natural generalization to $\beta \in [1/2,1)$.

For the Gelfand--Shilov space $\mathcal S_{1/2}^{1/2}$, we get a universality result that is slightly weaker than the Paley-Wiener version: for a function with sufficient tail decay to lie in $\mathcal S_{1/2}$, if we can achieve $\epsilon$-approximation using polynomials of degree $O(\log(1/\epsilon))$ over Gaussian space, then for general sub-Gaussian measures we can conclude there exists $\epsilon$-approximating polynomials of degree $O(\log^2(1/\epsilon))$. Regardless, the qualitative interpretation of the result is the same  --- for Schwarz functions, exponentially rapid approximation by polynomials implies similar approximation guarantees over all strictly-sub-exponential measures. 
\begin{theorem}
Let $f \in \mathcal S_{1/2}$, the one-sided Gelfand-Shilov space, and let its Hermite polynomial expansion be denoted
 \[
   f(x)=\sum_{n=0}^{\infty}a_{n}\,h_{n}(x)
 \]
 in $L_2(N(0,1))$. 
The following two conditions are equivalent:
\begin{enumerate}
    \item $f \in \mathcal S_{1/2}^{1/2}$, the two-sided Gelfand-Shilov space.
    \item The Hermite polynomial coefficients $a_n$ decay at an exponential rate, i.e. there exists $t > 0$ such that
    \[ |a_n| = O(\exp(-tn)). \]
\end{enumerate}
Either of the equivalent conditions implies that
\begin{enumerate}
    \item[3.] For any $r$-strictly-sub-exponential probability measure $\mu$, there exists a constant $C > 0$ such that the following is true. For any $\epsilon \in (0,1/2)$, there exists a polynomial $P$ of degree at most $C \log^r(1/\epsilon))$ such that 
    \[ \|f - P\|_{L_2(\mu)} \le \epsilon. \]
\end{enumerate}
\end{theorem}
\begin{remark}
For the last conclusion, it is enough to assume the one-sided condition $f \in \mathcal S^{1/2}$, we do not need to use that $f \in \mathcal S_{1/2}$.
\end{remark}
\begin{proof}
Given the first assumption that $f \in \mathcal S_{1/2}^{1/2}$, the third point is a consequence of our approximation result for Gaussian-analytic functions, Proposition~\ref{prop:poly-approx}, using the Fourier-theoretic definition of $\mathcal S^{1/2}$.

We proceed to explain the equivalence of the first two conditions. Suppose that $f \in \mathcal S_{1/2}^{1/2}$, then
the second statement about the Hermite coefficients follows by combining Corollary~\ref{corr:gs-fun} with the complex-analytic characterization of $\mathcal S_{1/2}^{1/2}$ from Theorem~\ref{thm:saa}. 

Now suppose that $|a_n| = O(\exp(-tn))$. By Corollary~\ref{corr:gs-fun}, we know that $f$ is an entire function of order $2$, i.e. $|f(z)| \lesssim e^{O(|z|^2)}$. Combining this with the assumption that $f \in \mathcal S_{1/2}$, the assumptions of Corollary~\ref{cor:extension-estimate} hold with $p = h = 2$ and $\mu = 1$. Then the conclusion of Corollary~\ref{cor:extension-estimate} exactly tells us that $f \in \mathcal S_{1/2}^{1/2}$.
\end{proof}
 More generally, we have the following result for H\"older conjugates $1/\alpha$ and $1/\beta$:
\begin{theorem}
Let $\beta \in [1/2,1)$ and $\alpha = 1 - \beta$, and let $f \in \mathcal S_{\alpha}$ be an element of Schwartz space with Hermite polynomial expansion
 \[
   f(x)=\sum_{n=0}^{\infty}a_{n}\,h_{n}(x)
 \]
 in $L_2(N(0,1))$. 
The following two conditions are equivalent:
\begin{enumerate}
    \item $f \in \mathcal S_{\alpha}^{\beta}$, the two-sided Gelfand--Shilov space.
    \item  The Hermite polynomial coefficients $a_n$ decay at an exponential rate in $n^{1/2\beta}$, i.e. there exists $t > 0$ such that
    \[ |a_n| = O(\exp(-tn^{1/2\beta})) \]
\end{enumerate}
Either of the equivalent conditions implies\footnote{Again, the one-sided condition $f \in \mathcal S^{\beta}$ is sufficient here without requiring $f \in \mathcal S_{\alpha}$.} that
\begin{enumerate}
    \item[3.]  For any $r$-strictly-sub-exponential probability measure $\mu$, there exists a constant $C > 0$ such that the following is true. For any $\epsilon \in (0,1/2)$, there exists a polynomial $P$ of degree at most $C \log^{2\beta r}(1/\epsilon)$ such that
    \[ \|f - P\|_{L_2(\mu)} \le \epsilon. \]
\end{enumerate}
\end{theorem}
\begin{proof}
First we prove equivalence of the first two conditions. 
As in Lemma~\ref{lem:hermite-gaussian-r}, define
\[ r(x) := \frac{1}{\pi^{1/4}} e^{-x^2/2} f(\sqrt{2} x) = \sum_{n=0}^{\infty} a_{n}\,\varphi_{n}(x). \]

Start by assuming that $f \in \mathcal S^{\beta}$. 
Using the complex-analytic characterization of $\mathcal S^{\beta}$ from Theorem~\ref{thm:gs-entire}, as well as the two-sided version from Theorem~\ref{thm:saa} and the fact that $\beta \ge 1/2$, we see that $r \in \mathcal S^{\beta}_{\beta}$. By Theorem~\ref{thm:van1987}, this implies the stated decay rate for the Hermite coefficients $a_n$ of $r$, which by Lemma~\ref{lem:hermite-gaussian-r} yields the second conclusion.

Now assume that $|a_n| = O(\exp(-tn^{1/2\beta}))$. Then by Theorem~\ref{thm:van1987}, we have that $r \in \mathcal S^{\beta}_{\beta}$, and by Theorem~\ref{thm:saa} this means that for some $a,b,C > 0$ we have an entire extension of $r$ such that
\[ |r(x + iy)| \le C\exp(-a |x|^{1/\beta} + b |y|^{1/(1 - \beta)}). \]
From the definition of $r$, this implies that $f$ admits an entire extension specified by $f(\sqrt{2} z) = r(z)\pi^{1/4} e^{z^2/2}$ satisfying
\[ |f(x + iy)| = \pi^{1/4} |e^{z^2/4}||r((x + iy)/\sqrt{2})| \le C' \exp(a' x^2 + b' |y|^{1/(1 - \beta)}) \]
for some $a',b',C' > 0$. Since $2 \le 1/(1 - \beta)$, it follows that $f$ is an entire function of order $p = 1/(1 - \beta)$. 
Since $\alpha = 1 - \beta$, the assumptions of Corollary~\ref{cor:extension-estimate} are satisfied and so we find that $f \in \mathcal S_{\alpha}^{\beta}$ as desired. 

Finally, we prove that having $f \in \mathcal S^{\beta}$ implies the third condition. From the Fourier-theoretic definition of $\mathcal S^{\beta}$ we know that $|\hat f(\xi)| \le C'' e^{-b''|\xi|^{1/\beta})}$ for some $C'',b'' > 0$. Therefore for $\Omega > 0$ we have the tail bound
\[ \int_{\mathbb R \setminus [-\Omega,\Omega]} |\hat f(\xi)|d\xi \lesssim e^{-b'' \Omega^{1/\beta}}. \]
So applying our polynomial approximation result from Theorem~\ref{thm:pw-strict}, we find that there exists a polynomial $p$ of degree at most $D$ such that
\[ \|f - p\|_{L_2(\mu)} \lesssim  \frac{\|\hat f\|_{L^1([-\Omega,\Omega])}}{2\pi}
       \Bigl(\frac{r\,e}{D}\Bigr)^{D/r} \cdot (K\Omega)^{D}
    + e^{-b'' \Omega^{1/\beta}}. \]
Taking $\Omega = \Theta(\log^{\beta}(1/\epsilon))$ and $D = \Theta(\log^{2r \beta}(1/\epsilon))$ proves the result. 
\end{proof}
\subsection{Examples}
As we see from the computation below, the function $x \mapsto e^{cx^2}$ for $c \le 0$ is an example of a bounded function with exponentially decaying Hermite expansion, so our universality result applies to it. Because the Fourier transform maps Gaussians to Gaussians, we know that this function is an element of the Gelfand--Shilov space $\mathcal S^{1/2}$, but not any Paley-Wiener space. This is because its Fourier transform lacks compact support (or equivalently, by the Paley-Wiener theorem, because it grows too quickly on $\mathbb C$).

On the other hand, the example of $e^{cx^2}$ for $c \in (0,1/4)$ shows that having a Hermite coefficient expansion which decays at an exponential rate is not itself sufficient
to obtain any universality result over sub-exponential distributions. This function is in $L_2(N(0,1))$ but it is not in Schwartz space, and it is not square-integrable for all strictly sub-exponential distributions.
\begin{proposition}
Let $X\sim N(0,1)$.
If $c<\tfrac14$, then the $L_2(N(0,1))$ expansion of $e^{c X^2}$ into orthonormal polynomials is
\[
e^{c X^2}
=(1-2c)^{-\tfrac12}\sum_{k=0}^\infty \frac{\sqrt{(2k)!}}{k!}
\left(\frac{c}{1-2c}\right)^{\!k}\, h_{2k}(X),
\]
with no odd terms (i.e.\ the coefficients of $h_{2k+1}$ vanish).
\end{proposition}

\begin{proof}
%
Write the $L_2$ expansion as $e^{c x^2}=\sum_{n\ge0} a_n H_n(x)$ with $a_n=\mathbb E[e^{cX^2}H_n(X)]$.
To compute $a_n$, we use the generating function for Hermite polynomials. For $c<\tfrac12$,
\[
\mathbb E\!\left[e^{cX^2+tX}\right]
=(1-2c)^{-\tfrac12}\exp\!\Big(\frac{t^2}{2(1-2c)}\Big)
\]
by completing the square. Hence
\[
\sum_{n=0}^\infty \frac{t^n}{n!}\,\mathbb E\!\left[e^{cX^2}\mathrm{He}_n(X)\right]
=\mathbb E\!\left[e^{cX^2}e^{tX-\tfrac{t^2}{2}}\right]
=(1-2c)^{-\tfrac12}\exp\!\left(\frac{c}{1-2c}\,t^2\right).
\]
Expanding the right-hand side,
\[
(1-2c)^{-\tfrac12}\sum_{k=0}^\infty \frac{1}{k!}
\left(\frac{c}{1-2c}\right)^{\!k} t^{2k},
\]
we see all odd coefficients vanish, and for $k\ge0$,
\[
\mathbb E\!\left[e^{cX^2}\mathrm{He}_{2k}(X)\right]
=(1-2c)^{-\tfrac12}\frac{(2k)!}{k!}\left(\frac{c}{1-2c}\right)^{\!k}.
\]
Since $h_{2k}=\mathrm{He}_{2k}/\sqrt{(2k)!}$, we obtain
\[
a_{2k}=(1-2c)^{-\tfrac12}\frac{\sqrt{(2k)!}}{k!}
\left(\frac{c}{1-2c}\right)^{\!k},\qquad a_{2k+1}=0,
\]
which gives the stated expansion.
\end{proof}
Applying Stirling's approximation shows that the coefficients of the expansion indeed decay at an exponential rate.



\paragraph{Acknowledgements.} We thank Min Jae Song and Stefan Tiegel for helpful discussions about LWE, and Konstantinos Stavropoulos for helpful discussions of \cite{chandrasekaran2024smoothed,chandrasekaran2025learning}. 


\phantomsection
\addcontentsline{toc}{section}{References}

\bibliographystyle{alpha}
\bibliography{ref}

@book{akhiezer2020classical,
  title={The classical moment problem and some related questions in analysis},
  author={Akhiezer, Naum Il'ich},
  year={2020},
  publisher={SIAM}
}

@article{gronwall1918gamma,
  title={The gamma function in the integral calculus},
  author={Gronwall, TH},
  journal={Annals of mathematics},
  volume={20},
  number={2},
  pages={35--124},
  year={1918},
  publisher={JSTOR}
}

@article{diakonikolas2023sq,
  title={Sq lower bounds for non-gaussian component analysis with weaker assumptions},
  author={Diakonikolas, Ilias and Kane, Daniel and Ren, Lisheng and Sun, Yuxin},
  journal={Advances in Neural Information Processing Systems},
  volume={36},
  pages={4199--4212},
  year={2023}
}

@article{lubinsky2007survey,
  title={A survey of weighted approximation for exponential weights},
  author={Lubinsky, Doron S},
  journal={arXiv preprint math/0701099},
  year={2007}
}

@inproceedings{chandrasekaran2024smoothed,
  title={Smoothed Analysis for Learning Concepts with Low Intrinsic Dimension},
  author={Chandrasekaran, Gautam and Klivans, Adam and Kontonis, Vasilis and Meka, Raghu and Stavropoulos, Konstantinos},
  booktitle={The Thirty Seventh Annual Conference on Learning Theory},
  pages={876--922},
  year={2024},
  organization={PMLR}
}

@book{diakonikolas2023algorithmic,
  title={Algorithmic high-dimensional robust statistics},
  author={Diakonikolas, Ilias and Kane, Daniel M},
  year={2023},
  publisher={Cambridge university press}
}

@inproceedings{diakonikolas2021optimality,
  title={The optimality of polynomial regression for agnostic learning under gaussian marginals in the SQ model},
  author={Diakonikolas, Ilias and Kane, Daniel M and Pittas, Thanasis and Zarifis, Nikos},
  booktitle={Conference on Learning Theory},
  pages={1552--1584},
  year={2021},
  organization={PMLR}
}

@inproceedings{diakonikolas2017statistical,
  title={Statistical query lower bounds for robust estimation of high-dimensional gaussians and gaussian mixtures},
  author={Diakonikolas, Ilias and Kane, Daniel M and Stewart, Alistair},
  booktitle={2017 IEEE 58th Annual Symposium on Foundations of Computer Science (FOCS)},
  pages={73--84},
  year={2017},
  organization={IEEE}
}

@article{kalai2008agnostically,
  title={Agnostically learning halfspaces},
  author={Kalai, Adam Tauman and Klivans, Adam R and Mansour, Yishay and Servedio, Rocco A},
  journal={SIAM Journal on Computing},
  volume={37},
  number={6},
  pages={1777--1805},
  year={2008},
  publisher={SIAM}
}

@inproceedings{daniely2016complexity,
  title={Complexity theoretic limitations on learning halfspaces},
  author={Daniely, Amit},
  booktitle={Proceedings of the forty-eighth annual ACM symposium on Theory of Computing},
  pages={105--117},
  year={2016}
}

@inproceedings{tiegel2023hardness,
  title={Hardness of agnostically learning halfspaces from worst-case lattice problems},
  author={Tiegel, Stefan},
  booktitle={The Thirty Sixth Annual Conference on Learning Theory},
  pages={3029--3064},
  year={2023},
  organization={PMLR}
}

@inproceedings{klivans2014embedding,
  title={Embedding hard learning problems into gaussian space},
  author={Klivans, Adam and Kothari, Pravesh},
  booktitle={Approximation, Randomization, and Combinatorial Optimization. Algorithms and Techniques (APPROX/RANDOM 2014)},
  pages={793--809},
  year={2014},
  organization={Schloss Dagstuhl--Leibniz-Zentrum f{\"u}r Informatik}
}

@article{spielman2004smoothed,
  title={Smoothed analysis of algorithms: Why the simplex algorithm usually takes polynomial time},
  author={Spielman, Daniel A and Teng, Shang-Hua},
  journal={Journal of the ACM (JACM)},
  volume={51},
  number={3},
  pages={385--463},
  year={2004},
  publisher={ACM New York, NY, USA}
}

@book{shalev2014understanding,
  title={Understanding machine learning: From theory to algorithms},
  author={Shalev-Shwartz, Shai and Ben-David, Shai},
  year={2014},
  publisher={Cambridge university press}
}

@article{klivans2009cryptographic,
  title={Cryptographic hardness for learning intersections of halfspaces},
  author={Klivans, Adam R and Sherstov, Alexander A},
  journal={Journal of Computer and System Sciences},
  volume={75},
  number={1},
  pages={2--12},
  year={2009},
  publisher={Elsevier}
}

@article{arriaga2006algorithmic,
  title={An algorithmic theory of learning: Robust concepts and random projection},
  author={Arriaga, Rosa I and Vempala, Santosh},
  journal={Machine learning},
  volume={63},
  pages={161--182},
  year={2006},
  publisher={Springer}
}

@inproceedings{kane2013learning,
  title={Learning halfspaces under log-concave densities: Polynomial approximations and moment matching},
  author={Kane, Daniel and Klivans, Adam and Meka, Raghu},
  booktitle={Conference on Learning Theory},
  pages={522--545},
  year={2013},
  organization={PMLR}
}

@inproceedings{daniely2021local,
  title={From local pseudorandom generators to hardness of learning},
  author={Daniely, Amit and Vardi, Gal},
  booktitle={Conference on Learning Theory},
  pages={1358--1394},
  year={2021},
  organization={PMLR}
}

@book{vapnik2006estimation,
  title={Estimation of dependences based on empirical data},
  author={Vapnik, Vladimir},
  year={2006},
  publisher={Springer Science \& Business Media}
}

@article{davis1977mean,
  title={Mean integrated square error properties of density estimates},
  author={Davis, Kathryn Bullock},
  journal={The Annals of Statistics},
  pages={530--535},
  year={1977},
  publisher={JSTOR}
}

@article{devroye1992note,
  title={A note on the usefulness of superkernels in density estimation},
  author={Devroye, Luc},
  journal={The Annals of Statistics},
  pages={2037--2056},
  year={1992},
  publisher={JSTOR}
}

@book{devroye2001combinatorial,
  title={Combinatorial methods in density estimation},
  author={Devroye, Luc and Lugosi, G{\'a}bor},
  year={2001},
  publisher={Springer Science \& Business Media}
}

@article{ibragimov1983estimation,
  title={Estimation of distribution density belonging to a class of entire functions},
  author={Ibragimov, IA and Khas’ minskii, RZ},
  journal={Theory of Probability \& Its Applications},
  volume={27},
  number={3},
  pages={551--562},
  year={1983},
  publisher={SIAM}
}

@article{hall1988choice,
  title={Choice of kernel order in density estimation},
  author={Hall, Peter and Marron, JS},
  journal={The Annals of Statistics},
  volume={16},
  number={1},
  pages={161--173},
  year={1988},
  publisher={Institute of Mathematical Statistics}
}

@article{gollakota2020polynomial,
  title={The polynomial method is universal for distribution-free correlational SQ learning},
  author={Gollakota, Aravind and Karmalkar, Sushrut and Klivans, Adam},
  journal={arXiv preprint arXiv:2010.11925},
  year={2020}
}

@book{hörmander1983analysis,
  title={The Analysis of Linear Partial Differential Operators I: Distribution theory and Fourier analysis},
  author={H{\"o}rmander, L.},
  isbn={9783540121046},
  lccn={lc83000616},
  series={Die Grundlehren der mathematischen Wissenschaften in Einzeldarstellungen mit besonderer Ber{\"u}cksichtigung der Anwendungsgebiete},
  url={https://books.google.com/books?id=JmEPAQAAMAAJ},
  year={1983},
  publisher={Springer-Verlag}
}

@book{rudin1987real,
  title={Real and complex analysis},
  author={Rudin, Walter},
  year={1987},
  publisher={McGraw-Hill, Inc.}
}

@inproceedings{dachman2014approximate,
  title={Approximate resilience, monotonicity, and the complexity of agnostic learning},
  author={Dachman-Soled, Dana and Feldman, Vitaly and Tan, Li-Yang and Wan, Andrew and Wimmer, Karl},
  booktitle={Proceedings of the twenty-sixth annual ACM-SIAM symposium on Discrete algorithms},
  pages={498--511},
  year={2014},
  organization={SIAM}
}

@inproceedings{diakonikolas2024sum,
  title={Sum-of-squares lower bounds for non-gaussian component analysis},
  author={Diakonikolas, Ilias and Karmalkar, Sushrut and Pang, Shuo and Potechin, Aaron},
  booktitle={2024 IEEE 65th Annual Symposium on Foundations of Computer Science (FOCS)},
  pages={949--958},
  year={2024},
  organization={IEEE}
}

@article{janssen1990spaces,
  title={Spaces of type W, growth of Hermite coefficients, Wigner distribution, and Bargmann transform},
  author={Janssen, AJEM and Van Eijndhoven, SJL},
  journal={Journal of Mathematical Analysis and Applications},
  volume={152},
  number={2},
  pages={368--390},
  year={1990},
  publisher={Elsevier}
}

@book{ODonnell2014,
  author       = {O'Donnell, Ryan},
  title        = {Analysis of Boolean Functions},
  publisher    = {Cambridge University Press},
  address      = {New York},
  year         = {2014},
  isbn         = {978-1-107-03857-0},
  url          = {https://doi.org/10.1017/CBO9781139814782}
}

@book{reed1980methods,
  title={Methods of modern mathematical physics I: Functional analysis},
  author={Reed, Michael and Simon, Barry},
  year={1980},
  publisher={Gulf Professional Publishing}
}

@book{reed1975ii,
  title={Methods of modern mathematical physics II: Fourier analysis, self-adjointness},
  author={Reed, Michael and Simon, Barry},
  year={1975},
  publisher={Elsevier}
}

@book{grochenig2001foundations,
  title={Foundations of time-frequency analysis},
  author={Gr{\"o}chenig, Karlheinz},
  year={2001},
  publisher={Springer Science \& Business Media}
}

@book{rudin2017fourier,
  title={Fourier analysis on groups},
  author={Rudin, Walter},
  year={2017},
  publisher={Courier Dover Publications}
}

@book{gel2013spaces,
  title={Generalized Functions II: Spaces of fundamental and generalized functions},
  author={Gelfand, Israel and Shilov, Georgiy},
  year={1968},
  publisher={Academic Press}
}

@book{jackson1930theory,
  title={The theory of approximation},
  author={Jackson, Dunham},
  volume={11},
  year={1930},
  publisher={American Mathematical Soc.}
}

@article{bargmann1962remarks,
  title={Remarks on a Hilbert space of analytic functions},
  author={Bargmann, Valentine},
  journal={Proceedings of the National Academy of Sciences},
  volume={48},
  number={2},
  pages={199--204},
  year={1962}
}

@book{bernstein1912ordre,
  title={Sur l'ordre de la meilleure approximation des fonctions continues par des polyn{\^o}mes de degr{\'e} donn{\'e}},
  author={Bernstein, Serge},
  volume={4},
  year={1912},
  publisher={Hayez, imprimeur des acad{\'e}mies royales}
}

@inproceedings{van1987functional,
  title={Functional analytic characterizations of the Gelfand-Shilov spaces S$\alpha$$\beta$},
  author={Van Eijndhoven, SJL},
  booktitle={Indagationes Mathematicae (Proceedings)},
  volume={90},
  number={2},
  pages={133--144},
  year={1987},
  organization={Elsevier}
}

@inproceedings{goel2019learning,
  title={Learning neural networks with two nonlinear layers in polynomial time},
  author={Goel, Surbhi and Klivans, Adam R},
  booktitle={Conference on Learning Theory},
  pages={1470--1499},
  year={2019},
  organization={PMLR}
}

@inproceedings{koehler2018comparative,
  title={The comparative power of relu networks and polynomial kernels in the presence of sparse latent structure},
  author={Koehler, Frederic and Risteski, Andrej},
  booktitle={International Conference on Learning Representations},
  year={2018}
}

@book{hall2013quantum,
  title={Quantum theory for mathematicians},
  author={Hall, Brian C},
  year={2013},
  publisher={Springer}
}

@book{devore1993constructive,
  title={Constructive approximation},
  author={DeVore, Ronald A and Lorentz, George G},
  volume={303},
  year={1993},
  publisher={Springer Science \& Business Media}
}

@inproceedings{koehler2024influences,
  title={Influences in Mixing Measures},
  author={Koehler, Frederic and Lifshitz, Noam and Minzer, Dor and Mossel, Elchanan},
  booktitle={Proceedings of the 56th Annual ACM Symposium on Theory of Computing},
  pages={527--536},
  year={2024}
}

@inproceedings{cordero2012hypercontractive,
  title={Hypercontractive measures, Talagrand’s inequality, and influences},
  author={Cordero-Erausquin, Dario and Ledoux, Michel},
  booktitle={Geometric Aspects of Functional Analysis: Israel Seminar 2006--2010},
  pages={169--189},
  year={2012},
  organization={Springer}
}

@article{ivanisvili2024eldan,
  title={On the Eldan-Gross inequality},
  author={Ivanisvili, Paata and Zhang, Haonan},
  journal={arXiv preprint arXiv:2407.17864},
  year={2024}
}

@article{rosenthal2020ramon,
  title={Ramon van Handel’s Remarks on the Discrete Cube},
  author={Rosenthal, Gregory},
  journal={Notes available at https://www. cs. toronto. edu/rosenthal/RvH\_discrete\_cube. pdf},
  year={2020}
}

@book{adams2003sobolev,
  title={Sobolev spaces},
  author={Adams, Robert A and Fournier, John JF},
  volume={140},
  year={2003},
  publisher={Elsevier}
}

@book{katznelson2004introduction,
  title={An introduction to harmonic analysis},
  author={Katznelson, Yitzhak},
  year={2004},
  publisher={Cambridge University Press}
}

@book{treves2016topological,
  title={Topological Vector Spaces, Distributions and Kernels: Pure and Applied Mathematics, Vol. 25},
  author={Treves, Fran{\c{c}}ois},
  volume={25},
  year={2016},
  publisher={Elsevier}
}

@book{larsen2012introduction,
  title={An introduction to the theory of multipliers},
  author={Larsen, Ronald},
  volume={175},
  year={2012},
  publisher={Springer Science \& Business Media}
}

@book{folland1999real,
  title={Real analysis: modern techniques and their applications},
  author={Folland, Gerald B},
  year={1999},
  publisher={John Wiley \& Sons}
}

@book{koosis1998logarithmic,
  title={The Logarithmic Integral: Volume 1},
  author={Koosis, Paul},
  volume={1},
  year={1998},
  publisher={Cambridge university press}
}

@article{cohen1968simple,
  title={A simple proof of the Denjoy-Carleman theorem},
  author={Cohen, PJ},
  journal={The American Mathematical Monthly},
  volume={75},
  number={1},
  pages={26--31},
  year={1968},
  publisher={Taylor \& Francis}
}

@article{trefethen2020quantifying,
  title={Quantifying the ill-conditioning of analytic continuation},
  author={Trefethen, Lloyd N},
  journal={BIT Numerical Mathematics},
  volume={60},
  number={4},
  pages={901--915},
  year={2020},
  publisher={Springer}
}

@article{demanet2019stable,
  title={Stable extrapolation of analytic functions},
  author={Demanet, Laurent and Townsend, Alex},
  journal={Foundations of Computational Mathematics},
  volume={19},
  number={2},
  pages={297--331},
  year={2019},
  publisher={Springer}
}

@article{franklin1990analytic,
  title={Analytic continuation by the fast Fourier transform},
  author={Franklin, Joel},
  journal={SIAM journal on scientific and statistical computing},
  volume={11},
  number={1},
  pages={112--122},
  year={1990},
  publisher={SIAM}
}

@article{grabovsky2021optimal,
  title={Optimal Error Estimates for Analytic Continuation in the Upper Half-Plane},
  author={Grabovsky, Yury and Hovsepyan, Narek},
  journal={Communications on Pure and Applied Mathematics},
  volume={74},
  number={1},
  pages={140--171},
  year={2021},
  publisher={Wiley Online Library}
}

@book{carleman1926fonctions,
  title={Les Fonctions quasi analytiques: le{\c{c}}ons profess{\'e}es au College de France},
  author={Carleman, Torsten},
  year={1926},
  publisher={Gauthier-Villars}
}

@book{vershynin2018high,
  title={High-dimensional probability: An introduction with applications in data science},
  author={Vershynin, Roman},
  volume={47},
  year={2018},
  publisher={Cambridge university press}
}

@article{klivans2025power,
  title={The Power of Iterative Filtering for Supervised Learning with (Heavy) Contamination},
  author={Klivans, Adam R and Stavropoulos, Konstantinos and Tian, Kevin and Vasilyan, Arsen},
  journal={arXiv preprint arXiv:2505.20177},
  year={2025}
}

@inproceedings{gupte2022continuous,
  title={Continuous lwe is as hard as lwe \& applications to learning gaussian mixtures},
  author={Gupte, Aparna and Vafa, Neekon and Vaikuntanathan, Vinod},
  booktitle={2022 IEEE 63rd Annual Symposium on Foundations of Computer Science (FOCS)},
  pages={1162--1173},
  year={2022},
  organization={IEEE}
}

@article{valiant2015finding,
  title={Finding correlations in subquadratic time, with applications to learning parities and the closest pair problem},
  author={Valiant, Gregory},
  journal={Journal of the ACM (JACM)},
  volume={62},
  number={2},
  pages={1--45},
  year={2015},
  publisher={ACM New York, NY, USA}
}

@inproceedings{bruna2021continuous,
  title={Continuous lwe},
  author={Bruna, Joan and Regev, Oded and Song, Min Jae and Tang, Yi},
  booktitle={Proceedings of the 53rd Annual ACM SIGACT Symposium on Theory of Computing},
  pages={694--707},
  year={2021}
}

@article{song2021cryptographic,
  title={On the cryptographic hardness of learning single periodic neurons},
  author={Song, Min Jae and Zadik, Ilias and Bruna, Joan},
  journal={Advances in neural information processing systems},
  volume={34},
  pages={29602--29615},
  year={2021}
}

@article{kearns1994efficient,
  title={Efficient distribution-free learning of probabilistic concepts},
  author={Kearns, Michael J and Schapire, Robert E},
  journal={Journal of Computer and System Sciences},
  volume={48},
  number={3},
  pages={464--497},
  year={1994},
  publisher={Elsevier}
}

@book{rachev2013methods,
  title={The methods of distances in the theory of probability and statistics},
  author={Rachev, Svetlozar T and Klebanov, Lev B and Stoyanov, Stoyan V and Fabozzi, Frank},
  volume={10},
  year={2013},
  publisher={Springer}
}

@article{klivans2013moment,
  title={Moment-matching polynomials},
  author={Klivans, Adam and Meka, Raghu},
  journal={arXiv preprint arXiv:1301.0820},
  year={2013}
}

@inproceedings{RubinfeldVasilyan23,
  author    = {Ronitt Rubinfeld and Arsen Vasilyan},
  title     = {Testing Distributional Assumptions of Learning Algorithms},
  booktitle = {Proceedings of the 55th Annual ACM Symposium on Theory of Computing (STOC '23)},
  year      = {2023},
  pages     = {1643--1656},
  publisher = {ACM},
  address   = {New York, NY, USA},
  doi       = {10.1145/3564246.3585117},
  url       = {https://doi.org/10.1145/3564246.3585117},
  note      = {Extended version: arXiv:2204.07196}
}

@inproceedings{GollakotaKlivansKothari23,
  author    = {Aravind Gollakota and Adam R. Klivans and Pravesh K. Kothari},
  title     = {A Moment-Matching Approach to Testable Learning and a New Characterization of {R}ademacher Complexity},
  booktitle = {Proceedings of the 55th Annual ACM Symposium on Theory of Computing (STOC '23)},
  year      = {2023},
  pages     = {1657--1670},
  publisher = {ACM},
  address   = {New York, NY, USA},
  doi       = {10.1145/3564246.3585206},
  url       = {https://doi.org/10.1145/3564246.3585206},
  note      = {Extended version: arXiv:2211.13312}
}

@article{bizeul2025polynomial,
  title={Polynomial Approximation in $L^2$ of the Double Exponential via Complex Analysis},
  author={Bizeul, Pierre and Klartag, Boaz},
  journal={arXiv preprint arXiv:2502.07448},
  year={2025}
}

@inproceedings{diakonikolas2025sos,
  title={Sos certifiability of subgaussian distributions and its algorithmic applications},
  author={Diakonikolas, Ilias and Hopkins, Samuel B and Pensia, Ankit and Tiegel, Stefan},
  booktitle={Proceedings of the 57th Annual ACM Symposium on Theory of Computing},
  pages={1689--1700},
  year={2025}
}

@article{blum2002smoothed,
  title={Smoothed analysis of the perceptron algorithm for linear programming},
  author={Blum, Avrim and Dunagan, John},
  year={2002},
  publisher={Carnegie Mellon University}
}

@book{stein2010complex,
  title={Complex analysis},
  author={Stein, Elias M and Shakarchi, Rami},
  volume={2},
  year={2010},
  publisher={Princeton University Press}
}

@article{ivanisvili2020rademacher,
  title={Rademacher type and Enflo type coincide},
  author={Ivanisvili, Paata and Van Handel, Ramon and Volberg, Alexander},
  journal={Annals of mathematics},
  volume={192},
  number={2},
  pages={665--678},
  year={2020},
  publisher={Department of Mathematics, Princeton University Princeton, New Jersey, USA}
}

@article{nazarov1993local,
  title={Local estimates for exponential polynomials and their applications to inequalities of the uncertainty principle type},
  author={Nazarov, Fedor L'vovich},
  journal={Algebra i analiz},
  volume={5},
  number={4},
  pages={3--66},
  year={1993},
  publisher={St. Petersburg Department of Steklov Institute of Mathematics, Russian~…}
}

@book{anderson2010introduction,
  title={An introduction to random matrices},
  author={Anderson, Greg W and Guionnet, Alice and Zeitouni, Ofer},
  number={118},
  year={2010},
  publisher={Cambridge university press}
}

@article{han2023universality,
  title={Universality of regularized regression estimators in high dimensions},
  author={Han, Qiyang and Shen, Yandi},
  journal={The Annals of Statistics},
  volume={51},
  number={4},
  pages={1799--1823},
  year={2023},
  publisher={Institute of Mathematical Statistics}
}

@article{bayati2015universality,
  title={Universality in polytope phase transitions and message passing algorithms},
  author={Bayati, Mohsen and Lelarge, Marc and Montanari, Andrea},
  year={2015}
}

@article{talagrand2014upper,
  title={Upper and lower bounds for stochastic processes},
  author={Talagrand, Michel},
  year={2014},
  publisher={Springer}
}

@article{freud1977markov,
  title={On Markov-Bernstein-type inequalities and their applications},
  author={Freud, G{\'e}za},
  journal={Journal of Approximation Theory},
  volume={19},
  number={1},
  pages={22--37},
  year={1977},
  publisher={Academic Press}
}

@article{ditzian1987polynomial,
  title={Polynomial approximation with exponential weights},
  author={Ditzian, Z and Lubinsky, DS and Nevai, P and Totik, V},
  journal={Acta Mathematica Hungarica},
  volume={50},
  number={1-2},
  pages={165--175},
  year={1987},
  publisher={Akad{\'e}miai Kiad{\'o}, co-published with Springer Science+ Business Media BV~…}
}

@article{chandrasekaran2025learning,
  title={Learning neural networks with distribution shift: Efficiently certifiable guarantees},
  author={Chandrasekaran, Gautam and Klivans, Adam R and Lee, Lin Lin and Stavropoulos, Konstantinos},
  journal={arXiv preprint arXiv:2502.16021},
  year={2025}
}

@article{mastroianni2008vallee,
  title={De la Vall{\'e}e Poussin means and Jackson's theorem},
  author={Mastroianni, G and Themistoclakis, Woula and others},
  journal={Acta Scientiarum Mathematicarum},
  volume={74},
  number={1-2},
  pages={147--170},
  year={2008},
  publisher={Szeged [Hungary]: M. Kir. Ferencz Jozsef-Tudomanyegyetem Baratai~…}
}

@inproceedings{joo1988answer,
  title={Answer to a problem of Paul Tur{\'a}n},
  author={Jo{\'o}, I and Ky, NX},
  booktitle={Annales Univ. Sci. Budapest., Sectio Math},
  volume={31},
  pages={229--241},
  year={1988}
}

@article{freud1978approximation,
  title={Sur l'approximation polynomiale avec poids exp (-| x|)},
  author={Freud, G and Giroux, A and Rahman, QI},
  journal={Canadian Journal of Mathematics},
  volume={30},
  number={2},
  pages={358--372},
  year={1978},
  publisher={Cambridge University Press}
}

@article{goel2020boltzmann,
  title={From boltzmann machines to neural networks and back again},
  author={Goel, Surbhi and Klivans, Adam and Koehler, Frederic},
  journal={Advances in neural information processing systems},
  volume={33},
  pages={6354--6365},
  year={2020}
}

@article{lubinsky2006jackson,
  title={Jackson and Bernstein theorems for the weight exp (-| x|) on $\mathbb R$},
  author={Lubinsky, DS},
  journal={Israel Journal of Mathematics},
  volume={153},
  number={1},
  pages={193--219},
  year={2006},
  publisher={Springer}
}

@article{lubinsky2006weights,
  title={Which weights on $\mathbb R$ admit Jackson Theorems?},
  author={Lubinsky, DS},
  journal={Israel Journal of Mathematics},
  volume={155},
  number={1},
  pages={253--280},
  year={2006},
  publisher={Springer}
}

@article{berger1986toeplitz,
  title={Toeplitz operators and quantum mechanics},
  author={Berger, CA and Coburn, LA},
  journal={Journal of Functional Analysis},
  volume={68},
  number={3},
  pages={273--299},
  year={1986},
  publisher={Elsevier}
}

@article{folland1997uncertainty,
  title={The uncertainty principle: a mathematical survey},
  author={Folland, Gerald B and Sitaram, Alladi},
  journal={Journal of Fourier analysis and applications},
  volume={3},
  number={3},
  pages={207--238},
  year={1997},
  publisher={Springer}
}

@incollection{tsybakov2008nonparametric,
  title={Nonparametric estimators},
  author={Tsybakov, Alexandre B},
  booktitle={Introduction to Nonparametric Estimation},
  pages={1--76},
  year={2008},
  publisher={Springer}
}

@article{kashpirovsky1978equality,
  title={Equality of the spaces S$\beta$ $\alpha$ and S$\alpha \cap$ S$\beta$},
  author={Kashpirovsky, AI},
  journal={Functional Anal. Appl},
  volume={14},
  pages={60},
  year={1978}
}
\appendix

\section{Other basic tools}\label{app:auxillary}    




\subsection{Standard bump functions}
Fix $R>1$. Let
\[
\psi(t) \;=\; 
\begin{cases}
e^{-1/t}, & t>0,\\[2pt]
0, & t\le 0,
\end{cases}
\qquad 
\text{and}\qquad
\vartheta(s) \;=\; \frac{\psi(s)}{\psi(s)+\psi(1-s)}\,,\quad s\in\R.
\]
Then $\vartheta\in C^\infty(\R)$, $\vartheta(s)=0$ for $s\le 0$, $\vartheta(s)=1$ for $s\ge 1$, and $0\le \vartheta\le 1$.
Define the (even) cutoff 
\[
\chi_R(x) \;=\; \vartheta\bigl(R-|x|\bigr), \qquad x\in\R.
\]
By construction, $\chi_R\in C^\infty(\R)$, $\supp(\chi_R)\subset[-R,R]$, $0\le \chi_R\le 1$, and 
\[
\chi_R(x)=1 \quad \text{for all } x\in [-R+1,\,R-1].
\]
Thus $\chi_R$ is a standard bump with unit-thickness transition layers on $[R-1,R]$ and $[-R,-R+1]$.


\subsection{Concentration inequalities}

\begin{lemma}[Robust concentration]\label{lem:robust-concentration}
Suppose that $Z_i \in \RR, i\le n$ are i.i.d. samples drawn from a distribution such that $\EE[|Z_i|^k] < \infty $ for some $k\ge 2$.  
Let $\mu = \EE[Z_i]$ be the mean of the distribution.  
Consider the truncated empirical mean $\hat \mu_M   = n^{-1} \sum_{i\le n} \cT_M(Z_i)$, where $\cT_M(z) = \max\{\min\{z,M\}, -M\}$ is the truncation operator at level $M>0$. 
Then with $M = \Big((k-1) \EE[|Z|^k] \cdot \sqrt{4\log(2/\delta) / n }\Big)^{ 1/k} $, it holds with probability at least $1-\delta$ that  
\begin{align}
    |\hat \mu_M - \mu| &\le  \Big(2\sqrt{\frac{\log(2/\delta)}{n}}\Big)^{\frac{k-1}{k }} \,   \frac{k\EE[|Z_i|^k ]^{1/k}}{(k-1)^{1-1/k}}.   
\end{align}
\end{lemma} 

\begin{proof}
We fix $\delta>0$ throughout the proof and $M>0$ to be determined later.  
Since each $\cT_M(Z_i)$ is bounded in $[-M,M]$, we can apply Hoeffding's inequality to obtain that with probability at least $1-\delta$, that 
\begin{align}
    |\hat \mu_M - \EE[\cT_M(Z_i)]| &\le  2M \sqrt{\frac{\log(2/\delta)}{n}}.  
\end{align}
On the other hand, we can bound the bias term as 
\begin{align}
    \big| \EE[\cT_M(Z_i)] - \mu \big|  &\le   \EE[|Z_i|\,  \ind\{|Z_i| > M\}] \le \EE[|Z_i|^k] / M^{k-1}.  
\end{align}
Therefore, choosing $M = \Big((k-1) \EE[|Z|^k] \cdot \sqrt{4\log(2/\delta) / n }\Big)^{ 1/k} $ yields that
\begin{align}
    |\hat \mu - \mu| &\le  \frac{2\EE[|Z_i|^k]}{(k-1) M^{k-1}} + 4 M \sqrt{\frac{\log(2/\delta)}{n}} =  \Big(2\sqrt{\frac{\log(2/\delta)}{n}}\Big)^{\frac{k-1}{k }} \,   \frac{k\EE[|Z_i|^k ]^{1/k}}{(k-1)^{1-1/k}}.    
\end{align}  

\end{proof}


\subsection{VC theory and uniform convergence}

\begin{definition}[VC dimension]\label{def:vc-dimension}
Let $\mathcal{H}$ be a class of binary classifiers on a space $Z$. 
A set $S \subseteq Z$ is \emph{shattered} by $\mathcal{H}$ if for every subset $T \subseteq S$, there exists $h \in \mathcal{H}$ such that $h(z) = 1$ for $z \in T$ and $h(z) = 0$ for $z \in S \setminus T$. 
The \emph{VC-dimension} of $\mathcal{H}$, denoted as $\mathsf{VCdim}(\mathcal{H})$, is the size of the largest set shattered by $\mathcal{H}$. 
If arbitrarily large sets can be shattered, $\mathsf{VCdim}(\mathcal{H}) = \infty$.
\end{definition}

\begin{theorem}[Uniform convergence, Theorem 6.8 \cite{shalev2014understanding}]\label{thm:vc-uniform-convergence}
Let $\mathcal{H}$ be a class of binary classifiers with VC-dimension $V < \infty$. For i.i.d.\ samples $\{z_i\}_{i \le n}$ drawn from a distribution $\mu$, define the empirical and true error by $\hat{R}_n(h) = n^{-1}\sum_{i \le n} h(z_i)$ and $R(h) = \EE_{z \sim \mu}[h(z)]$ respectively. Then for any $\delta > 0$, with probability at least $1-\delta$ over the sample, uniformly for all $h \in \mathcal{H}$:
\[
|R(h) - \hat{R}_n(h)| \lesssim \sqrt{\frac{V \log(n/V) + \log(1/\delta)}{n}}.
\]
\end{theorem}

The second lemma generalizes the VC-based uniform convergence result to general real-valued loss functions with finite VC-dimension. 
\begin{lemma}[\cite{vapnik2006estimation}, Assertion 2 in Chapter 7.8]\label{lem:approx-lemma-vapnik}  
Let $Q_\alpha(z)$ be a class of functions indexed by $\alpha\in \Lambda$, such that: 
\begin{enumerate}
    \item For each $\alpha\in \Lambda$, $Q_\alpha(z)$ is a nonnegative and measurable function
    \item Set collection $\{Q_\alpha(z)\}_{\alpha}$ has finite VC-dimension $V$, in the sense that the binary classifiers $\big\{z\mapsto \ind\{Q_\alpha(z) \ge t\}\big\}_{\alpha,t}$ has VC-dimension $V$.   
    \item For all $\alpha$, $\EE_z[Q_\alpha(z)^2] < \infty$. 
\end{enumerate} 
Then for i.i.d. dataset $\{z_i\}_{i\le n}$ drawn from a distribution such that $\EE_z[Q_\alpha(z)^2] < \infty$ for all $\alpha$, we have that for any $t > 0$ that
\begin{align}
    \PP\Big(\sup_{\alpha} \frac{\EE_z [Q_\alpha(z)]  - n^{-1}\sum_{i\le n} Q_\alpha(z_i)}{ \EE_z[Q_\alpha(z)^2]^{1/2}} >t\,  \sqrt{(1-\log t /2 )} \Big) \le 12 \frac{(2n)^V}{V!} \exp(-nt^2 /4).  
\end{align}
As a consequence, when $n \ge  V\log (2en/V) +\log(12/\delta)$, then with probability at least $1-\delta$, we have for any $\alpha\in \Lambda$ that 
\begin{align}
     \EE_z [Q_\alpha(z)]  \le  n^{-1}\sum_{i\le n} Q_\alpha(z_i) + 2 \EE[Q_\alpha(z)^2]^{1/2}\,   \sqrt{\frac{V\log(2en/V) + \log(12/\delta)}{n}}. 
\end{align}
 
\end{lemma}

\section{Deferred proofs from \cref{sec:app-smoothed-analysis}}\label{apdx:proof-smoothed}  

To control related quantities, we frequently involve the following form of Stirling's approximation from \citep[Eq. 112]{gronwall1918gamma}
\begin{lemma}[Stirling's approximation]\label{lem:stirling}
For any $x>0$, it holds that
\begin{align} 
  \sqrt{ \frac {2\pi} {z}} \Big(\frac {z} {e}\Big)^z  <\Gamma(x) < \sqrt{ \frac {2\pi} {z }} \Big(\frac{z }{e}\Big)^z \exp\{1/(12z)\}.  
\end{align} 
\end{lemma}

\subsection{Proofs of polynomial approximation results}\label{apdx:proof-poly-approx}
\begin{proof}[Proof of \cref{lem:gaussian-spatial-truncation}]\label{proof:gaussian-spatial-truncation}
  In the following, we assume that $x\sim \mu$ and $z\sim \gamma_d$ are independent. 
For any fixed $R>0$ and $x\in \RR^d$, it holds by the linearity of $T_\sigma$ and boundedness of $g$ that 
\begin{align}
  |T_\sigma (f\cdot \ind_{B_R})(x) - T_\sigma f(x)|  
    & = \big|\EE_{z }[f(x+\sigma z) \ind \{\|x+\sigma z\|\ge R \}]\big| \\  
    & \le  \EE_{z}[|f(x+\sigma z)| \cdot \ind_{\{\|x+\sigma z\|_2 > R\}}] \\  
    &\le M \cdot \PP_{z}(\|x+\sigma z\|_2 > R).  
\end{align}
Hence, the squared $L^2(\mu)$ error can be bounded as 
\begin{align}
\norm{T_\sigma (f\cdot \ind_{B_R}) - T_\sigma f}_{\mu}^2 &\le  M^2   \EE_x\big[\EE_z[\ind \{ \big\|x+\sigma z\|_2 >R \} > R ]^2 \big] \\ 
&\le M^2\, \PP_{x,z}(\norm{x+  \sigma z}_2 >R)  \\  
&\le M^2\,\PP( \norm{x}_2  +  \sigma \norm{z}_2 >R) \\  
&\le M^2\, \Big(\PP\big( \norm{x}_2 > \frac{bR}{\sigma + b} \big) +\PP\big( \norm{z} _2 > \frac{R}{ \sigma  +b} \big)\Big).   
\end{align}  
Here, the second inequality follows from Jensen's inequality.  
For the right-hand side above, it suffices to apply the tail bound to  $x$ and $z $ respectively. 
We set  $c' = \min \{c,2\}$ in the sequel.
Note that the covering number of the sphere is $\cN_{\|\cdot\|_2}(  \SS^{d-1},1/2) \le 5^d$, it holds by the union bound and the maximal inequality  that 
\begin{align}
  \PP( \norm{x}_2 >  \frac{bR}{\sigma + b}) + \PP( \norm{z}_2 >  \frac{R}{\sigma + b})  
  &\le  A \cdot 5^d \,  \Big(\exp\big\{- \frac{R^c }{  2^c(\sigma+ b)^c} \big\} + \exp \big\{  -\frac{R^2}{8(\sigma+ b)^2 }\big\}\Big)   \\ 
  &\le 2A \cdot   5^d \, \exp\big\{- \frac{R^{c' } }{  2^{c'}(\sigma + b)^{c' }} \big\},
\end{align}
when $R > 2(\sigma + b)$. 
Thus, when  setting $R> 2(\sigma + b) \cdot \big(\log (2A  / \eps^2) +d \log 5 \big)^{1/c'}$, last inequality implies the desired result.

\end{proof}


\begin{proof}[Proof of \cref{prop:poly-approx}] \label{proof:poly-approx}

Since $f\in \cH(k)$, there exist a function $g:\RR^k \to \{\pm 1\}$ and a matrix $U\in \RR^{k\times d}$ such that $f(x) = g(Ux)$ and $UU^T = I_k$. 
Therefore, 
\begin{align}
  T_\sigma f(x) &= \EE_{z\sim \gamma_d}[g(U(x+\sigma z))] \\ 
   &= \EE_{z\sim \gamma_k}[g(Ux + \sigma z)] \\ 
   &= T_\sigma g(Ux). 
\end{align} 
On the other hand, for both strictly sub-exponential and sub-exponential cases, push forward $U\# \mu$ on $\RR^k$ satisfy the same condition on $\RR^k$ as the original distribution $\mu$ on $\RR^d$. 
Suppose that we find a polynomial $Q$ on $\RR^k$ such that $\norm{g - Q}_{U\# \mu} \le \eps$, then it holds for $P(x) = Q(Ux)$ that
\begin{align}
  \norm{T_\sigma f - P }_{\mu} &= \norm{T_\sigma g - Q(x)}_{U \# \mu} \le \eps.   
\end{align}
Hence, we can assume that $k=d $ and $f:\RR^d\to \{\pm 1 \}$ without loss of generality. 
The subsequent proof in the following separates into three parts: strictly sub-exponential cases with $r>1$, $r=1$, and sub-exponential case.

\paragraph{Strictly sub-exponential bound with $r>1$.} Note that $T_\sigma f$ is a Gaussian analytic function, it suffices to verify that measure $\mu$ satisfies the tail‐bound condition before applying \cref{lem:gaussian-spatial-truncation}.  
Using a Chernoff bound, it holds for any $t>0$ and $u\in\SS^{d-1}$ that,  
\begin{align}
  \PP (\dotp{x}{u} >t) &\le \inf_{\lambda>0} e^{-\lambda t} \mathbb{E}_{x\sim \mu}\big[ e^{\lambda \dotp{x}{u}} \big] \\ 
  &\le \inf_{\lambda>0}  A\exp\Bigl((K|\lambda|)^{r} -\lambda t\Bigr).  
\end{align}
Setting $\lambda= (t/r )^{1/(r-1)}\, K^{-r/(r-1)}$in the right-hand side yields that
\begin{align}
  \PP (\dotp{x}{u} >t) \le A\exp \Big( - (r-1) \cdot  \Big(\frac{t}{Kr}\Big)^{r/(r-1)} \Big). 
\end{align}
By symmetry, we can verify that the condition in \cref{eq:tail-bound} holds with $c=r/(r-1)$ and $b = Kr\cdot (r-1)^{r^{-1} - 1} $. 
We then apply Lemma~\ref{lem:gaussian-spatial-truncation} to obtain that 
$
  \bigl\|\,T_\sigma f - T_\sigma(f\cdot \ind_{B_R})\bigr\|_{\mu}\le \varepsilon /2  
$ whenever 
\begin{align}
  R > 2(\sigma+b) \cdot \big(2\log (4 / \eps) +d \log 5 \big)^{1/c'}.  \label{eq:R-bound-strict}
\end{align}
Here the decaying exponent is $c' = \min\{r/(r-1),2\}$.  
In the sequel, we persist with the notation $b,c'$ for simplicity.  

The second part is to approximate $T_\sigma(f \cdot \ind_{B_R})$ with polynomials using \cref{thm:Rd-pw-strict}. 
We begin with controlling $|\rho|(B_\Omega^c)$. 
Denote $V_d =  \frac{\pi^{d/2}}{\Gamma(d/2 + 1)}$ as the volume of the unit ball in $\RR^d$ and $S_{d-1} = \frac{2\pi^{d/2}}{\Gamma(d/2)}$ as the surface area of the unit sphere $\SS^{d-1}$. 
Note that the Fourier transform of $T_\sigma(f \cdot \ind_{B_R})$ is $ \cF[T_\sigma(f \cdot \ind_{B_R})](\xi) =  \cF[f\cdot \ind_{B_R}](\xi)\cdot \exp\{-\sigma^2 \|\xi\|_2^2/2\}$ and that $\big|\cF[f \cdot \ind_{B_R}](\xi)\big|\le  V_d \cdot R^d$, and we can control the exterior volume $|\rho|(B_\Omega^c)$ as 
\begin{align}
    |\rho| (B_\Omega^c ) &\le  \int_{\RR^d \setminus B_\Omega} |\cF[T_\sigma(f \cdot \ind_{B_R})](\xi)|d\xi \\  
    &\le  V_d \cdot R^d \cdot  \int_{\RR^d \setminus B_\Omega} \exp\{-\sigma^2 \|\xi\|_2^2/2\} d\xi \\ 
  &= V_d\cdot  R^{d} \cdot S_{d-1} \sigma^{-d} \int_{\sigma\Omega}^\infty r^{d-1} \exp\{ - r^2/2\}dr . \label{eq:rho-B-Omega-c-bound-1}
\end{align}
To continue bounding the term above, we split the integral into two parts using Cauchy-Schwarz inequality:
\begin{align}
  \int_{\sigma\Omega}^\infty r^{d-1} \exp\{ - r^2/2\}dr &\le  \Bigl( \int_{\sigma\Omega}^\infty r^{2d-2} \exp\{ - r^2/2\}dr \Bigr)^{1/2} \cdot \Bigl( \int_{\sigma\Omega}^\infty \exp\{ - r^2/2\}dr \Bigr)^{1/2} \\ 
  &\le 2^{d/2- 3/4}\Big(\int_0^\infty r^{d-  3/2}\exp \{-r\} dr \Big)^{1/2}\cdot \sqrt{\frac \pi 2} \exp\{-\sigma^2\Omega^2/4\} \\  
  &\le 2^{d/2- 5/4}\sqrt{\pi }  \,  \Gamma(d - 1/2)^{1/2} \, \exp\{-\sigma^2\Omega^2/4\}. \label{eq:rho-B-Omega-c-bound-2}
\end{align}
Using Stirling's approximation in \cref{lem:stirling}, we have that 
\begin{align}  
  \frac{\Gamma(d-1/2)^{1/2}}{\Gamma(d/2)\cdot \Gamma(d/2+1)} &\le  \frac{\Gamma(d)^{1/2}}{\Gamma(d/2)^2 d/2 } \\  
  &\le \frac{\big(d/e\big)^{d/2} \cdot \big(2\pi/ d\big)^{1/4}}{  (d/2e)^d  \cdot (2\pi )}\exp\{1/(24d)\} \\ 
  &\le d^{-d/2 -1/4}\cdot (4e)^{d/2} \cdot (2\pi)^{-3/4}e.   \label{eq:stirling-bound}
\end{align}
Combing \cref{eq:rho-B-Omega-c-bound-1}, \cref{eq:rho-B-Omega-c-bound-2} and \cref{eq:stirling-bound}, we have that 
\begin{align}
  \frac{|\rho| (B_\Omega^c ) }{2\pi }&\le  \sigma^{-d} \exp \{ -\sigma^2\Omega^2/4\}\, R^d \cdot  2^{d/2- 5/4}   \cdot 2\pi^{d+1/2}\,\frac{\Gamma(d-1/2)^{1/2}}{\Gamma(d/2)\cdot \Gamma(d/2+1)} \\  
  &\le\frac{e}{2\pi^{1/4}}\, \sigma^{-d} \exp \{ -\sigma^2\Omega^2/4\}\,  \cdot d^{-d/2}\cdot (8\pi^2 e R^2 )^{d/2}.  \label{eq:rho-B-Omega-c-bound}
\end{align}
Therefore, there exists a constant $C_1>0$ such that when
\begin{align}
R & = 2(\sigma+b) ( 2 \log (4/\eps) + d\log 5)^{1/c'};  \label{eq:R-choice}\\  
\Omega &\simeq  \big(\log (1/\eps) + d \log(R^2 / (\sigma^2 d))\big)^{1/2}\sigma^{-1},
\end{align}
the requirement in \cref{eq:R-bound-strict} is satisfied, and it holds that $|\rho|(B_\Omega^c) /(2\pi)\le \eps/4$.  
 
The final step is to settle the degree such that the first term in the upper bound of best polynomial approximation error in \cref{thm:Rd-pw-strict} is small. 
In order to comply with the conditions in \cref{thm:Rd-pw-strict}, we require that $D>r(K\Omega)^r$. 
Note that  $|\cF[T_\sigma(f \cdot \ind_{B_R})](\xi)|$ is bounded by $V_d R^d \exp\{-\sigma^2 \|\xi\|_2^2/2\}$, we have that 
\begin{align}
  |\rho|(B_\Omega)\;\cdot\;\bigl(er/D\bigr)^{D/r}\,(K\Omega)^{D}
  &\le \int_{B_\Omega} \bigl|\cF[T_\sigma(f\cdot\ind_{B_R})](\xi)\bigr|\,d\xi
       \;\cdot\;\bigl(er/D\bigr)^{D/r}\,(K\Omega)^{D} \\[0.5ex]
  &\le V_{d}\,(2\pi\sigma^{2}R^{2})^{d/2}\;\bigl(er/D\bigr)^{D/r}\,(K\Omega)^{D} \\
  &\le  (4\pi\sigma R\cdot d^{-1/2})^{d}  \bigl(er/D\bigr)^{D/r}\,(K\Omega)^{D}. \label{eq:rho-Omega-bound}
\end{align}
where the last line holds by \cref{lem:stirling}. 
We now set $D =  er( K\Omega)^{r} + r\log\big(4(4\pi \sigma R d^{-1/2})^d /\eps \big)$, under which it holds that
\begin{align}
  \text{\eqref{eq:rho-Omega-bound}} &\le  \frac{\eps}{4} + \exp\Big(- \frac{D}{r} \log \frac{D}{er(K\Omega)^r} +  \log \frac{4(4\pi \sigma R d^{-1/2})^d}{\eps}\Big) \\ 
  &\le \frac{\eps}{4} + \exp\Big(- \frac{D}{r} \Big(1- \frac{er(K\Omega )}{D}\Big) +  \log \frac{4(4\pi \sigma R d^{-1/2})^d}{\eps}\Big) \\ 
  &\le  \frac{\eps}{4} + \exp\Big(-  \log \frac{4(4\pi\sigma Rd^{-1/2})^d}{\eps} +  \log \frac{4(4\pi \sigma R d^{-1/2})^d}{\eps}\Big) \le \frac{\eps}{4}. \label{eq:rho-Omega-bound-final}
\end{align}
Here the second line holds from the inequality $\log x \ge 1 - 1/x$ for any $x>0$, and the third line follows from the choice of $D$.  
When $r>2$, we have that $c' = r/(r-1) <2$, 
In this case, we have that 
\begin{align}
  R&\simeq (\log(1/\eps) +d)^{1/c'}; \\ 
  \Omega &\simeq  \big(\log (1/\eps) + d \log (d/\sigma)\big)^{1/2}\sigma^{-1}; \\ 
  D & \simeq \sigma^{-r} \big(\log(\eps^{-1}) + d \log (d/\sigma) \big)^{r/2} .   
\end{align}
On the other hand, when $1<r\le 2$, we have that $c' = 2$. 
And we have that 
\begin{align}
  R&\simeq (\log(1/\eps) +d)^{1/2}; \\ 
  \Omega &\simeq  \big(\log (1/\eps) + d \log(1/\sigma)\big)^{1/2}\sigma^{-1};\\ 
  D & \simeq \sigma^{-r} \big(\log(\eps^{-1}) + d\log(1/\sigma\big)^{r/2} .    
\end{align}
This concludes the proof when $r>1$. 

\paragraph{Strictly sub-exponential bound with $r=1$.}
 In this case, we claim that $\supp(\mu) \subset B_{K }$. 
To see this, note that for any $M>K$ and $u\in\SS^{d-1}$, we have that 
\begin{align}
   \PP_\mu (|\dotp{u}{x}|>M) &\le \inf_{\lambda>0} \e^{-\lambda M} \mathbb{E}_{x\sim \mu}\big[ e^{\lambda \dotp{u}{x}} \big] \\  
    &\le \inf_{\lambda>0}  \exp\Bigl(K\lambda-\lambda M\Bigr) = 0. 
\end{align}
Since the distribution is bounded in $B_{K}$ the conditions in \cref{lem:gaussian-spatial-truncation} are satisfied with $A=2$, $b = K$ and any $c>0$. 
Therefore, choosing $R=2(\sigma + K)\cdot  \sqrt{2\log (2/\eps) + d\log 5}$ yields that $\norm{T_\sigma f - T_\sigma(f \cdot \ind_{B_R})}_{L^2(\mu)} \le \eps/2$. 

For the second part, note that \cref{eq:rho-B-Omega-c-bound} still holds, and when 
\begin{align}
  \Omega &\simeq \sigma^{-1} \Big( \log (1/ \eps) + d \big(\log\log(\eps^{-1})\big)\Big)^{1/2}, \label{eq:Omega-bound-strict-bounded}
\end{align}
it holds that $|\rho|(B_\Omega^c)/(2\pi) \le \eps/4$ as in \cref{eq:rho-B-Omega-c-bound}. 
On the other hand, similar to \cref{eq:rho-Omega-bound-final}, setting $D = e (K\Omega) + \log\big(4(4\pi \sigma R d^{-1/2})^d /\eps \big)$ yields that $\norm{T_\sigma(f \cdot \ind_{B_R}) - P}_{L^2(\mu)} \le \eps/2$ for some polynomial $P$ of degree $D$. 
In this case 
\begin{align}
   D\simeq  \sigma^{-r} \Big( \log (1/ \eps) + d\log(1/\sigma)\Big)^{r/2}. 
\end{align}
Applying the triangle inequality, we conclude the existence of the desired polynomial approximation.

\paragraph{Bound for sub-exponential input.} 
We apply \cref{lem:gaussian-spatial-truncation}  to truncate the tail of $\mu$. 
For fixed $u\in \SS^{d-1}$, the Markov’s inequality implies that 
\begin{align}
  \PP(\dotp{x}{u} > t) &\le  \EE[ \exp\{2\dotp{x}{u}/K - 2t /K\}] \\
  &\le e^2  \exp\Bigl( - \frac{t}{K}\Bigr). 
\end{align}
By symmetry, we have that \(\mu\) satisfies the hypothesis of Lemma~\ref{lem:gaussian-spatial-truncation} with $A = 2e^2$, $c=1$ and $b=K$. 
Therefore, we have that  
\begin{align}
  \bigl\|\,T_\sigma f - T_\sigma(f \cdot \ind_{B_R})\bigr\|_{L^{2}(\mu)} &\le   \varepsilon/2, 
\end{align} 
whenever  $R > 2(\sigma+K)\cdot ( 2\log (4e/ \eps  )+ d\log 5 )$.  

To approximate $T_\sigma(f \cdot \ind_{B_R})$ using polynomials with \cref{lem:complex-Rd}, we replicate the second part of the proof of the strictly sub-exponential bound that controls the exterior volume $|\rho|(B_\Omega^c)$.  
Recall that for fixed $R, \Omega$,  \cref{eq:rho-B-Omega-c-bound} asserts that
\begin{align}
\frac{|\rho| (B_\Omega^c ) }{2\pi } &\le\frac{e}{2\pi^{1/4}}\,  \exp \{ -\sigma^2\Omega^2/4\}\,  \cdot d^{-d/2}\cdot (8\pi^2 e R^2 /\sigma^2)^{d/2}.  \label{eq:rho-B-Omega-c-bound-replicated} 
\end{align}
Then under the choice 
\begin{align}
R &= 4(\sigma+K ) \cdot (2\log (4e/ \eps) + d\log 5 ), \\ 
\Omega &\simeq \sigma^{-1}\Big(\log(\eps^{-1})+ d  \log (d/\sigma)\Big)^{1/2}, \label{eq:Omega-bound-subexp} 
\end{align}
we have that $|\rho|(B_\Omega^c)/(2\pi) \le \eps/4$ as \cref{eq:rho-B-Omega-c-bound-replicated} implies. 
For the first term in the upper bound in \cref{lem:complex-Rd}, we first note that 
\begin{align}
  \tanh(t)^D \le \exp\{ D\log \tanh (t)\} \le  \exp \{ D(\tanh(t) - 1)\}= \exp \Big\{ - D \frac{ 2}{1+e^{2t}}\Big\}. 
\end{align}
Then the first term in the upper bound in \cref{lem:complex-Rd} can be controlled as 
\begin{align}
 \frac{ e|\rho|(B_\Omega)}{2\pi} \tanh \Big(\frac{K\pi\Omega}{4}\Big)^D &\le  \frac{e}{2\pi} \cdot V_d\cdot(\sqrt{2\pi}\sigma  R)^d \cdot \exp\Big(- \frac{2D }{  e^{K\pi \Omega/ 2} +1 }\Big) \\ 
 &\le \frac{e}{2\pi} \cdot\big(2\pi\sigma R/ \sqrt{d}\big)^d \cdot \exp\Big(- \frac{D }{  e^{K\pi \Omega/ 2} }\Big) .  
\end{align}
Therefore, when  
\begin{align}
  D\gtrsim  \big( \eps^{-1}\cdot  (d/\sigma)^d\big)^{ O(\sigma^{-1})}  \cdot  \Big(\log(\eps^{-1} )+ d \big( \log\log(\eps^{-1}) \vee \log d\big)\Big),  
\end{align}
it holds that $ \frac{ e|\rho|(B_\Omega)}{2\pi} \tanh \Big(\frac{K\pi\Omega}{4}\Big)^D \le \eps/4$, and there exists a polynomial $P$ of degree $D$ such that
$\norm{T_\sigma(f \cdot \ind_{B_R}) - P}_{L^2(\mu )}<\eps/4  $.  
 Using \cref{lem:complex-Rd} and triangle inequality, we can conclude that $\norm{T_\sigma f - P}_{L^2(\mu)} \le \eps$. 

\end{proof}

\subsection{Proofs of polynomial regression results} \label{app:poly-regression-proofs}
\begin{proof}[Proof of \cref{lem:poly-regression-learns}]
Let $f^*$ be the target function in $\cH(k)$ such that 
\begin{align}
\PP_{z,(x,y)}(f^*(x+\sigma z) \neq y ) \leq \mathrm{opt}_{\sigma,\cH(k)} + \eps/4. \label{eq:target-f-approx} 
\end{align}
By the assumption, there exists a polynomial $P^*$ of degree $D$ such that $\|T_\sigma f^* - P^*\|_{L^2(\mu)} \le \eps/4$. 
And $T_\sigma f^*$ follows that  
\begin{align}
  \EE_{(x,y)\sim \mu}[ |T_\sigma f^*(x) - y|] &\le  \EE_{(x,y),z}[| f^*(x+ \sigma z) - y|] \\ 
  &\le 2\mathrm{opt}_{\sigma,\cH(k)} + \eps/2. \label{eq:target-f-approx-L1}  
\end{align}

\paragraph{Single-shot guarantee.}
Given dataset $\cD$, we consider the output $\hath_\cD$ of \cref{alg:poly-regression}. 
We relate the misclassification error of $\hath_\cD$ to the $L^1$ predicting error of $\hatP_\cD$. 
Note that when $y_i=1$,  $ \mathrm{sign}(\hat  P_\cD(x_i) - t )\neq  y_i$ if and only if $t$ lies in the interval $[ -1, -\hat P_\cD(x_i)]$ (we assume that the interval is empty if $\hat P_\cD(x_i) \ge 1$).   
Therefore, we can represent the classification error using $\mathrm{sign}(\hat P_\cD - t)$ as
\begin{align}
  \frac{1}{n} \sum_{i=1}^n \ind\{\mathrm{sign}(\hat P_\cD(x_i) - t)\neq y_i\} &= \frac{1}{n} \sum_{i=1}^n \ind\{ -1 \le t \le -\hat P_\cD(x_i) ,y_i=1\}  \\ & \qquad + \frac{1}{n} \sum_{i=1}^n \ind\{-\hat P_\cD(x_i) \le t \le 1, y_i=-1\}. 
\end{align}
Since $\hat t$ is chosen to minimize the left-hand side, we have the following average case relaxation
\begin{align}
  \frac{1}{n}\sum_{i=1}^n \ind\{\hath_\cD(x_i) \neq y_i\} &\le \frac{1}{2} \int_{-1}^1 \Big( \frac{1}{n} \sum_{i=1}^n \ind\{ -1 \le t \le -\hat P_\cD(x_i) ,y_i=1\} \\ &\qquad + \frac{1}{n} \sum_{i=1}^n \ind\{-\hat P_\cD(x_i) \le t \le 1, y_i=-1\}\Big) dt \\ 
  &\le \frac{1}{2n} \sum_{i=1}^n |\hat P_\cD(x_i) - y_i|.  \label{eq:poly-regression-insample-err-1}
\end{align}
Using the $L^1$ optimality of $\hat P_\cD$ in $\cP_D$, \cref{eq:poly-regression-insample-err-1} implies that
\begin{align}
  \frac{1}{n} \sum_{i=1}^n \ind\{\hath_\cD(x_i)\neq y_i\} &\le \frac{1}{2n} \sum_{i=1}^n |P^*(x_i) - y_i| \\ 
  & \le \frac{1}{2n}\sum_{i=1}^n \Big( |P^*(x_i) - T_\sigma f^*(x_i)| + |T_\sigma f^*(x_i) - y_i|\Big).  \label{eq:poly-regression-insample-err-2}
\end{align} 
For the left-hand side above, we note that the space of all polynomial threshold functions with degree at most $D$ has VC dimension $O(m \log m)$, where $m = \binom{d+D}{D}$ is the number of monomials of degree at most $D$. 
Therefore, when $n = \Omega(\eps^{-2}m \log (m))$, uniform convergence of 
\begin{align}
\Big|\PP_{(x,y)\sim \mu}(\hath_\cD(x) \neq y) - \frac{1}{n} \sum_{i=1}^n \ind\{\hath_\cD(x_i)\neq y_i\} \Big|  &\le  \frac{\eps}{4}, \label{eq:poly-regression-insample-err-3}
\end{align} 
Now we turn to bound the right-hand side of \cref{eq:poly-regression-insample-err-2}.  
For the first term, it holds by Chebyshev's inequality that  
\begin{align}
    \PP \Big(\frac{1}{n} \sum_{i= 1}^n  \big| T_\sigma f^*(x_i) - P^* (x_i)\big| \ge \EE_x[| T_\sigma f^*(x) - P^* (x)\big|] + \frac{\eps}{4 }\Big) &\le \frac{16 \EE_x[| T_\sigma f^*(x) - P^* (x)\big|^2]}{n\eps^2} .  
\end{align}
Therefore, when $n = \Omega(\eps^{-2 })$, we have with probability at least $11/12$ that  
\begin{align}
  \frac{1}{n} \sum_{i=1}^n |T_\sigma f^*(x_i) - P^*(x_i)| &\le \EE_x[| T_\sigma f^*(x) - P^* (x)\big|] + \frac{\eps}{4} \\
  &\le \EE_{x }[| T_\sigma f^*(x) - P^* (x)\big|^2]^{1/2} + \frac{\eps}{4} \le \eps/2. \label{eq:poly-regression-insample-err-2-1}  
\end{align}
For the second term in \cref{eq:poly-regression-insample-err-2}, since $|T_\sigma f^*(x_i)|\le \sup_x |f^*(x)| =1$ is bounded, 
Chebyshev's inequality again implies that with probability at least $11/12$ over the randomness of the training data $\cD$, 
\begin{align}
\frac{1}{n} \sum_{i=1}^n |T_\sigma f^*(x_i) - y_i| \le \EE_{(x,y)\sim \mu}[|T_\sigma f^*(x) - y|] + \frac{\eps}{4}. \label{eq:poly-regression-insample-err-4}
\end{align}
Combining \cref{eq:poly-regression-insample-err-2,eq:poly-regression-insample-err-3,eq:poly-regression-insample-err-2-1,eq:poly-regression-insample-err-4} and \cref{eq:target-f-approx}
, we have that with probability at least $3/4$ over the randomness of the training data $\cD$, it holds that
\begin{align}
\PP_{(x,y)\sim \mu}(\hath_\cD(x) \neq y) &\le \frac{1}{2} \EE_{(x,y)\sim \mu}[|T_\sigma f^*(x) - y|] + \frac{3\eps}{4} \\  
&\le \mathrm{opt}_{\sigma,\cH(k)} + \eps,\label{eq:single-shot-poly-regression}
\end{align}
In each fit, we take samples of size $n = \Theta\big(\eps^{-2} m \log (m)\big)$.

{
  \paragraph{Best-of-$R$ guarantee.}

Suppose that we run \cref{alg:poly-regression} for $R =\Theta\big(\log(1/\delta)\big)$ times, each with independent samples $\cD_r$ of size $n= \Theta \big(\eps^{-2} m \log (m)\big)$, and get a family of predictors $\{\hath_{\cD_r}\}$. 
Then \cref{eq:single-shot-poly-regression} implies that 
\begin{align}
\PP\Big(\big(\min_r \PP_{(x,y)\sim \mu}(\hath_{\cD_r}(x) \neq y) \big) > \mathrm{opt}_{\sigma,\cF(k)} + 3\eps /4 \Big) &\le  4^{-R} \le \delta/2. 
\end{align} 
We denote $\cE_0 = \{\min_r \PP_{(x,y)\sim \mu}(\hath_{\cD_r}(x) \neq y) \le \mathrm{opt}_{\sigma,\cF(k)} + 3\eps /4 \}$ and set $r^* = \argmin_r \PP_{(x,y)\sim \mu}(\hath_{\cD_r}(x) \neq y)$.  

We take another validation set $\cV$ of size $n_{v}  = \Omega \big(\eps^{-2}\log(1/\delta)\big)$.   
Using a union bound and Hoeffding's inequality, we have that
\begin{align}
  \PP \Big(\max_r\big| \PP_{(x,y)\sim \mu}(\hath_{\cD_r}(x) \neq y) - \frac{1}{n_v}\sum_{(x_i, y_i)\in \cV} \ind\{\hath_{\cD_r}(x_i) \neq y_i\}\big| > \eps/8\Big) &\le 2R \cdot \exp\Big(-\frac{n_v \eps^2}{32}\Big) \le \delta/2.
\end{align}
Accordingly, we set
\begin{align}
  \cE_1 = \Big\{ \max_r\big| \PP_{(x,y)\sim \mu}(\hath_{\cD_r}(x) \neq y) - \frac{1}{n_v}\sum_{(x_i, y_i)\in \cV} \ind\{\hath_{\cD_r}(x_i) \neq y_i\}\big| \le \eps/8\Big\}. 
\end{align}
Then on the event $\cE_0 \cap \cE_1$, where $\PP(\cE_0 \cap \cE_1) \ge 1-\delta$, we have that  
\begin{align}
\PP_{(x,y)\sim \mu}(\hath_{\cD_\hatr}(x) \neq y) &\stackrel{\cE_1}{\le} \frac{1}{n_v}\sum_{(x_i, y_i)\in \cV} \ind\{\hath_{\cD_\hatr}(x_i) \neq y_i\} + \eps/8 \\
&\le \frac{1}{n_v}\sum_{(x_i, y_i)\in \cV} \ind\{\hath_{\cD_{r^*}}(x_i) \neq y_i\}  + \eps/8 \\
&\stackrel{\cE_1}{\le} \PP_{(x,y)\sim \mu}(\hath_{\cD_{r^*}}(x) \neq y) + \eps/4 \\
&\stackrel{\cE_0}{\le} \mathrm{opt}_{\sigma,\cF(k)} + \eps. 
\end{align}
Here the second line holds by the optimality of $\hatr$ on the validation set. 
Our final predictor is $\hath_{\cD_\hatr}$, and the toal sample complexity is $Rn + n_v = O\big(\eps^{-2} m \log (m)\log(1/\delta)\big)$.
This concludes the proof of \cref{lem:poly-regression-learns}.
}

\begin{proof}[Proof of \cref{thm:smoothed-learning}]\label{proof:thm-smoothed-learning} 
This is a direct consequence of combining \cref{lem:poly-regression-learns} and \cref{prop:poly-approx}. 
\end{proof}

\end{proof}


\section{Deferred proofs from \cref{app:learning-smooth-function}}

\begin{proof}[Proof of \cref{lem:bernoulli-smoothing}]\label{proof:lem:bernoulli-smoothing}
We inductively define $B_l:\RR \to \RR$ for $l=0,\ldots,k$ as 
\begin{align}
    B_0 &= 1;\quad B_1(x) = x-1/2; \quad B_{l+1}(x) = (l+1) \Big(\int_0^x B_l(t) \, dt - \int_0^1 B_l(t) \, dt\Big), \, l\ge 1.  
\end{align}
It is clear that, $B_l$ is a polynomial of degree $l$ such that $B_l(1) = B_l(0)$ for $l\neq 1$, and $B'_{l+1}(x) = (l+1) \cdot B_{l}(x)$. 
Note that $|B_1(x)| = |x-1/2| \le 1+|x|$.   
Given that $|B_l(x)|\le 2^l (1+|x|)^l$ for some $l\ge 0$, we have that 
\begin{align}
    |B_{l+1}(x)|& \le (l+1) \int_0^{|x|} |B_l(t)| \, dt + (l+1) \int_0^1 |B_l(t)| \, dt \le (1+|x|)^{l+1} \\ 
    &\le 2^{l+1} \int_0^{|x|\vee 1} (l+1) (1+t)^l \, dt  \\ 
    &\le 2^{l+1} (1+|x|)^{l+1}.   
\end{align}
Therefore, we have that $|B_l(x)|\le 2^l (1+|x|)^l$ for all $l\ge 0$ by induction.  

We now define $p_k(x) = \sum_{l=0}^k a_l \cdot B_{l+1}(x) / (l+1)!$.  
Then for any $0\le m\le k$, we have that 
\begin{align}
    p_k^{(l)}(1) - p_k^{(l)}(0) &= \sum_{m=0}^k \frac{a_m}{(m+1)!} \cdot \big(B_{m+1}^{(l)}(1) - B_{m+1}^{(l)}(0)\big) \\ 
    &= \frac{a_l}{(l+1)!} \cdot \big(B_{l+1}^{(l)}(1) - B_{l+1}^{(l)}(0)\big) \\ 
    &=a_l \big( B_1(1) - B_1(0) \big)= a_l. 
\end{align}
On the other hand, we have that 
\begin{align}
    p_k^{(k)}(x) &= \sum_{l=0}^k \frac{a_l}{(l+1)!} \cdot B_{l+1}^{(k)}(x) = \frac{a_k}{(k+1)!} \cdot B_{k+1}^{(k)}(x) =  a_k\,(x-1/2).
\end{align}
Therefore, the result for the $k$-th derivative holds by taking the maximum of the absolute value of $p_k^{(k)}(x) $ over $x\in [0,1]$.  
Using the absolute bound of $B_l$, we have that 
\begin{align}
    |p_k(x)| &\le \sum_{l=0}^k \frac{|a_l|}{(l+1)!} \cdot |B_{l+1}(x)| \le \sum_{l=0}^k \frac{|a_l|}{(l+1)!} \cdot 2^{l+1} (1+|x|)^{l+1} \\ 
    &\le (1+|x|)^k \sum_{l=0}^k \frac{2^{l+1}|a_l|}{(l+1)!}. 
\end{align}
This concludes the absolute bound of $p_k$.  
\end{proof}

\begin{proof}[Proof of \cref{thm:jackson}]\label{proof:jackson-coefficients}
Following the proof of Theorem 2.3 and Corollary 2.4 in \cite{devore1993constructive}, the approximating trigonometric polynomial is given by 
\begin{align}
    T_D(x) &= \sum_{ 1\le j \le k}\int_0^{2\pi} (-1)^{j+1}  f(x+jt) \cdot  K_{D,k}(t) \, dt,  \label{eq:jackson-kernel}\\ 
    K_{D,k}(t) &= \lambda_{D,k} \cdot \bigg(\frac{\sin\big( (\lfloor D /2\rfloor +1)t/2\big)}{ \sin(t/2)}\bigg)^{2k}, 
\end{align}  
where $\lambda_{D,k}$ is a normalizing constant such that $\int_0^{2\pi} K_{D,k}(t) \, dt =1$. 
We denote $c_m(f) =  \int_0^{2\pi} f(t) e^{i m t} \, dt$ as the $m$-th Fourier coefficient of $f$.   
Using the property of periodic convolution, we have that 
\begin{align}
    c_m(T_D) &= \int_0^{2\pi} T_D(t) \cdot e^{-i m t} \, dt \\ 
    &=  \sum_{1\le j \le k} (-1)^{j+1} \cdot \int_0^{2\pi} \int_0^{2\pi} f(t + j s) \cdot K_{D,k}(s) \, ds \cdot e^{-i m t} \, dt \\ 
    &= \sum_{1\le j \le k} (-1)^{j+1} \cdot \int_0^{2\pi} K_{D,k}(s) e^{imjs }\cdot \int_0^{2\pi} f(t + j s)  \cdot e^{-i m( t+js)} \, dt \, ds \\
    &= \sum_{1\le j \le k} (-1)^{j+1} \cdot c_{-mj}(K_{D,k}) \cdot c_m(f).  \label{eq:fourier-coeff-jackson}
\end{align}
We bound the Fourier coefficients of $K_{D,k}$ and $f$ separately. 
For $f$, we can use property of Fourier transform to obtain that $c_m(f) = (im)^{-k} c_m(f^{(k)})$, thus obtaining that $|c_m(f)| \le 2\pi M / |m|^k$.  
For $T_D$, we can use the normalizing property of $K_{D,k}$ to get that 
\begin{align}
    \Big| \sum_{1\le j \le k} (-1)^{j+1} \cdot c_{-mj}(K_{D,k}) \Big| &= \int_0^{2\pi} K_{D,k}(t) \cdot \Big| \sum_{1\le j \le k} (-1)^{j+1} e^{i m j t} \Big| \, dt \\ 
    &\le  \int_0^{2\pi} K_{D,k}(t) \cdot 2^{k+1} \, dt = 2^{k+1}. 
\end{align}
This implies that $c_m (T_D) \le 2^{k+2}\pi M / |m|^k$.  
Now the relation 
\begin{align}
    a_m = \frac{c_m + c_{-m}}{2\pi}, \quad b_m = \frac{c_m - c_{-m}}{2\pi i}
\end{align}
yields that $\sum_{1\le m \le D} |a_m| + |b_m| \le \sum_{1\le m \le D} |c_m| / \pi \le 2^{k+2}M/\pi \sum_{1\le m \le D} m^{-k}$.
\end{proof}

\begin{proof}[Proof of \cref{lem:nn-trig-poly-approx}]\label{proof:nn-trig-poly-approx}
Note that the Fourier transform of $T_D$ is given by
\begin{align}
    \widehat{T_{D_0}}(\xi) &= \sum_{m\le D_0 }  \Big(\frac{(a_m - b_m i)(-1)^m}{2}  \cdot \delta_{m\pi /R}(\xi) + \frac{(a_m + b_m i)(-1)^m}{2}  \cdot \delta_{-m\pi /R}(\xi)\Big). 
\end{align}
Note that it is a finite linear combination of Dirac delta functions, and is equivalent to a complex Radon measure $\rho_0$ on $\RR$.  
Moreover, it is supported on $[-D_0\pi/R, D_0\pi/R]$. 
And the total variation of the corresponding complex Radon measure is given by
\begin{align}
    |\rho_0| &\le \sum_{m\le D_0 } \frac{|a_m -b_m i|}{2} + \frac{|a_m + b_m i|}{2} = \sum_{m\le D_0 } (|a_m| + |b_m|) \coloneqq C_D. 
\end{align} 
Then, applying \cref{lem:fourier-rep-fs}, we have that for any $\Omega\ge D_0 \pi/R$ and $D >0$ that 
\begin{align}
    \norm{r_D}_{\mu}^2    &\le \sup_{\|\xi\| \le \Omega} \, |\varphi(\xi)|\cdot  |\rho_0|.  
\end{align}
\paragraph{Strictly sub-exponential bound.}For the strictly sub-exponential case, when $D = er (K\Omega)^r (1 + K^{-r} \Omega^{-r})$, we have that 
 \begin{align}
    \norm{r_D}_{\mu} &\lesssim  C_{D_0} \cdot e^{-e\Omega}.  
 \end{align}
Now, we choose $\Omega \simeq \log(C_{D_0}/\eps) \vee  (D_0/R)$, which yields that  $D =   er (K\Omega)^r (1 + K^{-r} \Omega^{-r})\simeq (( \log(C_{D_0}/ \eps))\vee {(D_0/R)})^{r}$.
This ensures that $\norm{r_D}_\mu \le \eps$.  

\paragraph{Sub-exponential bound.} For sub-exponential case, it holds by \cref{lem:complex-Rd} that 
 \begin{align}
    \norm{r_D}_{\mu} &\lesssim  C_{D_0} \cdot \exp\{-4D \cdot e^{-K\pi \Omega/4 }\}. 
 \end{align} 
 Therefore choosing $D = e^{K\pi \Omega/4} \cdot \log(C_{D_0}/\eps)\simeq \exp \{O(D_0  /  R)\} \cdot \log(C_{D_0}/\eps)$  yields that $\norm{r_D}_\mu \le \eps$. 
\end{proof}

\begin{proof}[Proof of \cref{thm:poly-approx-lip-network}]\label{proof:thm:poly-approx-lip-network}
without loss of generality, it suffices to find a polynomial that approximates $\tsigma$ that is defined in \cref{eq:taylor-jet}. 
We fix $R$ that is to be specified in the end.
The construction of $\tsigma$ leads to the following derivative bound for $0\le l \le k$: 
\begin{align}
    \big|\tsigma^{(l)}(x)\big| &\le \int_0^x \frac{|\tsigma^{(k)}(t)|}{(k-l-1)!} |x-t|^{k-l-1} dt \\ 
    &\le \int_0^x \frac{M}{(k-l-1)!} |x-t|^{k-l-1} dt = \frac{M |x|^{k-l}}{(k-l)!}. \label{eq:tsigma-derivative-bound}
\end{align}
Given $R$ and $\tsigma$, we construct a degree $k$ polynomial $q_k$ in \cref{eq:bern-qk}. 
Then \cref{lem:bernoulli-smoothing} indicates the following growth bound:
\begin{align}
    |q_k(x)| &\le   \frac{(3R+|x|)^k}{(2R)^k}\cdot \sum_{0\le l\le k} \frac{2^{ l+1} (2R)^l \big( \tsigma^{(l)}(R) - \tsigma^{(l)}(-R)\big)}{(l+1)!}\\ 
    &\le \frac{(3R+|x|)^k}{(2R)^k} \cdot \sum_{0\le l\le k} \frac{4M\cdot (4R)^l \cdot R^{k-l}}{(l+1)!\cdot (k-l+1)!} \\ 
    &\le  \frac{4^{k+1}M}{k!} (3R+|x|)^k \cdot . \label{eq:qk-growth-bound}  
\end{align}
Combining, we get that
\begin{align}
    \big|\tsigma(x) - q_k(x)\big| &\le \big|\tsigma(x)\big| + |q_k(x)| \le \frac{4^{k+2}M}{k!} (3R+|x|)^k. \label{eq:tsigma-qk-bound} 
\end{align}
We use this bound to control the first term in \cref{eq:three-terms}. 
Suppose that $\PP(|x|\ge t) \lesssim \exp\{-(t/b)^\alpha\}$ for some $\alpha\ge1$, then we have that  
\begin{align}
    \EE[(3R+|x|)^{2k} \ind\{|x|\ge R\}] &\le (4R)^{2k} \cdot \PP(|x|\ge R) + \int_{R}^\infty \PP(|x|\ge t) \cdot 2k(3R+t)^{2k-1}\,dt \\
        &\le Ae^{-(R/b)^\alpha} \cdot (4R)^{2k} + 2k\,A \,  \int_{R}^\infty e^{-(t/b)^\alpha} (3R+t)^{2k-1} dt \\ 
        &\le Ae^{-(R/b)^\alpha} \cdot (4R)^{2k} + 2^{2k-1}k \, A\,  e^{-(R/b)^\alpha}\int_{R}^\infty e^{-(t/b)^\alpha} \big((3R)^{2k-1} + t^{2k-1}\big) dt \\
        &\lesssim Ae^{-(R/b)^\alpha} \cdot (4R)^{2k} + 2^{2k}k\,A\,e^{-(R/b)^\alpha} \left( (3R)^{2k-1} b+ b^{2k} (2k-1)! \right) \\
        &\le  AM\cdot e^{-(R/b)^\alpha} \cdot \big((6R)^{2k}+ (2b)^{2k}(2k-1)!\big) . \label{eq:second-term-final}
\end{align}
Taking square-root yields that
\begin{align}
     \norm{(\tsigma(x) - q_k(x))\cdot \ind\{|x|\ge R\}}_{\mu} &\lesssim  \sqrt{A}M e^{-(R/b)^\alpha/2} \cdot \big((6R)^{k}+ (2b)^{k}\sqrt{(2k-1)!}\big)\frac{4^{k+2}}{k!}. 
\end{align}

On the other hand, we can derive similar bound for the last term in \cref{eq:three-terms}. 
We note from the construction in \cref{eq:jackson-kernel} that
\begin{align}
\sup_{x\in \RR} |T_{D_0}(x)| &\le k\sup_{x\in [-R,R]}| \tsigma(x) - q_k(x)| \\ 
&\le k \frac{ 4^{2k+2} R^k M}{k!}. 
\end{align}
where the second inequality follows from \cref{eq:tsigma-qk-bound}. 
As a consequence, we have that
\begin{align}
    \norm{ T_{D_0} \cdot \ind\{|x|\ge R\}}_\mu &\le \sup_{x\in \RR} |T_{D_0}(x)| \cdot \PP(|x|\ge R)^{1/2} \\ 
    &\le \frac{4^{2k+2}R^k  M}{(k-1)!} \cdot \sqrt{A}e^{-(R/b)^\alpha/2}. 
\end{align}
Combining with \cref{eq:second-term-final}, we have 
\begin{align}
    \norm{\tsigma(x) - q_k(x)}_{\mu}+\norm{ T_{D_0} \cdot \ind\{|x|\ge R\}}_\mu &\lesssim M e^{-(R/b)^\alpha/2} \cdot \big((24R)^{k}+ (8b)^{k}\sqrt{(2k-1)!}\big)/(k-1)!.\label{eq:combined-23}
\end{align}
We will specify the choice $R$ with the specific assumption on the distribution. 

Clearly, the adjusted function $\check\sigma = \tsigma - q_k$ has matching derivatives up to order $k$ at the boundaries, and $|\check\sigma^{(k)}| \le 2M$ as shown in \cref{eq:qk-kprime-bound}. 
As long as $D_0 = \Omega(M^{1/k} R  \eps^{-1/k})$, \cref{cor:jackson-rescale} implies that there exists a trigonometric polynomial $T_{D_0}$ such that $\sup_{x\in[-R,R]} \big| \check\sigma(x) - T_{D_0}(x) \big| \le \eps/4$. 
In addition, the coefficients of the approximating trigonometric polynomial satisfy that
\begin{align}
    \sum_{1\le m \le D_0} |a_m| + |b_{m}| \le\frac{ 2^{k+2} }{\pi} \sum_{1\le m\le D_0}  m^{-k} \le \frac{2^{k+2}}{\pi} \cdot  \log(D_0)^{\ind\{k=1\}} 
\end{align} 
\cref{lem:nn-trig-poly-approx} indicates that there exists an algebraic polynomial $p_T$ that approximates of $T_{D_0}$, in the sense that $\norm{T_{D_0} - p_T}_{\mu} \le \eps/4$.  

\paragraph{Bounded case, $r=1$.} 
In this case, we can directly choose $R=K$ that is the support of the distribution, and all terms except the first \cref{eq:three-terms} vanish. 
We set the  degree of the trigonometric polynomial as $\Theta(M^{1/k} K \eps^{-1/k})$. 
To settle the degree of the corresponding algebraic polynomial, we note that $\log(C_{D_0}/\eps ) \lesssim k\log(\log(D_0)/\eps)$. 
While $D_0 /R$ is always $\mathrm{poly}(1/\eps)$. 
Therefore, \cref{lem:nn-trig-poly-approx} implies that the corresponding algebraic polynomial is of degree $O\Big( M^{1/k} \cdot \eps^{-1/k}\Big)$. 

\paragraph{Strictly sub-exponential case, $r>1$.} 
In this case, we have that $b=K$ and $\alpha = r/ (r-1)$. 
From \cref{eq:combined-23}, choosing 
\begin{align}
    R\simeq K\cdot \Big(\log(M/\eps)+\frac{(r-1)k}{r}\log k   \Big)^{(r-1)/r} 
\end{align}
is sufficient to ensure that \cref{eq:combined-23} is bounded by $\eps/2$.  
Similarly, we note that $D_0/R$ dominates $\log(C_{D_0}/\eps)$, and thus the degree of the corresponding algebraic polynomial is 
\begin{align}
    D\simeq (D_0/R)^r \simeq  \Big(\frac{M}{\eps}\Big)^{r/k} \cdot \Big(\log(M/\eps)+\frac{(r-1)k}{r}\log k   \Big)^{(r-1)}. 
\end{align}

\paragraph{Sub-exponential case. } For sub-exponential $\mu$, we can choose $A=2, b=K$ and $\alpha = 1$.  
The corresponding truncation level is $R = \Theta\Big(\log(M/\eps) + k \log k\Big)$. 
As a consequence, \cref{lem:nn-trig-poly-approx} implies that the corresponding algebraic polynomial is of degree 
\begin{align}
    D = O\Big( \exp\big(O({D_0}/{R})\big) \cdot \log(\log(D_0) / \eps )\Big) = O\Big( \exp\big\{O\big((M/\eps )^{1/k }\big)\big\}\Big).
\end{align}
This concludes the proof of the desired results. 
\end{proof}

\subsection{Comparison with \texorpdfstring{\cite{chandrasekaran2025learning}}{Chandrasekaran et al., 2025}}\label{apdx:comparison_futher}

Suppose that $f:\RR\to\RR$ is a function that can be $(\eps,R)$-uniformly approximated by  a polynomial of degree $D(\eps)\cdot R\log R$. 
Let $\mu$ be a $r$-strictly sub-exponential distribution as defined in \cref{def:strictly-subexp} with $r>1$. 
Then the proof of Theorem 4.6 in \cite{chandrasekaran2025learning} implies that it is sufficient for the truncation level $R$ to follow that 
\begin{align}
    R^{r/(r-1)}  \ge C(D(\eps) \log(D(\eps))\cdot R\log (R)^2  +\log (1/\eps))
\end{align}
for some constant $C>0$. 
A minimal sufficient condition for this to hold is that 
\begin{align}
    R \gtrsim  (D(\eps) \log D(\eps) )^{ 2(r-1)} + \log(1/\eps)^{(r-1)/r}.
\end{align}
As a consequence, the final degree of the approximation polynomial is $O(D(\eps)R\log R) = \tO(D(\eps)^{ 2r-1})$, where $\tO$ hides a $\mathrm{polylog}(D(\eps))$ factor. 
For Lipschitz function $f$, Jackson's theorem implies that $D(\eps)= \eps^{-1}$, and the 
approximating polynomial is of degree $\tO(\eps^{-(2r-1)})$.

\section{Deferred proofs from~\cref{sec:app-learning-nn}}\label{apdx:proof-nn}

\subsection{Proofs of polynomial regression results}

\begin{proof}[Proof of \cref{lem:gen-error-polynomial-regression}] \label{proof:gen-error-polynomial-regression}
Throughout the proof, we assume that the degree for the polynomial regression is fixed. 
For any hypothesis $h:\RR^d \to \RR$, we denote $R(h) = \EE[\ell(y, h(x))]$ as the test risk and $\hat R_n(h) = n^{-1}\sum_{i\le n} \ell(y_i, h(x_i))$ as the empirical risk. 
For any coefficient vector $w\in \RR^m$, we use $p_w (x) = \phi_D(x)^\top w$ to denote the corresponding polynomial predictor and $p_w^M(x) = \cT_M(\phi_D(x)^\top w)$ to denote its truncated version.  
We claim that truncation does not increase the training error, given that the level $M$ is chosen properly.  
Indeed, for any $w\in \RR^m$, the training risk of the truncated polynomial can be bounded as
\begin{align}
    \hat R_n(p_w^M) 
    &= \frac{1}{n} \sum_{i\le n} \ell\big(y_i, \cT_M(\phi_D(x_i)^\top w)\big) \ind\{|y_i| \le M\} \\ &\qquad + \frac{1}{n} \sum_{i\le n} \ell\big(y_i, \cT_M(\phi_D(x_i)^\top w)\big) \ind\{|y_i| > M\} \\  
    &\le    \frac{1}{n} \sum_{i\le n} \ell\big(y_i, \phi_D(x_i)^\top w\big) \ind\{|y_i| \le M\} \\& \qquad +\frac{\mathfrak{L}}{n} \sum_{i\le n} \max\{|y_i|,M \} \cdot \big|\cT_M (\phi_D(x_i)^\top w) - y_i\big| \ind\{|y_i| > M\} \\ 
    &\le \frac{1}{n} \sum_{i\le n} \ell\big(y_i, \phi_D(x_i)^\top w\big) + \frac{2\mathfrak{L}}{n} \sum_{i\le n} |y_i|^2  \ind\{|y_i| > M\}. \label{eq:truncation-training-error-l1} 
\end{align}
Here the first inequality holds because of the non-decreasing property of $\ell$ in $|y-y'|$ and the pseudo-Lipschitzness of $\ell$. 
Then the optimality of $\hat w$ combined with \cref{eq:truncation-training-error-l1} implies that for any $w\in \RR^m$, it holds that 
\begin{align}
    \hat R_n(p_{\hat w}^M) &\le \hat R_n(p_w) + \frac{2\mathfrak{L} }{n} \sum_{i\le n} |y_i|^2 \ind\{|y_i| > M\} . \label{eq:empirical-risk-bound}  
\end{align}  
We analyze the two terms on the right-hand side of \cref{eq:empirical-risk-bound} separately, starting with the second term. 
With the Markov inequality, it holds with probability at least $15/16$ that 
\begin{align}
    \frac{1}{n} \sum_{i\le n} |y_i|^2 \ind\{|y_i|>M\} &\le  \EE[|y|^2 \ind\{|y|>M\}] +\sqrt{\frac{16\EE[|y|^4 \ind \{|y|\ge M\}]}{n}} \\ 
    &\le  \PP(|y|\ge M)^{1/2} \cdot \EE[|y|^4]^{1/2} + \sqrt{\frac{\EE[16|y|^4]}{n}} \\
    &\le  \frac{\EE[|y|^4]}{M^2} + \sqrt{\frac{16\EE[|y|^4]}{n}}. \label{eq:tail-bound-y}   
\end{align}
For the first term in \cref{eq:empirical-risk-bound}, we substitute $w$ with some oracle choice that approximates the optimal predictor in $\cF$. 
Concretely, we consider $f^* = \argmin_{f\in \cF} \EE[\ell(y, f(x))]$ and $w_f^* = \argmin_{w\in \RR^m} \EE[(f^*(x) - \phi_D(x)^\top w)^2]$. 
By the assumption, we know that $\EE[(f^*(x) - \phi_D(x)^\top w_f^*)^2]^{1/2} \le \eps$.  
And we can decompose the loss in the first term as
\begin{align}
\ell\big(y_i, \phi_D(x_i)^\top w_f^*\big) &= l\big(|y_i - f^*(x_i) + f^*(x_i) - \phi_D(x_i)^\top w_f^*|\big) \\  
&\le l\big( \big|y_i - f^*(x_i)\big| + \big|f^*(x_i) - \phi_D(x_i)^\top w_f^*\big|\big) \\ 
&\le  l\big(|y_i - f^*(x_i)|\big) + 2\mathfrak{L} \max\{|y_i|, |f^*(x_i)|,1\} \cdot \big|f^*(x_i) - \phi_D(x_i)^\top w_f^*\big|.  \label{eq:decompose-loss-oracle}
\end{align}
Here the first inequality follows from the triangle inequality and the non-decreasing property of $l$, and the second inequality follows from the pseudo-Lipschitzness of $l$.  
On the other hand, we can use the Cauchy-Schwarz inequality to bound the expectation of the last term as 
\begin{align}
     \EE[\max\{|y|, |f^*(x)|,1\} \cdot |f^*(x) - \phi_D(x)^\top w_f^*|] &\le \EE[\max\{|y|, |f^*(x)|,1\}^2]^{1/2} \cdot \EE[|f^*(x) - \phi_D(x)^\top w_f^*|^2]^{1/2} \\ 
     &\le \eps \cdot \Big(\EE[y^4]^{1/4} + \sup_{f\in \cF} \EE[f(x)^4]^{1/4}+ 1\Big). 
\end{align}
Then, the Markov's inequality implies that with probability at least $15/16$, it holds that 
\begin{align}
      \frac{1}{n}\sum_{i\le n} \max\{|y_i|, |f^*(x_i)|,1 \} \cdot  |f^*(x_i ) - \phi_D(x_i)^\top w_f^*| &\le 16\, \eps \cdot \Big(\EE[y^4]^{1/4} + \sup_{f\in \cF} \EE[f(x)^4]^{1/4}+ 1\Big).  \label{eq:empirical-approximation-error} 
\end{align}
On the other hand, we can use the Chebyshev's inequality to bound the first term in \cref{eq:decompose-loss-oracle}, which gives that with probability at least $15/16$, it holds that 
\begin{align}
        \frac{1}{n} \sum_{i\le n} \ell\big(y_i, f^*(x_i)\big) &\le \mathrm{opt}_\cF + \frac{(16\, \EE[\ell (y, f^*(x))^2])^{1/2}}{\sqrt{n}}. \label{eq:empirical-optimal-risk}  
\end{align}
Combining \cref{eq:decompose-loss-oracle,eq:empirical-approximation-error,eq:empirical-optimal-risk}, we have with probability at least $7/8$ that 
\begin{align}
\hat R_n(p_{w_f^*})  &\le \mathrm{opt}_\cF + O\Big(\eps + \frac{\EE[16\ell (y, f^*(x))^2]^{1/2}}{\sqrt{n}}\Big).  \label{eq:oracle-risk-bound}
\end{align}
This concludes the analysis of the first term in \cref{eq:empirical-risk-bound}.  
Plugging \cref{eq:oracle-risk-bound} and \cref{eq:tail-bound-y} into \cref{eq:empirical-risk-bound}, we have with probability at least $13/16$ that, 
\begin{align}
\hat R_n (p_{\hat w}^M) &\le \mathrm{opt}_\cF + O\Big(\eps + \frac{\EE[\ell (y, f^*(x))^2]^{1/2}}{\sqrt{n}} + \frac{\EE[|y|^4]}{M^2} + \sqrt{\frac{\EE[|y|^4]}{n}}\Big).   
\end{align}
Now we have derived an upper bound on the empirical risk of the truncated polynomial predictor, it remains to bound its test risk using \cref{lem:approx-lemma-vapnik}.  
Note that the classifier $(y,x) \mapsto \ind\{\ell(y, \cT_M(\phi_D(x)^\top \hat w)) \ge t\}$ has VC-dimension at most $m+1$ for any $t>0$.
Additionally, the second moment of the loss function can be bounded as 
\begin{align}
    \EE[\ell(y, \cT_M(\phi_D(x)^\top w))^2]^{1/2} &\lesssim \EE[ \big(l(0) + |y| + M\big)^2]^{1/2}. 
\end{align}
Indeed, given that $n\ge m \log(2en/m) + \log(192)$, \cref{lem:approx-lemma-vapnik} implies that with probability at least $15/16$ over the dataset $\{(x_i,y_i)\}_{i\le n}$, it holds that
\begin{align}
    R(p_{\hat w}^M) &\le \hat R_n(p_{\hat w}^M) + O\Big(\EE[ \big(l(0) + |y| + M\big)^2]^{1/2} \cdot \sqrt{\frac{m \log (n /m) }{n}}\Big).
\end{align} 
In conclusion, with probability at least $3/4$ over the dataset, we have that
\begin{align}
    R(p_{\hat w}^M) &\le \mathrm{opt}_\cF + O\Big(\eps + \frac{\EE[\ell (y, f^*(x))^2]^{1/2}}{\sqrt{n}} + \frac{\EE[|y|^4]}{M^2} + \sqrt{\frac{\EE[|y|^4]}{n}} \\ 
    &\qquad + \EE[ \big(l(0) + |y| + M\big)^2]^{1/2} \cdot \sqrt{\frac{m \log (n /m) }{n}}\Big).
\end{align}
We choose $M =\Omega( \eps^{-1/2})$ and $n= \Omega(\eps^{-3 }m)$
Then it holds that with probability at least $3/4$ over the dataset, that. 
\begin{align}
    R(p_{\hat w}^M) &\le \mathrm{opt}_\cF + \eps.  \label{eq:final-risk-bound-single-round}
\end{align}

\paragraph{Replication and boosting.} Suppose that we replicate the procedure above with $R = \Theta(\log(8/\delta))$ times on independent datasets, and denote the output truncated polynomial predictors as $\{p_{\hat w_r}^M\}_{r\le R}$. 
Then \cref{eq:final-risk-bound-single-round} implies that  
\begin{align}
    \PP\Big( \min_r R(p_{\hat w_r}^M) > \mathrm{opt}_\cF + \eps \Big) &\le (1/4)^R \le \delta/2.    
\end{align}
Denote the event in above as $\mathcal{E}_1$.  
We take another validation set of size $n_{\text{val}}= \Theta(\log(8 /\delta)\, \eps^{-( k-1)/(2k)})$ to choose the best hypothesis among $\{p_{\hat w_r}^M\}_{r\le R}$. 
To validate each predictor $p_{\hat w_r}^M$, we use the truncated loss $\bar\ell(y,y') = \cT_{\mathfrak{M}} \ell(y,y')$, where $ \mathfrak{M} = \Theta\Big(\Big(\log(1/\delta)\, n^{-1}\Big)^{-(k-1)/(2k)}\Big)$. 
Conditioned on the training dataset, we can use \cref{lem:robust-concentration} on the random variables $Z_i^r = \bar\ell(y_i^{\text{val}}, p_{\hat w_r}^M(x_i^{\text{val}}))$ for $i\le n_{\text{val}}$ and an union bound over $r\le R$ to obtain that with probability at least $1-\delta/2$ over the validation dataset, it holds for all $r\le R$ that 
\begin{align}
     \Big|\frac{1}{n_{\text{val}}} \sum_{i\le n_{\text{val}}}  \cT_\mathfrak{M } \ell\big(y_i^{\text{val}}, p_{\hat w_{r }}^M(x_i^{\text{val}})\big)  - R(p_{\hat w_r}^M)  \Big| &\le \eps/8.  \label{eq:event-validation-error}
\end{align}
We set the event in \cref{eq:event-validation-error} as $\mathcal{E}_2$. 
We denote $\hat r = \argmin_{r\le R} \frac{1}{n_{\text{val}}} \sum_{i\le n_{\text{val}}}  \cT_\mathfrak{M}\ell\big(y_i^{\text{val}}, p_{\hat w_{r }}^M(x_i^{\text{val}})\big) $ as the index of the best predictor on the validation set. 
Then the joint event $\mathcal{E}_1 \cap \mathcal{E}_2$ follows that 
\begin{align}
     \PP(\mathcal{E}_1 \cap \mathcal{E}_2) = \PP(\PP(\mathcal{E}_2 \mid \mathcal{E}_1)\, \PP(\mathcal{E}_1)  &\ge (1 - \delta/2)^2 \ge 1 - \delta.  
\end{align} 
On the event $\mathcal{E}_1 \cap \mathcal{E}_2$, we have  
\begin{align}
    R(p_{\hat w_{\hat r}}^M) &\le \frac{1}{n_{\text{val}}} \sum_{i\le n_{\text{val}}}  \cT_\mathfrak{M} \ell\big(y_i^{\text{val}}, p_{\hat w_{\hat r}}^M(x_i^{\text{val}})\big) + \eps/8 \\ 
    &\le \frac{1}{n_{\text{val}}} \sum_{i\le n_{\text{val}}}  \cT_\mathfrak{M} \ell\big(y_i^{\text{val}}, p_{\hat w_{ r^*}}^M(x_i^{\text{val}})\big) + \eps/8 \\ 
    &\le R(p_{\hat w_{ r^*}}^M) + \eps/4 \\ 
    &\le \mathrm{opt}_\cF + \eps. 
\end{align}
In total, the pipeline takes $O\big(\eps^{-3 } m\log(1/\delta)\big)$ samples to output a predictor with test risk at most $\mathrm{opt}_\cF + \eps$ with probability at least $1-\delta$. 
This concludes the proof of the lemma. 
\end{proof}

\subsection{Proofs of agnostically learning neural networks}

\begin{proof}[Proof of \cref{lem:sigmoid-fourier}] \label{proof:lem:sigmoid-fourier}
By the definition of Fourier transform of tempered distributions, we have for any test function $\psi \in \cS(\RR)$ that
\begin{align}
\dotp{\widehat{\sgm}}{\psi} &= \dotp{\sgm}{\hat\psi} \\
&=  \int_\RR \Big( \frac{\tanh(x/2)}{2} + \frac{1}{2}\Big)\cdot \hat\psi(x) \, dx \\
&= \int_\RR \frac{\tanh(x/2)}{2} \hat\psi(x)  \, dx + \frac{1}{2} \int_\RR \hat\psi(x) \, dx. 
\end{align}
Using the Fourier inversion formula for Schwartz function, we have that $\int_\RR \hat\psi(x) \, dx = 2\pi \psi(0)$.  
Now we focus on the first term. 
For simplicity we denote $g(x) = \tanh(x/2)/2$, then $g'(x) = 1/(4\cosh^2(x/2))$ and $\widehat{ g'}(\xi)  =  \pi \xi/\sinh(\pi \xi)$ is an even function. 
Here $g'\in L^1(\RR)$, and thus $\widehat{g'}$ is understood as the classical Fourier transform. 
Using the definition of tempered distribution, we first verify that $\dotp {\hat g}{\psi} = \dotp{\pv \frac{\widehat{g'}}{i\xi}}{\psi}$.
Indeed, we can reformulate the right-hand side with a symmetric bump function $\chi$ that is supported on $[-2,2]$:
\begin{align}
 \dotp{\pv \frac{\widehat{g'}}{i\xi}}{\psi} &= \lim_{\eps \to 0^+} \int_{|\xi|>\eps} \frac{\widehat{g'}(\xi)}{i\xi} \cdot \psi(\xi) \, d\xi \\
    &= \lim_{\eps \to 0^+} \int_{|\xi|>\eps} \widehat{g'}(\xi) \cdot \frac{\psi(\xi) - \psi(0)\chi(\xi)}{i\xi} \, d\xi \\ 
    &= \int_{\RR}  g'(x) \cdot \Big(\frac{\psi - \psi(0)\chi}{i\xi}\Big)^\wedge (x) \, dx \\
    &= -\int_\RR g(x) \cdot \Big(- i\xi \cdot \frac{\psi - \psi(0)\chi}{i\xi}\Big)^\wedge{}(x) \, dx. 
\end{align}
Here the fourth line follows from the definition of Fourier transform of tempered distributions and the fact that $(\psi - \psi(0)\chi)/\xi \in \cS(\RR)$.  
Note that $\chi$ is even, so is its Fourier transform. 
Therefore, last term is equal to $\dotp{ g}{\hat\psi} = \dotp{\hat g}{\psi}$. 
In other words, we have that
\begin{align}
    \int_\RR g(x) \cdot \hat\psi(x) \, dx &= \pv \int_\RR \frac{\widehat{g'}(\xi)}{i\xi} \cdot \psi(\xi) \, d\xi \\ 
    &= \pv \int_\RR \frac{\pi }{i\sinh(\pi \xi)} \cdot \psi(\xi) \, d\xi. 
\end{align}
This concludes the proof of the lemma. 
\end{proof}

\begin{proof}[Proof of \cref{lem:poly-approx-sigmoid-network}]\label{proof:lem:poly-approx-sigmoid-network}
Without loss of generality, we consider the one-dimensional problem of approximating $\sgm$. 
Once we get a polynomial approximation $p_D$ for the one-dimensional case, the multi-dimensional case follows by setting $\tp_D(x) = \frac{1}{k} \sum_{j=1}^k p_D(w_j^\top x)$. 
Then applying triangle inequality yields the desired result. 

We define $r_D = f - p_D$, where $p_D$ is the projection of $f$ onto the space of polynomials of degree $D$ in $L^2(\mu)$. 
We also denote $\varphi(\xi) = \cF[r_D\mu](\xi)$ as the Fourier transform of $r_D\mu$.
The proof consists of two parts. 
In the beginning, we provide a decomposition of $\hatf$ using \cref{lem:sigmoid-fourier}.
Then we use \cref{lem:fourier-rep-fs} to bound the $L^2(\mu)$-norm of the residual when approximating $f$ with polynomials. 
\paragraph{Decomposing Fourier transform.} 
For the activation $\sgm$, it holds for any $\psi \in\cS(\RR)$ that  
\begin{align}
    \dotp{\widehat{\sgm}}{\psi} &= \pi\psi(0) +\int_{\RR} \Big(\frac{\pi }{i\sinh(\xi)} -  \frac{\pi }{i\xi}\ind \{|\xi| \le 1\}\Big) \cdot \psi(\xi)\,d\xi  \\ &\qquad +  \pv \int_{[-1,1]} \frac{\pi }{i\xi } \cdot \psi(\xi) \, d\xi.  \label{eq:sigmoid-fourier-test}
\end{align}
Note that 
\begin{align}
    \pv \int_{[-1,1]} \frac{ \psi(\xi)}{ \xi } \, d\xi &=  \lim_{\eps \to 0^+} \int_\eps^1 \frac{\psi(\xi) - \psi(-\xi)}{\xi} \, d\xi \\ 
    &= \lim_{\eps \to 0^+} \Big(\int_\eps^1 \log \xi\, \big(\psi'(\xi) + \psi'(-\xi)\big) \, d\xi  - \frac{\psi(\eps) - \psi(-\eps)}{2\eps}\cdot 2\eps \log \eps \Big) \\ 
    &= \int_{[-1,1]} \log |\xi| \cdot \psi'(\xi) \, d\xi. 
\end{align}
Now we define $\varrho_0 = \pi\delta_0 + i^{-1}\pi  (\sinh^{-1}(\xi) - \xi^{-1}\ind\{|\xi|\le 1\}) d\xi$ and $\varrho_1 = i^{-1}\pi \log |\xi| \cdot \ind\{|\xi|\le 1\} d\xi$, then clearly they are both complex Radon measures on $\RR$ with finite total variations. 
And $\varrho_1$ is supported on $[-1,1]$. 
We can then rewrite  \cref{eq:sigmoid-fourier-test} as $\dotp{\widehat{\sgm}}{\psi} = \dotp{\varrho_0}{\psi} + \dotp{\varrho_1\partial }{\psi}$.  
This representation implies that $\sgm$ belongs to the class of functions characterized in \cref{lem:fourier-rep-fs}. 
Additionally, it holds for $\Omega >1$ that 
\begin{align}
    |\varrho_0|\big(\RR \setminus [-\Omega, \Omega ]\big) &\le  2 \int_\Omega^\infty e^{-t} \frac{2}{1-e^{-2\Omega}} \, dt\le 16 e^{-\Omega}.   \label{eq:rho-0-upper} 
\end{align} 
Applying \cref{lem:fourier-rep-fs}  yields that 
\begin{align}
        \norm{r_D}_\mu ^2 &=\frac{1}{2\pi}\int_{\RR}  \varphi(\xi) d\varrho_0(\xi) + \varphi'(\xi) d\varrho_1(\xi)  \\ 
        &\le \frac{1}{2\pi} \Big(\sup_{|\xi| \le \Omega} |\varphi(\xi)| \cdot |\varrho_0|  +   \sup_{\xi}|\varphi(\xi)|\,  |\varrho_0| (\RR\setminus [-\Omega, \Omega ])   + \sup_{|\xi| \le 1} |\varphi'(\xi)| \cdot |\varrho_1|  \Big)   \\
        &\stackrel{\eqref{eq:rho-0-upper}}{\lesssim} \frac{1}{2\pi} \Big(\sup_{|\xi| \le \Omega} |\varphi(\xi)| \cdot |\varrho_0|  +  \sup_{\xi}|\varphi(\xi)| e^{-\Omega}  + \sup_{|\xi| \le 1} |\varphi'(\xi)| \cdot |\varrho_1|  \Big).   
        \label{eq:r-norm-upper}
\end{align}  

\paragraph{Bounds for strictly sub-exponential distribution.}

We assume that $\mu$ is strictly sub-exponential with parameters $(A,K,r)$ and $\Omega>1$. 
Invoking \cref{lem:Rd-derivative-zero} and \cref{lem:complex2-Rd}, we have for $D> r(K\Omega)^r $ that  
\begin{align}
    \sup_{\xi}|\varphi(\xi)| &\le \norm{r_D}_\mu , \\ 
    \sup_{|\xi| \le \Omega} |\varphi(\xi)| &\le A\norm{r_D}_\mu  \cdot \Big(\frac{e r}{D}\Big)^{D/r}\cdot (K\Omega)^D, \\ 
    \sup_{|\xi| \le 1} |\varphi'(\xi)| &\le A\norm{r_D}_\mu\cdot \norm{x^2}_\mu^{1/2}  \cdot \Big(\frac{e r}{D-1}\Big)^{(D-1)/r}\cdot (2K)^{D-1}.  
\end{align}
Substituting these bounds into \cref{eq:r-norm-upper} yields that 
\begin{align}
     \norm{r}_\mu &\lesssim  \Big(\frac{e r}{D}\Big)^{D/r}\cdot (K\Omega)^D +  e^{-\Omega} +\sup_j \norm{\dotp{w_j}{x}^2}_\mu^{1/2}\Big(\frac{2 e r}{D-1}\Big)^{(D-1)/r}\cdot K^{D-1}.  
\end{align} 
Clearly the first term is decreasing in $D$ for $D> r(K\Omega)^r$.  
We choose $D = a_{K,\Omega,r} \cdot r(K\Omega)^r  $, where $a_{K,\Omega,r} = e \Big(1+ K^{-r} \Omega^{-r+1}\Big)$. 
Then we have that 
\begin{align}
\Big(\frac{er}{D}\Big)^{D/r} \cdot (K\Omega)^D &\le  \exp\Big\{D \log (K\Omega) -  \frac{D}{r}  \log(\frac{D}{er})\Big\} \\ 
&=  \exp\Big\{ D \log(K\Omega) - \frac{D}{r} \log \frac{a_{K,\Omega,r} (K\Omega)^r }{e} \Big\}  \\ 
&=  \exp\Big\{  - a_{K,\Omega,r}  (K\Omega)^r \log(a_{K,\Omega,r}/e)  \Big\} \le \exp\{-e\Omega\}. \label{eq:first-term-sube}
\end{align}
Last line holds since $a_{K,\Omega,r}\log(a_{K,\Omega,r}/e) \ge  K^{-r} \Omega^{-r+1}$ using the convexity of $x\log x$.  
For the final bound, we choose $\eps\simeq \log(1/\eps)$ and $D  =  er\big((K \Omega) ^r +  \Omega + K^r  +1 \big) \simeq r (K\log(1/\eps))^r$.  
Then we have that $\norm{r}_\mu \le \eps$.  

\paragraph{Bounds for sub-exponential distribution.}
For the sub-exponential case, we apply \cref{lem:Rd-derivative-zero} and \cref{lem:complex-Rd} and obtain that 
\begin{align}
    \sup_{\xi}|\varphi(\xi)| &\le \norm{r_D}_\mu , \\ 
    \sup_{|\xi| \le \Omega} |\varphi(\xi)|&\le  \norm{r_D}_\mu  \cdot \tanh\Big(\frac{K\pi \Omega}{4}\Big)^D, \\  
    \sup_{|\xi| \le 1} |\partial_\xi \varphi(\xi)| &\le \norm{r_D}_\mu\cdot \norm{x^2}_\mu^{1/2} \tanh\Big(\frac{K\pi }{2}\Big)^{D-1}.  
\end{align}
Plugging these bounds into \cref{eq:r-norm-upper-relu} gives that 
\begin{align}
\norm{r_D}_\mu &\lesssim  \tanh\Big(\frac{K\pi \Omega}{4}\Big)^D +  e^{-\Omega} +\norm{x^2}_\mu^{1/2} \tanh\Big(\frac{K\pi }{2}\Big)^{D-1} \\ 
&\le  \exp\{-4D \cdot e^{-K\Omega/4 }\} + e^{-\Omega} + \norm{x^2}_\mu^{1/2} \exp\{-2(D-1) e^{-K\pi/2}\}.   \label{eq:r-norm-subexp}
\end{align}
Therefore, choosing $\Omega \simeq \log(1/\eps)$ and $D =   \log(3/\eps) \cdot \big(e^{K\pi \Omega/4}  \vee e^{K\pi/2} \big) = O( \log(1/\eps)\cdot \eps^{-O(1)})$.




\end{proof}


\begin{proof}[Proof of \cref{lem:relu-fourier}]\label{proof:lem:relu-fourier} 
First of all, we can decompose $\relu$ as $\relu(x) = (x + |x|)/2$. 
Since 
\begin{align}
    \dotp{\widehat{x}}{ \psi} &= \int_\RR x \cdot \hat\psi(x) \, dx  \\
    &= - i\int \widehat{ \psi'}(x) \, dx=  - 2\pi i \psi'(0),
\end{align}
where the last line follows from the Fourier inversion formula. 
Now it suffices to verify that $-\frac{1}{2}\dotp{\widehat{|x|}}{\psi} = \dotp{\mathrm{f.p.}\, \frac{1}{\xi^2}}{\psi}$.
We can expand the right-hand side as
\begin{align}
    \dotp{\mathrm{f.p.}\,\frac{1}{\xi^2}}{\psi} &= \lim_{\eps \to 0^+} \int_{|\xi|>\eps} \frac{\psi(\xi) - \psi(0)}{\xi^2} \, d\xi \\ 
    &= \lim_{\eps \to 0^+} \int_{|\xi|>\eps}  \frac{\psi' (\xi)}{\xi} \, d\xi \\
    &= \int_{\RR} \frac{\psi'(\xi) - \psi'(0)\chi(\xi )}{\xi} \, d\xi. \label{eq:relu-fourier-1}
\end{align}
Here the second line holds by integration by parts, and $\chi$ is a symmetric bump function that is supported on $[-2,2]$ and equals to $1$ on $[-1,1]$. 
The last line is a consequence of the fact that $\big(\psi'(\xi)- \psi'(0)\chi(\xi) \big)/\xi\in \cS(\RR)$. 
To proceed, we have that 
\begin{align}
    \cref{eq:relu-fourier-1} &= \dotp{1}{\big(\psi' - \psi'(0)\chi\big)/\xi} \\
    &= \dotp{\hat\delta_0}{\big(\psi' - \psi'(0)\chi\big)/\xi} \\
    &= \frac{1}{2}\dotp{\mathrm{sign}'}{\big((\psi' - \psi'(0)\chi )/ \xi\big)^\wedge }\\
    &= \frac{-1}{2}\dotp{\mathrm{sign}}{\big(-i\xi \cdot (\psi' - \psi'(0)\chi)/\xi\big)^\wedge}. \label{eq:relu-fourier-2}
\end{align}
Here the third line follows from the fact that $\delta_0 = \mathrm{sign}'/2$ in the sense of tempered distributions. 
To proceed, we have that $\dotp{\mathrm{sign}}{\chi^\wedge} = 0$ since $\chi^\wedge$ is even. 
Therefore, we obtain that
\begin{align}
    \cref{eq:relu-fourier-2} &= \frac{i}{2}\dotp{\mathrm{sign}}{\widehat{\psi'}} \\
    &= \frac{i}{2}\dotp{\mathrm{sign}}{ix\cdot \hat\psi} \\
    &= -\int_\RR |x| \cdot \hat\psi(x) \, dx = -\dotp{|x|}{\hat\psi}.
\end{align}
This concludes the proof of \cref{lem:relu-fourier}.
\end{proof}

\begin{proof}[Proof of \cref{prop:poly-approx-relu-network}]\label{proof:prop:poly-approx-relu-network}

Similar to the proof of \cref{lem:poly-approx-sigmoid-network}, we only need to provide polynomial approximation bound to the one dimensional ReLU function.
We decompose of the Fourier transform of $\relu$ in \cref{lem:relu-fourier}. 
For any test function $\psi \in \cS(\RR)$, we have that 
\begin{align}
    \dotp{\widehat{\relu}}{\psi} &= \dotp{\mathrm{f.p.}\, \frac{-1}{\xi^2}}{\psi} -  i\pi \psi'(0) \\ 
    &= -\lim_{\eps \to  0^+ }\int_\eps^{1} \frac{ \psi(\xi) +  \psi(- \xi) - 2\psi(0)}{\xi^2} \, d\xi - \int_{|\xi|>1} \frac{\psi(\xi) - \psi(0) }{\xi^2} \, d\xi -  i\pi \psi'(0) .
\end{align}
For the first term, we can use the integration by parts to obtain that
\begin{align}
    -\lim_{\eps\to 0^+}\int_\eps^{1} \frac{ \psi(\xi) +  \psi(- \xi)- 2\psi(0)}{\xi^2} \, d\xi &=  \psi(1) +  \psi(-1) -2\psi(0)- \lim_{\eps\to 0^+}\frac{\psi(\eps) +  \psi(-\eps)- 2\psi(0)}{\eps } \\ &\qquad +\lim_{\eps\to 0^+}\int_\eps^{1} \frac{\psi'(\xi) - \psi'(-\xi)}{\xi} \, d\xi \\   
    &= \psi(1) + \psi(-1) - 2\psi(0) - \int_{[-1,1]} \log |\xi| \cdot \psi'' (\xi) \, d\xi. 
\end{align}
Therefore, we can define  $d\varrho_0  = \delta_1 + \delta_{-1} - 2\delta_0  -  \xi^{-2} \ind\{|\xi|>1\} d\xi$ and $d\varrho_1 = -i\pi \delta_0$ and $d\varrho_2 = - \log |\xi| \cdot \ind\{|\xi|\le 1\} d\xi$.   
It is clear that both $\varrho_l$ for $l=0,1,2$ are complex Radon measures on $\RR$ with finite total variations and $\varrho_1,\varrho_2$ are compactly supported on $[-1,1]$, and $\widehat\relu = \dotp{\varrho_0}{\psi} + \dotp{\varrho_1}{\psi'} + \dotp{\varrho_2}{\psi''}$ in the sense of tempered distributions. 
Additionally, it holds for any $\Omega>1$ that
\begin{align}
    |\varrho_0|(\RR\setminus [-\Omega, \Omega]) &\le  \int_\Omega^\infty \frac{2}{\xi^2} \, d\xi = \frac{4}{\Omega}.  \label{eq:rho-0-upper-relu}
\end{align}

\cref{lem:Rd-derivative-zero} implies that $\partial_{w_j} \varphi(0) = 0$ for each $j\le k$, therefore the second term vanishes since $\rho_1^j$ is supported on the origin. 
Applying \cref{lem:fourier-rep-fs}, we have that for any $\Omega>1$ and $D>0$ that, 
\begin{align}
    \norm{r}_\mu ^2 &=\frac{1}{(2\pi)}\int_{\RR}  \varphi(\xi) d\varrho_0(\xi) + \partial_\xi^2 \varphi(\xi) d\varrho_2(\xi) 
    \\ 
    &\le \frac{1}{2\pi} \Big(\sup_{|\xi| \le \Omega} |\varphi(\xi)| \cdot |\varrho_0|  +   \sup_{\xi}|\varphi(\xi)|\,  |\varrho_0| (\RR\setminus [-\Omega, \Omega]) + \sup_{|\xi| \le 1} |\partial_\xi^2 \varphi(\xi)| \cdot |\varrho_2|  \Big)   \\
    &\stackrel{\eqref{eq:rho-0-upper-relu}}{\lesssim} \frac{1}{2\pi} \Big(\sup_{|\xi| \le \Omega} |\varphi(\xi)|  + \sup_{\xi}|\varphi(\xi)|\,  \Omega^{-1}  + \sup_{|\xi| \le 1} |\partial_\xi^2 \varphi(\xi)| \Big).   
    \label{eq:r-norm-upper-relu}
\end{align}
Now we upper bound $\varphi$ with the concentration property of $\mu$. 

\paragraph{Strictly sub-exponential bound.}  
Suppose that $\mu$ is strictly sub-exponential with parameters $(A,K,r)$. 
We choose $D = er (K\Omega)^r (1 + K^{-r} \Omega^{-r+1})$. 
Then similar to \cref{eq:first-term-sube}, we have that $\norm{r}_\mu \lesssim e^{-e\Omega} +  \Omega^{-1}  \le \eps$ by choosing $\Omega \simeq \eps^{-1}$. 

\paragraph{Sub-exponential bound.} 
Suppose that $\mu$ is sub-exponential. 
Then similar to \cref{eq:r-norm-subexp}, we have that 
\begin{align}
    \norm{r}_\mu &\lesssim  \exp\{-4D \cdot e^{-K\pi \Omega/4 }\} + \Omega^{-1} + \exp\{-2(D-1) e^{-K\pi/2}\}.   
\end{align}
 We choose $\Omega \simeq \eps^{-1}$ and $D = \exp\{K\pi \Omega/2\} \cdot \log(3/\eps)\simeq \exp\{O(\eps^{-1})\}$.  
 Then we have that $\norm{r}_\mu \le \eps$ for sufficiently small $\eps$. 

\begin{proof}[Proof of \cref{thm:lip_samples}]\label{proof:thm:lip_samples}
Applying \cref{thm:poly-approx-lip-network} and \cref{lem:gen-error-polynomial-regression} directly yields the result. 
\end{proof}





\end{proof}

\section{Application to testable learning}\label{apdx:testable}
In this appendix, we sketch how if we have a guarantee for learning a class over sub-Gaussian data (e.g., via our main results), then using the sub-Gaussian certification result of \cite{diakonikolas2025sos}, combining SoS and polynomial regression yields testable learners in the sense of \cite{RubinfeldVasilyan23}.
\subsection{Preliminaries}
Let $X$ be an instance space and let $\mathcal C \subseteq \{\,h: X \to \{\pm1\}\,\}$ be a (binary) concept class.
A distribution $D$ over $X \times \{\pm1\}$ has $X$-marginal $D_X$.
For a hypothesis $h:X\to\{\pm1\}$, the (0–1) risk under $D$ is
\[
L_D(h)\;:=\;\Pr_{(x,y)\sim D}[\,h(x)\neq y\,],
\qquad
\mathrm{opt}(\mathcal C,D)\;:=\;\inf_{f\in C} L_D(f).
\]


\paragraph{Testable agnostic learning \cite{RubinfeldVasilyan23}.}
A (randomized) \emph{tester–learner pair} $(T,A)$ \emph{testably (agnostically) learns $C$ with respect to a target marginal $D_X$ up to excess error $\varepsilon$ and confidence $1-\delta$} if for every distribution $D$ on $X\times\{\pm1\}$:
\begin{itemize}
  \item \textbf{(Soundness)}:
  If the tester $T$ \emph{accepts} with probability at least $1-\delta$ on a sample from $D$, then with probability at least $1-\delta$ the learner $A$ outputs a hypothesis $h$ satisfying
  \[
  L_D(h)\;\le\; \mathrm{opt}(\mathcal C,D) + \varepsilon.
  \]
  \item \textbf{(Completeness)}:
  Whenever $D$ has $X$-marginal exactly $D_X$, the tester $T$ \emph{accepts} with probability at least $1-\delta$.
\end{itemize}
The “with high probability’’ constant (e.g.\ $2/3$ or $0.99$) can be amplified by repetition; sample and time bounds are measured in the usual way as functions of $(\varepsilon,\delta)$ and relevant parameters.

\subsection{Certification}
We will use the following result for SoS-certifying sub-Gaussian moments of an empirical distribution
formed from i.i.d. samples from a sub-Gaussian distribution. 
\begin{theorem}[Theorem~4.6 of \cite{diakonikolas2025sos}]\label{thm:sos-certify}
Let $P$ be an $s$-sub-Gaussian distribution on $\mathbb{R}^d$ with mean $\mu$, and let $m\in 2\mathbb{N}$.
Assume $n \ge \mathrm{poly}(d^m)$.
Let $\mathcal{T}'$ denote the set of all degree-$m$ SoS pseudoexpectations $\widetilde{\mathbb{E}}$ over
$v=(v_i)_{i=1}^d$ such that $\widetilde{\mathbb{E}}[\|v\|_2^{m}]>0$.
If $X_1,\dots,X_n \stackrel{\mathrm{i.i.d.}}{\sim} P$, then
\[
\mathbb{E}_{X_{1:n}}\!\left[
  \max_{\widetilde{\mathbb{E}}\in\mathcal{T}'}\;
  \frac{\widetilde{\mathbb{E}}\!\left[\frac{1}{n}\sum_{i=1}^n \langle v, X_i\rangle^{m}\right]}
       {\widetilde{\mathbb{E}}[\|v\|_2^{m}]}
\right]
\;\le\; (C\, s \sqrt{m})^{\,m},
\]
for a universal constant $C>0$.
\end{theorem}
\begin{remark}
If we able to SoS certify sub-Gaussian moments of the empirical distribution, it does not necessarily
imply that the population distribution has sub-Gaussian moments (there could be extremely rare outliers which
violate the moment conditions on average but did not appear in the data). So, e.g., we cannot guarantee that the algorithm
in the next section always rejects if the population distribution actually fails to be sub-Gaussian. Instead, the idea is that as long as the certification step succeeds, the learning algorithm should work (with high probability), even if the true distribution was not sub-Gaussian.  
\end{remark}
\subsection{Building a testable learner for sub-Gaussian data}
Suppose we are guaranteed that over any $K$-sub-Gaussian distribution $\mu$, polynomials of degree $m = m(K,\epsilon)$
are able to $\epsilon$-approximate concept class $\mathcal C$ in $L_1$ loss. Then
the following meta-algorithm, parameterized by $\delta, \epsilon, m > 0$, gives a testable learner:
\begin{enumerate}
    \item Form the empirical distribution from $n = poly(d^m,1/\epsilon)$ samples $(X_1,Y_1),\ldots,(X_n,Y_n)$ from $D$.
    \item Using semidefinite programming, certify that all degree $m$ SoS pseudo-expectations $\widetilde{\mathbb E}$ satisfy
    \[ \widetilde{\mathbb{E}}\!\left[\frac{1}{n}\sum_{i=1}^n \langle v, X_i\rangle^{m} -  (1/2\delta) ( C\, s \sqrt{m})^{\,m}\|v\|_2^m\right] \le 0. \]
    If SoS fails to certify (i.e., there exists a certificate not satisfying the inequality), then \emph{reject}.
    Otherwise, we know that the empirical distribution has sub-Gaussian moments up to degree $m$, because expectations corresponding to point masses at a particular vector $v$ are a subset of the space of pseudoexpectations.
    \item Run $L_1$ regression over degree $m$ polynomials to find the polynomial $p$ minimizing the empirical $L_1$ risk $\sum_{i = 1}^n |Y_i - p(X_i)|$.
    \item Output the predictor $x \mapsto sgn(p(x))$.
\end{enumerate}

\emph{General completeness argument: } The key to completeness is Theorem~\ref{thm:sos-certify} and Markov's inequality, which ensures that the certification succeeds with high probability. If certification succeeds, then the empirical distribution has suitable sub-Gaussian moments. Therefore, using the assumption and setting the parameter $m$ appropriately, there exists a polynomial with empirical risk competitive with the empirical risk of the best-performing member of the concept class. By a standard argument (see, e.g., \cite{klivans2013moment,chandrasekaran2024smoothed}), this guarantees the empirical 0/1 risk of the corresponding PTF is also competitive, and by standard VC theory results its test error will be close provided $n$ was a large enough polynomial. 

\emph{General soundness argument: } As explained above, if the empirical distribution does not have sub-Gaussian moments up to degree $m$, then SoS will fail to certify and so the algorithm will reject.

\section{Paley--Wiener universality in higher dimensions}\label{sec:pw-rd}
Throughout fix an integer $d\ge 1$.  We write $\R^{d}$ for Euclidean
space with Lebesgue measure $dx$, denote by $|x|=|x|_{2}$ the
standard norm, and use multi–index notation

$$
  \alpha=(\alpha_{1},\dots,\alpha_{d})\in\N^{d},
  \qquad
  |\alpha|:=\alpha_{1}+\dots+\alpha_{d},
  \qquad
  \alpha!:={\alpha_{1}!}\cdots{\alpha_{d}!},
  \qquad
  x^{\alpha}:=x_{1}^{\alpha_{1}}\cdots x_{d}^{\alpha_{d}}.
$$

\subsection{Paley–Wiener theorem in \texorpdfstring{$\R^{d}$}{Rd}}

Let $\widehat f(\xi)=\int_{\R^{d}}f(x) e^{-i\xi\cdot x} dx$ be the
 Fourier transform on $\R^{d}$.  For $\Omega>0$ define the
band–limited (Paley–Wiener) subspace
$$
  PW_{\Omega}
  :=
  \bigl\{f\in L^{2}(\R^{d}) : \supp \widehat f\subset B_{\Omega}\bigr\},
$$
where $B_{\Omega}:={\xi\in\R^{d}:|\xi|\le \Omega}$ is the closed Euclidean
ball of radius $\Omega$ centered at the origin.

The following result is a special case of the more general Paley-Wiener-Schwartz theorem (Theorem~\ref{thm:paley-wiener-schwartz}, see e.g. \cite{hörmander1983analysis}).
\begin{theorem}[d–Dimensional Paley–Wiener]\label{thm:PWd}
For $f\in PW_{\Omega}$ the Fourier inversion formula extends $f$ to an
entire function
$$
  F(z):=\frac{1}{(2\pi)^{d}}\int_{B_{\Omega}}
          \widehat f(\xi)\,e^{i\,\xi\cdot z}\,d\xi,
  \qquad z\in\C^{d},
$$
which satisfies the exponential–type bound
\begin{equation}\label{eq:PWd-growth}
|F(z)| \le C e^{\Omega |\Im z|},
\qquad z\in\C^{d},
\end{equation}
for some constant $C$ depending only on $|f|_{L^{2}}$.  Conversely,
any entire $F$ obeying \eqref{eq:PWd-growth}
arises in this way from a
unique $f\in PW_{\Omega}$.
\end{theorem}

\paragraph{Taylor–coefficient asymptotics.}
Write $F(z)=\sum_{\alpha\in\N^{d}}a_{\alpha} z^{\alpha}$.  Applying
Cauchy’s estimate on the polydisc $|z_{j}|=r$ (all $j$) to
\eqref{eq:PWd-growth} gives
$$
  |a_{\alpha}|\;\le\;C\,r^{-|\alpha|}\,e^{\Omega r}
  \;\;(\forall\,r>0),
$$
which is minimised at $r=|\alpha|/\Omega$.  With Stirling’s formula one
obtains the multivariate analogue of \eqref{eq:PW-coeff-asymp}:
$$
  |a_{\alpha}|\;=\;O\!\Bigl(\frac{\Omega^{\,|\alpha|}}{\alpha!}\Bigr),
  \qquad
  \limsup_{|\alpha|\to\infty}\bigl(|a_{\alpha}|\,\alpha!\bigr)^{1/|\alpha|}
  \le \Omega.
$$
Thus the radius $\Omega$ controls the factorial–weighted decay of the
multivariate Taylor coefficients.

\subsection{Hermite polynomials and functions in \texorpdfstring{$\R^{d}$}{Rd}}

\paragraph{(1)  Tensor–product Hermite polynomials.}
Let $H_{n}$ (resp. $\He_{n}$) be the one–dimensional physicists’ (resp.
probabilists’) Hermite polynomials.  For a multi–index $\alpha$ set

$$
  H_{\alpha}(x):=\prod_{j=1}^{d}H_{\alpha_{j}}(x_{j}),
  \qquad
  \He_{\alpha}(x):=\prod_{j=1}^{d}\He_{\alpha_{j}}(x_{j}).
$$

The relations \eqref{eq:H_vs_He} hold componentwise so that
$H_{\alpha}(x)=2^{|\alpha|/2},\He\_{\alpha}(x/\sqrt{2})$.

\paragraph{(2)  d–Dimensional Hermite functions.}
Define
\begin{equation}\label{eq:hermite_d}
\varphi_{\alpha}(x)
:=
\frac{1}{\pi^{d/4}\sqrt{2^{|\alpha|}\alpha!}}
H_{\alpha}(x)e^{-\tfrac12|x|^{2}},
\qquad x\in\R^{d},\alpha\in\N^{d}.
\end{equation}
The collection $(\varphi_{\alpha})_{\alpha\in\N^{d}}$ forms an
orthonormal basis of $L^{2}(\R^{d})$.

\subsection{Bargmann transform on \texorpdfstring{$\R^{d}$}{Rd}}

\paragraph{Definition.}
The
d–dimensional Bargmann (or Segal–Bargmann) transform

$$
  \mathcal{B}:L^{2}(\R^{d})\longrightarrow\cF^{2}(\C^{d})
  \;:=\;\Bigl\{F\text{ entire} :
     \|F\|_{\cF^{2}}^{2}
     :=\pi^{-d}\!\int_{\C^{d}}|F(z)|^{2}\,e^{-|z|^{2}}\,dz<\infty\Bigr\}
$$
reads
\begin{equation}\label{eq:Bargmann_d-def}
(\mathcal{B}f)(z)
:=\pi^{-d/4}\int_{\R^{d}}
\exp\bigl(-\tfrac12|x|^{2}+\sqrt{2}\langle x, z \rangle -\tfrac12|z|^{2}\bigr)
f(x),dx,
\qquad z\in\C^{d}.
\end{equation}
One checks that $\mathcal{B}$ is unitary onto $\cF^{2}(\C^{d})$.

\paragraph{Action on the Hermite basis.}
Because the kernel in \eqref{eq:Bargmann_d-def} factorises over the $d$
coordinates, the one–dimensional identity
$\mathcal{B}\varphi_{n}=z^{n}/\sqrt{n!}$ extends to
\begin{equation}\label{eq:Bargmann_d-Hermite}
(\mathcal{B}\varphi_{\alpha})(z)
=\frac{z^{\alpha}}{\sqrt{\alpha!}},
\qquad\alpha\in\N^{d}.
\end{equation}
Thus $\mathcal{B}$ sends the Hermite basis of $L^{2}(\R^{d})$ onto the
monomial basis of the $d$–dimensional Fock space.

\subsection{Paley–Wiener and Hermite expansion in \texorpdfstring{$\R^{d}$}{Rd}}

Let
$$
  h_{\alpha}(x):=\frac{\He_{\alpha}(x)}{\sqrt{\alpha!}},
  \qquad \alpha\in\N^{d},
$$
be the probabilists’ Hermite polynomials normalised to form an
orthonormal basis of $L^{2}(N(0,_{d}))$ (the centred Gaussian measure
with covariance $I_{d}$).  Any $f\in L^{2}(N(0,I_{d}))$ admits the
expansion $f(x)=\sum_{\alpha}a_{\alpha}h_{\alpha}(x)$.

\begin{theorem}
\label{thm:paley-weiner-bargmann-rd}
Fix $\Omega>0$, $f \in L_2(\mathbb R^d)$, and write $f(x)=\sum_{\alpha}a_{\alpha}h_{\alpha}(x)$.
Then the following are equivalent:
\begin{enumerate}
\item $\supp\widehat f\subset B_{\Omega}$.
\item There exists $M>0$ such that
$|a_{\alpha}|\le M\dfrac{\Omega^{|\alpha|}}{\sqrt{\alpha!}}$
for every $\alpha\in\N^{d}$.
\end{enumerate}
\end{theorem}

The proof is identical to Theorem~\ref{thm:HPW}, replacing integers by
multi–indices and using \eqref{eq:Bargmann_d-Hermite}.  The key step is
the identity
\begin{equation}\label{eq:Bargmann_Fourier_d}
(\mathcal{B}r)(z)
= \int_{\R^{d}}
\widehat f(\xi) e^{-\tfrac12|\xi|^{2}} e^{i\xi\cdot z} d\xi,
\end{equation}
which follows from the same change–of–variables argument used in the
one–dimensional lemma.

\subsection{Application to universality}
\begin{theorem}
Suppose that $f \in L_2(\mathbb R^d)$, and its Hermite expansion in $L_2(N(0,1))$ is
\[ f = \sum_{\alpha} a_{\alpha} h_{\alpha}. \]
For any $\Omega > 0$ the following are equivalent:
\begin{enumerate}
    \item $f \in PW_{\Omega}$.
    \item There exist $M > 0$ and $\Omega > 0$ such that $|a_{\alpha}|\le M\dfrac{\Omega^{|\alpha|}}{\sqrt{\alpha!}}$
for every $\alpha\in\N^{d}$.
    \item For any $r$-strictly sub-exponential  distribution $\mu$ with scale parameter $K$, there exists $C > 0$ and a sequence of polynomials $(p_m)_{m = 1}^{\infty}$ such that
    \[ \|f - p_m\|_{L_2(\mu)} \le C \frac{(\Omega K (er)^{1/r})^m}{m^{m/r}}. \]
    Moreover, if $f \in PW_{\Omega}$ then we can take any $H = er^{1/r}K\Omega$.
\end{enumerate}
\begin{proof}
    The equivalence of the first two is Theorem~\ref{thm:paley-weiner-bargmann-rd}. Given the third condition, the first condition follows as in the proof of Theorem~\ref{thm:pw-universality}: we can prove the second condition for any $\Omega' > \Omega$ by applying the case $r = 2$ and $\mu = N(0,I)$ and using Parseval's theorem, and then by using the Paley-Wiener theorem we conclude $f \in PW_{\Omega}$.
    
    Assuming that $f \in PW_{\Omega}$, let $p_m$ be the $L_2(\mu)$ orthogonal projection of $f$ onto the space of degree $m$ polynomials and observe by Lemma~\ref{lem:fourier-error-rep} that
    \[ \|f - p_m\|_{L_2}^2 =  \frac{1}{(2\pi)^n} \int \hat f(\xi) \varphi(\xi)\, d\xi,  \]
    and by Theorem~\ref{thm:Rd-pw-strict} that provided $m > r (K\|\xi\|_2^2)$,
    \[ |\varphi(\xi)| \le  A \norm{f - p_m}_\mu \, \Big(\frac{er}{m}\Big)^{m/r} (K \|\xi\|_2)^{m}. \]
    So dividing by $\|f - p_m\|_{L_2}$, using that $f \in PW_{\Omega}$, and applying H\"older's inequality,
    \[ \|f - p_m\|_{L_2} \le \frac{A}{(2\pi)^{d}} \|\hat f\|_{L_1} \Big(\frac{er}{m}\Big)^{m/r} K^m \Omega^m \]
    which proves the result. 
\end{proof}
\end{theorem}

\end{document}